\documentclass[twoside,12pt,reqno]{amsart}
\usepackage{bbm, tikz}
\usetikzlibrary{positioning, arrows}
\usepackage{amssymb}
\usepackage{fullpage}
\usepackage[margin=1.5cm]{geometry}
\usepackage{bm}

\usepackage{graphicx}
\input{xy}
\xyoption{all}
\xyoption{curve}
\usepackage{color,soul}

\makeatletter

\newtheorem{Theorem}{Theorem}[section]
\newtheorem*{Theorem*}{Theorem}
\newtheorem{Lemma}[Theorem]{Lemma}
\newtheorem{Proposition}[Theorem]{Proposition}
\newtheorem{Corollary}[Theorem]{Corollary}

\theoremstyle{remark}
\newtheorem{Remark}[Theorem]{Remark}
\newtheorem*{proofofskryabin}{Proof of Theorem~\ref{T:equivariantSkryabin}}

\numberwithin{equation}{section}

\usepackage[pdftex,colorlinks,backref,pagebackref,hypertexnames=false, linkcolor=blue,urlcolor=blue,citecolor=blue]{hyperref}

\newcommand{\arxiv}[1]{{\tt arXiv:#1}}

\def\into{\hookrightarrow}
\def\onto{\twoheadrightarrow}
\def\isoto{\overset{\sim}{\longrightarrow}}
\DeclareRobustCommand\longtwoheadrightarrow
     {\relbar\joinrel\twoheadrightarrow}

\def\epsilon{\varepsilon}

\def\k{{\mathbbm k}}

\def\bc{\text{\boldmath$c$}}

\def\be{\text{\boldmath$e$}}

\def\bv{\text{\boldmath$v$}}

\def\diag{\operatorname{diag}}
\def\lmod{\!\operatorname{-mod}}
\def\Lie{\operatorname{Lie}}

\def\ab{{\operatorname{ab}}}

\def\reg{{\operatorname{reg}}}
\def\ad{\operatorname{ad}}
\def\Ad{\operatorname{Ad}}
\def\Ann{\operatorname{Ann}}
\def\Ker{\operatorname{Ker}}

\def\gr{\operatorname{gr}}
\def\Mat{\operatorname{Mat}}

\def\sspan{\operatorname{span}}
\def\Spec{\operatorname{Spec}}

\def\Rad{\operatorname{Rad}}
\def\Rep{\operatorname{Rep}}
\def\Id{\operatorname{Id}}
\def\Ind{\operatorname{Ind}}
\def\InD{\operatorname{Ind}_G}
\def\Decomp{{\mathrm{Dec}}}

\def\Indat{{\mathrm{LD}}}
\def\RIndat{{\mathrm{RLD}}}

\def\End{\operatorname{End}}

\def\Aut{\operatorname{Aut}}
\def\neu{{\operatorname{ne}}}
\def\op{{\operatorname{op}}}

\def\im{\operatorname{im}}
\def\rank{\operatorname{rank}}
\def\Max{\operatorname{Max}}

\def\bA{{\bm{A}}}
\def\bC{{\bm{C}}}
\def\bD{{\bm{D}}}
\def\bG{{\bm{G}}}
\def\bS{{\bm{S}}}
\def\bT{{\bm{T}}}
\def\bX{{\bm{X}}}
\def\bU{{\bm{U}}}
\def\bB{{\bm{\mathcal{B}}}}
\def\bc{{\bm{\mathfrak{c}}}}
\def\bsigma{{\bm{\sigma}}}
\def\bg{{\bm{\mathfrak{g}}}}
\def\be{{\bm{e}}}
\def\bh{{\bm{h}}}
\def\bn{{\bm{n}}}
\def\bu{{\bm{u}}}
\def\bx{{\bm{x}}}
\def\bPi{{\bm{\Pi}}}

\def\bgamma{{\bm{\gamma}}}
\def\bkappa{{\bm{\kappa}}}
\def\brho{{\bm{\rho}}}
\def\bpsi{{\bm{\psi}}}
\def\bchi{{\bm{\chi}}}
\def\bomega{{\bm{\omega}}}

\def\C{{\mathbb C}}
\def\F{{\mathbb F}}
\def\Z{{\mathbb Z}}
\def\Q{{\mathbb Q}}
\def\N{{\mathbb N}}
\def\G{{\mathbb G}}

\def\GL{\mathrm{GL}}
\def\SL{\mathrm{SL}}

\def\SO{\mathrm{SO}}
\def\Sp{\mathrm{Sp}}

\def\c{\mathfrak c}
\def\g{{\mathfrak g}}
\def\gl{\mathfrak{gl}}
\def\so{\mathfrak{so}}
\def\sp{\mathfrak{sp}}
\def\h{\mathfrak h}
\def\l{\mathfrak{l}}
\def\m{\mathfrak m}
\def\n{\mathfrak n}
\def\p{\mathfrak p}

\def\sl{\mathfrak{sl}}

\def\t{\mathfrak t}
\def\u{\mathfrak{u}}
\def\v{\mathfrak{v}}

\def\r{\mathfrak{r}}
\def\z{\mathfrak{z}}
\def\kk{\mathfrak{k}}

\def\E{\mathcal{E}}
\def\B{\mathcal{B}}
\def\D{\mathcal{D}}
\renewcommand{\O}{\mathcal{O}}

\renewcommand{\S}{\mathcal{S}}
\newcommand{\Nc}{\mathcal{N}}

\def\Fb{{\sf F}}

\def\tQ{\widetilde{Q}}

\def\tg{\tilde{\g}}
\def\tn{\tilde{\n}}
\def\hU{\widehat{U}}
\def\hZ{\widehat{Z}}
\def\hE{\widehat{\E}}
\def\cv{\check\v}

\def\PeN{\mathcal{P}_\epsilon(N)}

\title{\boldmath Parabolic induction for modular finite $W$-algebras}
\author{Simon M.~Goodwin, Lewis Topley and Matthew Westaway}
\address{School of Mathematics,
University of Birmingham,
Birmingham, B15 2TT,
UK}
\email{s.m.goodwin@bham.ac.uk}
\address{Department of Mathematical Sciences, University of Bath, Claverton Down, Bath, BA2 7AY}
\email{lt803@bath.ac.uk}
\email{mw2915@bath.ac.uk}


\begin{document}

\begin{abstract}
We study the modules of minimal dimension for reduced enveloping algebras of Lie algebras of reductive algebraic groups using the theory of modular finite $W$-algebras. First of all we consider the case where the $p$-character lies in a unique sheet, and demonstrate that in classical cases and in most exceptional cases all minimal modules are parabolically induced from a Levi subalgebra and a rigid $p$-character. Secondly we consider the minimal modules which are invariant under twisting by the component group, showing that in classical cases and in most exceptional cases these are also parabolically induced from a Levi subalgebra and a rigid $p$-character.
\end{abstract}

\maketitle

\section{Introduction}

Fix an algebraically closed field $\k$ of characteristic $p > 0$. Let $G$ be a connected, simply-connected simple group over $\k$ with Lie algebra $\g = \Lie G$. Suppose that $p$ is good for the root system of $G$ and that $\g$ admits a non-degenerate $G$-invariant bilinear form. These assumptions are often referred to as the standard hypotheses \cite[\textsection 6.3]{JaLA}.

\subsection{Sheets and minimal modules}\label{ss: IntroSheets}

The enveloping algebra $U(\g)$ is finite over the $p$-centre $Z_p(\g)$ and the prime spectrum of $Z_p(\g)$ naturally identifies with $(\g^*)^{(1)}$, the Frobenius twist of the coadjoint $G$-module. Using a version of Schur's lemma we see that every irreducible representation of $U(\g)$ factors through one of the central reductions over $(\g^*)^{(1)}$. The associated element $\chi \in \g^*$ is called the $p$-character of the representation, and the central reductions are known as reduced enveloping algebras and denoted $U_\chi(\g)$.

It is well-known that the dimensions of irreducible representations are controlled by the $p$-character: Premet's theorem asserts that whenever $G$ satisfies the standard hypotheses of \cite[\textsection 6.3]{JaLA} every $U_\chi(\g)$-module has dimension divisible by $p^{d_\chi}$ where $d_\chi := \frac{1}{2} \dim G\cdot \chi$, see \cite{PrKW}. Meanwhile, work of Premet and the second author \cite{PT2} shows that modules of dimension $p^{d_\chi}$ exist for every $p$-character, confirming a long-standing conjecture of Humphreys. We refer to $U_\chi(\g)$-modules of dimension $p^{d_\chi}$ as {\it minimal modules}. By Premet's theorem, minimal modules are necessarily simple and, as we shall explain, they play a distinguished role amongst all simple $U_\chi(\g)$-modules.

Premet showed that $U_\chi(\g)$ is Morita equivalent to a restricted finite $W$-algebra \cite{PrST}, and the minimal $U_\chi(\g)$-modules are equivalent to one dimensional representations for the $W$-algebra. This relationship allows many tools to be deployed in their study which are unavailable for arbitrary simple modules. One dimensional representations of complex finite $W$-algebras have been studied for many years, and it is an important general objective to extend these methods and results to positive characteristics. For example, the long-running project to show the existence of such one dimensional modules over $\C$ (see \cite{LoQSA, PrMF} and the references therein) was extended to positive characteristics in \cite{PT2}. The relationship between one dimensional modules and quantizations of nilpotent orbits (and their affinisations) over $\C$ is now understood, by the work of Losev and collaborators \cite{LoQNO, LoOM, LMM}. This has recently been extended to positive characteristics, for $G = \GL_N$, by Ambrosio and the second and third author \cite{ATW}.

The current paper makes a significant advancement in this program of research by firstly constructing a parabolic induction functor for finite $W$-algebras, analogous to Losev's construction over $\C$ in \cite{Lo}, and then utilising this functor to extend most of the main results of \cite{PT1} to the modular setting. In particular, we study minimal $\g$-modules for general $G$, satisfying one of two hypotheses: (1) that the $p$-character lies on a unique sheet or (2) that the module is $G^\chi$-stable. We expect this work will play a decisive role in the classification of minimal modules, which is a central open problem in the field. When $G = \GL_N$, this classification has been carried out by the first and second authors in \cite{GTmin}.

For $m \in \N$ we let $\g_{[m]}$ be the union of orbits of dimension $m$. A sheet of $\g^*$ is defined as an irreducible component of $\g_{[m]}$ for some $m$. Sheets are classified using the theory of induction of coadjoint orbits and decomposition classes \cite{LS, PS}, and the reader may refer to our Section~\ref{ss:decompclasses} for a brief survey. We recall that the decomposition classes of $\g^*$ are the equivalence classes under the relation
$$\chi \sim \psi \Leftrightarrow g \cdot \chi_{\rm n} = \psi_{\rm n} \text{ and } gG^{\chi_{\rm s}} g^{-1} = G^{\psi_{\rm s}} \text{ for some } g\in G,$$ where $\chi=\chi_{\rm s} + \chi_{\rm n}$ is the Jordan decomposition and $G^\chi$ denotes the coadjoint stabliser of $\chi$.

Sheets and decomposition classes are locally closed subvarieties of $\g^*$ and every sheet is stratified by classes. We write $\D_\chi$ for the class of $\chi \in \g^*$, and we say that $\D_\chi$ is rigid if it is dense in a sheet. Similarly we say that $\chi \in \g^*$ is {\it rigid} if $\D_\chi$ is rigid. If $\chi$ is nilpotent then this corresponds to the $G$-orbit of $\chi$ being rigid i.e. it cannot be induced from a proper Levi subalgebra under Lusztig-Spaltenstein induction.

We are interested in whether minimal modules can be parabolically induced. An affirmative answer in type ${\sf A}$ is given in \cite{GTmin} (there, every element lies in a unique sheet); our main result is the following generalisation of that result to other types. 

\begin{Theorem}
	\label{T:main}
    Let $G$ be a simple algebraic group, and pick $\chi \in \g^*$ nilpotent. If the $G$-orbit of $\chi$ is not listed in Tables~\ref{Tab: Bad Induced Orbits} or \ref{Tab: Bad Good Orbits} then the following hold.

	\begin{enumerate}
		\setlength{\itemsep}{4pt}
		\item If $\chi$ lies in a unique sheet then every minimal $U_\chi(\g)$-module is parabolically induced from a minimal module for a reduced enveloping algebra with rigid $p$-character.
		\item Every $G^\chi$-stable minimal $U_\chi(\g)$-module is parabolically induced from a minimal module for a reduced enveloping algebra with rigid $p$-character.
	\end{enumerate}
    \begin{table}
	\begin{center}
		\caption{Orbits excluded in \cite{PT1}}\label{Tab: Bad Induced Orbits}
		\bgroup
		\renewcommand{\arraystretch}{2}
		\begin{tabular}{|c|c|c|c|c|c|c|}
			\hline 
			$F_4$ & $E_6$ & $E_7$ & $E_8$ & $E_8$ & $E_8$ & $E_8$ \\
			\hline
			$C_3(a_1)$ & $A_3+A_1$ & $D_6(a_1)$ & $E_6(a_3)+A_1$ & $D_6(a_2)$ & $E_7(a_2)$ & $E_7(a_5)$ \\
			\hline
		\end{tabular}
		\egroup
	\end{center}
\end{table}

\begin{table}
	\begin{center}
		\caption{Other orbits excluded from the present work}\label{Tab: Bad Good Orbits}
		\bgroup
		\renewcommand{\arraystretch}{2}
		\begin{tabular}{|c|c|c|c|}
			\hline 
			$E_7$ & $E_8$ & $E_8$ & $E_8$  \\
			\hline
			$A_3+A_2$ & $A_3+A_2$ & $E_7(a_4)$ & $D_7(a_2)$\\
			\hline
		\end{tabular}
		\egroup
	\end{center}
\end{table}
\end{Theorem}

Since the type ${\sf A}$ case is understood by \cite{GTmin} we largely avoid discussing it further, save for Remark~\ref{R:glNtheorem} which briefly details how the methods of the current paper could be applied to $G = \GL_N$.

Note that our Theorem~\ref{T:main} holds for arbitrary $p$-characters and reductive groups mutatis mutandis, under the standard hypotheses of \cite[\textsection 6.3]{JaLA}; we call such $G$ a {\em standard reductive group}. In this case $\Lie(G)$ is a product of Lie algebras to which our theorem applies, see Remark~\ref{R:standardreductivedecomp}(2). Since this generalisation of our result involves a few non-trivial steps, we write out the details carefully in Theorem~\ref{P:nilpotentreduction}.

We now provide some commentary on the restrictions that appear in Theorem~\ref{T:main}. As we detail further below, the main tool used to prove these theorems is the theory of one dimensional modules for finite $W$-algebras. While over fields of characteristic $p>0$ this theory leads to results on reduced enveloping algebras, over $\C$ it yields results regarding completely prime primitive ideals of universal enveloping algebras. From this perspective, the current paper may be viewed as a sibling over fields of positive characteristic to \cite{PT1} (which also considered elements lying in a single sheet and stability under a component group action). The nilpotent orbits from Table~\ref{Tab: Bad Induced Orbits} are also excluded from consideration in \cite{PT1}; indeed, in \cite{BGn} it is shown that the results of \cite{PT1} fail for many (and conjecturally all) of these orbits. The failure of our results for these orbits also seems likely. 

Excluding the orbits in Table~\ref{Tab: Bad Induced Orbits}, we are able to tackle two types of nilpotent orbits in exceptional cases: those that lie in a single sheet and those that are even. The nilpotent orbits in Table~\ref{Tab: Bad Good Orbits} are those induced orbits not in Table~\ref{Tab: Bad Induced Orbits} which do not meet either of these requirements. Note that the results in \cite{PT1} do apply to these orbits, but the argument in that paper requires application of results of \cite{Fu} and cannot easily be generalised to positive characteristic.

We conclude this subsection by noting that in light of the above comparison with \cite{PT1} and recent work of the present authors \cite{GTW}, it seems quite likely for $G$ of type {\sf B}, {\sf C} or {\sf D} that Theorem~\ref{T:main}(1) holds true for all $\chi$, not just those lying in a unique sheet.

\subsection{Sketch of methodology}
\label{ss:sketch}


Let $e\in \g$ nilpotent. To $(\g,e)$ we can attach two algebras: the {\it finite $W$-algebra} $U(\g,e)$ and the {\it extended finite $W$-algebra} $\hU(\g,e)$. Under Jantzen's standard hypotheses \cite[\textsection 6.3]{JaLA}, both these algebras were defined and studied thoroughly in \cite{GTmod}, although they had both appeared in earlier work of Premet in large characteristics \cite{PrCQ}. We recall the construction in Section~\ref{ss:modularWalgebras}, generalising the approach of \cite{GTmod} by showing that large families of graded nilpotent Lie subalgebras of $\g$ are algebraic (cf. Lemma~\ref{L:appendix}).

Both $U(\g,e)$ and $\hU(\g,e)$ admit actions by the reductive part $C^e$ of the centraliser $G^e$, and their categories of modules are equivalent to certain full subcategories of $U(\g)$-modules, see \cite[\textsection 9.2]{GTmod}. We refer to these equivalences collectively as {\it Skryabin's equivalence}, after \cite[Appendix]{PrST}. Our first milestone shows that the corresponding functor is $C^e$-equivariant on simple modules, see Section~\ref{ss:Skryabinsequivalence}. This result is fairly natural; surprisingly, the proof is slightly non-trivial.

The main technical tool of this paper is the construction of a {\it parabolic induction functor} for modular finite $W$-algebras. Our results on this topic are inspired by Losev \cite{Lo}, however, our methods are very different.

Suppose that the orbit of $e \in \Nc(\g)$ is Lusztig--Spaltenstein induced from a nilpotent orbit $\O_0$ in a Levi subalgebra $\g_0$. Let $\p$ be a parabolic subalgebra of $\g$ with Levi decomposition $\p=\g_0\oplus\n$ such that there exists $e_0 \in \O_0$ with $e-e_0\in\n$. We say that a $\hU(\g_0, e_0)$-module is {\it semi-restricted} if the $p$-support is contained in $\chi_0 + \kappa_0 \z(\g_0)$, where $\kappa_0 := \kappa|_{\g_0} : \g_0 \to \g_0^*$ is a $G_0$-equivariant isomorphism. 
We construct a functor $\Ind_{(\g_0, e_0)}^{(\g,e)}$ from semi-restricted $\hU(\g_0, e_0)$-modules to $U(\g,e)$-modules, and show (Theorem~\ref{T:parabolicinductiontheorem})
that it enjoys the following properties.
\begin{enumerate}
	\setlength{\itemsep}{4pt}
	\item It is dimension preserving.
	\item The effect on $p$-characters can be described explicitly.
	\item Restricting to one dimensional representations gives a finite morphism of representation schemes.
	\item Any subgroup $\Gamma \subseteq C_0^{e_0}\cap C^e$, where $C_0^{e_0}$ is the reductive part of $G_0^{e_0}$, acts on the variety of semi-restricted one dimensional representations of $\hU(\g_0, e_0)$ in such a way that the finite morphism from (3) is $\Gamma$-equivariant.
\end{enumerate}
Note that the functor does depend on our choice of parabolic subalgebra $\p$, although we usually omit this from the notation.

To proceed with the description of the proof, we need a little more notation. If $X \subseteq (\g^*)^{(1)}$ is Zariski closed then there is a corresponding ideal $I_X \subseteq Z_p(\g) = \k[(\g^*)^{(1)}]$. Generalising the notation for reduced enveloping algebras, we write $U_X(\g) = U(\g) / I_XU(\g)$. In Section~\ref{ss:pcentreandPBWforfiniteW} we explain that $\hU(\g,e)$ admits a $p$-centre with a natural inclusion $\Spec \hZ_p(\g,e) \subseteq (\g^*)^{(1)}$. Thus if $X \subseteq \Spec \hZ_p(\g,e)$ then we can also use the notation $\hU_X(\g,e)$; we refer the reader to Section~\ref{ss:reducedfiniteW} for more detail. For example, the category of semi-restricted $\hU(\g_0, e_0)$-modules is $\hU_{\chi_0 + \kappa_0 \z(\g_0)}(\g_0, e_0)\lmod$.

The parabolic induction functor is constructed as a composition of functors, as follows
\begin{eqnarray}
    \label{eq:composefunctors}
    \begin{array}{rcccl}
    \hU_{\chi_0 + \kappa_0\z(\g_0)}(\g_0,e_0)\lmod 
    & \overset{\eqref{e:skryabin1}}{\to} &
    U_{\chi_0 + \kappa_0\z(\g_0)}(\g_0)\lmod 
    & \overset{\eqref{e:classicalinductionfunctor}}{\to} &
    U_{\chi + \kappa\z(\g_0)}(\g)\lmod \\
    & \overset{\eqref{e:skryabin2}}{\to} & 
    \hU_{\chi + \kappa\z(\g_0)}(\g,e)\lmod 
    &\overset{\eqref{e:pullbackfunctor}}{\to} &
    U(\g,e)\lmod
    \end{array}
\end{eqnarray}

The first and third functors are induced by Skryabin's equivalence. Under these functors, with the relevant restriction on $p$-support, the minimal modules of $U(\g_0)$ and $U(\g)$ correspond to one dimensional modules over $\hU(\g_0,e_0)$ and $\hU(\g,e)$, respectively. The second functor appearing in \eqref{eq:composefunctors} is parabolic induction for reduced enveloping algebras, while the fourth is simply pull-back through the natural homomorphism $U(\g,e)\to\hU_{\chi+\kappa\z(\g_0)}(\g,e)$. Using these facts, we can deduce Theorem~\ref{T:main}(1) by showing for particular $(\g_0,e_0)$ that the following map is surjective
\begin{eqnarray}
\label{eq:morphismbetweenmodules}
    \{\mbox{one dimensional semi-restricted }\hU(\g_0,e_0)\mbox{-modules}\}\to \{\mbox{one dimensional }U(\g,e)\mbox{-modules}\}.
\end{eqnarray}
Indeed, when $e$ lies in a unique sheet we show there exists a Levi subalgebra $\g_0$ and a rigid nilpotent $p$-character $\chi_0 \in \g_0^*$ (inducing to $\chi$) such that the variety of one dimensional $U(\g,e)$-modules is an affine space whose dimension is precisely the dimension of the variety of one dimensional semi-restricted $\hU(\g_0,e_0)$-modules. This ensures that the finite morphism \eqref{eq:morphismbetweenmodules} is surjective.

Using Property (4) of the parabolic induction functor, listed above, a similar argument proves Theorem~\ref{T:main}(2). For this, we must choose (when possible) the Levi subgroup $G_0$ with Lie algebra $\g_0$ to have the property that there exists a finite group $\Gamma\subseteq G_0\cap C^e$ whose order is not divisible by $p$ and for which $\Gamma$-stability of one dimensional representations coincides with $G^{\chi}$-stability, i.e. $\k\Gamma\lmod$ is semisimple and $\Gamma \subseteq G^\chi$ gives rise to a surjection $\Gamma \onto G^\chi / (G^\chi)^\circ$. Such restriction means that we may not always assume $\chi_0$ is rigid, but we can construct $\chi_0$ lying on a unique sheet of $\g_0$, and thus apply Theorem~\ref{T:main}(1). In this case we show that the variety of $\Gamma$-stable one dimensional $U(\g,e)$-modules is an affine space of the desired dimension.

\subsection{Structure of the paper}

In Section~\ref{S:preliminaries} we fix some general notation and gather some elementary facts from commutative algebra.

In Section~\ref{S:groupsandLiealgebras} we survey several theories which are fundamental to our constructions: associated cocharacters, minimal modules, Lusztig--Spaltenstein induction, decomposition classes and sheets. We also explain how to generalise the main theorem to arbitrary $p$-characters for standard reductive groups.

In Section~\ref{S:parabolicsection} we recap the theory of modular finite $W$-algebras, mostly following \cite{GTmod}. We recall the notions of $p$-centres and the PBW theorems and recall the Poisson structure on the transverse slice. We recap the properties of reduced finite $W$-algebras, which are obtained by specialising the $p$-character, similar to the reduced enveloping algebra mentioned above. Finally we obtain general results on the abelian quotients of the finite $W$-algebra, using methods inspired by \cite{PrCQ}.

In Section~\ref{S:parabolicinduction} we start by recapping Skryabin's equivalence and showing that it is $C^e$-equivariant on simple modules. We then construct our parabolic induction functor and spend the rest of that section proving Properties (1)--(4) from Section~\ref{ss:sketch} (Theorem~\ref{T:parabolicinductiontheorem}).

In Section~\ref{S:classicalLiealgebras} we focus on $G = \SO_N$ or $\Sp_N$. The majority of the work involves constructing a special choice of (Lusztig--Spaltenstein) induction datum. Using detailed information on the structure of centralisers in classical Lie algebras from \cite{PT1} we can bound the dimensions of varieties of one dimensional representations of the finite $W$-algebras (see Theorem~\ref{T:abelianquotients}). We complete the proof of the main theorem for classical cases in Section~\ref{ss:proofofmain}.

Finally, in Section~\ref{S: Exceptionals} we consider the case when $G$ is of exceptional type. The majority of work in this section is in computing a certain value $c_\pi(e)$ associated to $e$ and $\g$; the main theorem in the exceptional cases is then proved in Sections~\ref{ss:Exc sing sheet} and \ref{ss:Exc even}, using similar arguments to the classical case.

\subsection*{Acknowledgements} The authors would like to thank Filippo Ambrosio, Micha{\"e}l Bulois, Ivan Losev and Alexander Premet for useful discussions over the past several years. The first author is supported by EPSRC grant EP/R018952/1, the second and third authors are supported by UKRI FLF, grant numbers MR/S032657/1, MR/S032657/2, MR/S032657/3, MR/Z000394/1, and the third author was partially supported by a research fellowship from Royal Commission for the Exhibition of 1851.

\section{Preliminaries}
\label{S:preliminaries}
We fix $\k$ once and for all to be an algebraically closed field of characteristic $p > 0$. As the paper progresses we will place further restrictions on $p$.

All vector spaces, algebras, varieties and Lie algebras will be defined over $\k$. For a finitely generated commutative $\k$-algebra $A$ we write $\Spec A$ for the affine scheme of prime ideals. If we refer to a scheme as a variety then we are implicitly only interested in the topological subspace of closed points.

If $A$ is a $\k$-algebra we write $A\lmod$ for the category of finitely generated $A$-modules.

\subsection{Twists}
\label{ss:twists}
When a group acts by automorphisms on a group or an algebra, the modules of the latter can be twisted by the action of the former. In this paper twisting will manifest itself in two rather different ways: the Frobenius twist of an affine scheme and the twist of a representation by automorphisms. We summarise these notions of twist now, fixing the relevant notation.

The Frobenius map on $\k$ is the bijective ring homomorphism $\k \to \k$ given by $a\mapsto a^p$. The Frobenius twist of an affine $\k$-scheme $X$ is the $\k$-scheme $X^{(1)}$ which is equal to $X$ as a scheme but with the $\k$-linear structure on $\k[X]$ twisted by the inverse of the Frobenius on $\k$. The resulting $\k$-algebra is denoted $\k[X]^{(1)}$; we may similarly apply this process to any $\k$-vector space $V$, and we denote the resulting $\k$-vector space by $V^{(1)}$. If $X$ is a reduced scheme then $X^{(1)}$ may be naturally identified with $\Spec \k[X]^p$.

Since we may view any $\k$-vector space as a $\k$-scheme, this might first appear to lead to some ambiguity. Fortunately, these constructions are compatible in the obvious sense.

If $A$ is an algebra, $V\in A\lmod$ and $g\in \Aut(A)$ then we may define the $g$-twisted $A$-module ${}^g V$ which is equal to $V$ as a vector space, with twisted action $a \cdot_g v := (g^{-1}\cdot a)\cdot v$ for $a\in A$, $v\in V$. In this way, $\Aut(A)$ may be viewed as a group of autoequivalences of $A\lmod$. If $G \subseteq \Aut(A)$ then we say that $V$ is $G$-stable if it is fixed by these autoequivalences, in other words, if ${}^g V \cong V$ for every $g\in G$.

We record here a general fact to be used a few times: if $B \subseteq A$ are associative $\k$-algebras, $V\in B\lmod$, and $g \in \Aut(A)$ then there is a natural isomorphism ${}^g (A \otimes_{B} V) \isoto A \otimes_{g(B)} {}^g V$ given by $u \otimes v \mapsto (g\cdot u) \otimes v$ for $u \in A$, $v \in V$. In particular, if $g$ stabilises $B$ then ${}^g (A \otimes_{B} V) \isoto A \otimes_{B} {}^g V$.

\subsection{Abelian quotients}
\label{ss:basicsonabquotients}
Let $A$ be a finitely generated $\k$-algebra and $\Gamma$ a group of automorphisms of $A$. We write $A^\ab$ for the maximal abelian quotient of $A$, obtained by factoring $A$ by the ideal generated by $[a,b]$ with $a,b \in A$. The action of $\Gamma$ descends to $A^\ab$ and we write $A^\ab_\Gamma$ for the coinvariant algebra, obtained by factoring $A^\ab$ by the ideal generated by $a - \gamma\cdot a$ with $a\in A$ and $\gamma \in \Gamma$. 

Now let $Z\subseteq A$ be a central subalgebra and assume that $A$ is finitely-generated as a $Z$-module; this implies that $Z$ is a finitely-generated noetherian $\k$-algebra and that $A$ is noetherian (see \cite[\textsection III.1]{BG}). If $I \subseteq Z$ is any ideal then we can consider the algebra $A^\ab \otimes_Z Z/I \cong (A \otimes_Z Z/I)^\ab$. With $Z \subseteq A$ as above and $\chi \in \Spec Z$, we write $A_\chi$ for the corresponding central reduction; in this notation, the previous isomorphism becomes $(A^\ab)_\chi\cong (A_\chi)^\ab$.
\begin{Lemma}
\label{L:basiclemma}
The following hold:
\begin{enumerate}
\setlength{\itemsep}{4pt}
\item[(i)] The closed points of $\Spec(A^\ab)$ parametrise the one dimensional representations of $A$.
\item[(ii)] The closed points of $\Spec(A^\ab_\Gamma)$ parametrise those one dimensional representations of $A$ which are $\Gamma$-stable.
\item[(iii)] If $K = \Ker(Z \to A^\ab)$ then 
$$\Spec(Z/K) = \{\chi \in \Max Z \mid A_\chi \text{ admits a one dimensional } A\text{-module} \}.$$
\item[(iv)] If $K_\Gamma = \Ker(Z \to A^\ab_\Gamma)$ then 
$$\Spec(Z/K_\Gamma) = \{\chi \in \Max Z \mid A_\chi \text{ admits a }\Gamma\text{-stable one dimensional } A\text{-module} \}.$$
\end{enumerate}
\end{Lemma}
\begin{proof}
Both (i) and (ii) follow from the Nullstellensatz, so we turn to (iii).

Let $I_\chi\in\Spec(Z)$ be a maximal ideal such that $A_\chi$ admits a one dimensional representation, so $(A_\chi)^\ab\neq 0$. The composition $\phi:Z\to A \to A_\chi\to (A_\chi)^\ab$ is therefore non-zero as it preserves the identity, and $I_\chi\subseteq \Ker(\phi)$ by construction. Since $I_\chi$ is maximal, we thus have $I_\chi=\Ker(\phi)$. Furthermore, $\phi$ coincides with the composition $Z\to A \to A^\ab\to (A^\ab)_\chi\cong (A_\chi)^\ab$ and thus $K\subseteq \Ker(\phi)=I_\chi$; hence $I_\chi\in\Spec(Z/K)$.

Conversely, suppose $I_\chi$ is a maximal ideal of $Z$ containing $K$; then $I_\chi/K$ is a maximal ideal of $Z/K$, which we may view as a subalgebra of $A^\ab$. It suffices to show that $(I_\chi/K)A^\ab$ is a proper ideal of $A^\ab$, since then $(A_\chi)^\ab\cong A^\ab/(I_\chi/K)A^\ab\neq 0$. Since $A$ is a finitely-generated $\k$-algebra which is finitely-generated as a $Z$-module, it is immediate that $A^\ab$ is a finitely-generated $\k$-algebra which is finitely-generated as a $Z/K$-module. We may thus apply the lying-over theorem \cite[Proposition III.1.1(2)]{BG} to conclude that $(I_\chi/K)A^\ab$ is contained in a prime ideal $P$ of $A^\ab$ such that $P\cap Z/K=I_\chi/K$. Since $I_\chi/K\neq Z/K$ we must have $P\neq A^\ab$ and thus $(I_\chi/K)A^\ab$ must be a proper ideal of $A^\ab$ as required.

This proves (iii); the argument for (iv) is essentially the same.
\end{proof}

\subsection{A condition for finiteness}
We record a general result about finite extensions of commutative $\k$-algebras, which we use in Section~\ref{ss:twistingpcharacters} to prove the finiteness of a certain morphism of schemes. The following is a slight generalisation of \cite[p. 144]{kraft}.
\begin{Lemma}
\label{L:finitelemma}
Suppose that $A=\k[x_1,...,x_n]$ is positively graded with each $x_i$ homogeneous. If $f_1,...,f_m \in A$ are homogeneous and $\sqrt{(f_1,...,f_m)} = A_{>0} = \bigoplus_{i>0} A_i$ then $A$ is finite over $\k[f_1,...,f_m]$.
\end{Lemma}
\begin{proof}
Write $I = (f_1,...,f_m)$, $d = \max_{i} \deg x_i$ and $d_i = \deg f_i$. For $1\le i_1 \le \cdots \le i_k \le n$ the monomial $x_{i_1} \cdots x_{i_k}$ has degree $\le kd$ and it follows that for $r \ge kd$ we have $A_r \subseteq A_{>0}^k$. The assumption $\sqrt{I} = A_{>0}$ implies that $A_{>0}^N \subseteq I$ for some $N > 0$. Combining these facts, $A_r \subseteq I$ for all $r \ge Nd$. Since the $f_i$ are homogeneous it follows that $A_r = \sum_{i=1}^m A_{r - d_i} f_i$ for all such $r$ and, by induction, we see that $A$ is generated by $\bigoplus_{i=0}^{Nd-1} A_i$ as a $\k[f_1,...,f_m]$-module.
\end{proof}

\section{Reductive groups and Lie algebras}
\label{S:groupsandLiealgebras}
The following notation and hypotheses will be used throughout the current section:
\begin{itemize}
\setlength{\itemsep}{4pt}
\item $G$ is a reductive algebraic $\k$-group, satisfying the standard hypotheses  \cite[\textsection 6.3]{JaLA}; we call $G$ a {\em standard reductive group}.
\item $\g = \Lie(G)$ and $U(\g)$ is the universal enveloping algebra.
\item A choice of non-degenerate $G$-invariant bilinear form on $\g$ is denoted $\kappa : \g \times \g \to \k$. We often abuse notation and also write $\kappa:\g\xrightarrow{\sim} \g^{*}$ for the corresponding $G$-module isomorphism $x\mapsto \kappa(x,-)$.
\item The natural $G$-equivariant restricted structure on $\g$ is denoted $x \mapsto x^{[p]}$.
\item $\Nc(\g)$ is the set of nilpotent elements of $\g$, i.e. those satisfying $x^{[p]^r} = 0$ for $r \gg 0$.
\item If $x\in \g$ or $x\in \g^*$ then the (co)adjoint stabiliser in $G$ is denoted $G^x$, the (co)adjoint stabiliser in $\g$ is denoted $\g^x$, and the (co)adjoint orbit in $\g$ or $\g^*$ is denoted $G\cdot x$.
\end{itemize}

The standard hypotheses imply several useful properties for Lie algebras, which we record here for future reference.
\begin{Remark}
\label{R:standardreductivedecomp}
    \begin{enumerate}
        \item If $G$ is standard reductive then the (co)adjoint stabilisers are smooth: for $x\in \g$ (resp. $\g^*$) we have $\g^x = \Lie(G^x)$ \cite[\textsection 2.9]{JaNO}.
        \item By \cite[\textsection 2.1]{PS}, if $G$ is a standard reductive group then $\g = \Lie G$ is isomorphic to
    $\z \oplus \g_1 \oplus \cdots \oplus\g_r$ where $\z$ is a subalgebra of the centre of $\g$ and each $\g_i$ is a Lie algebra of one of the following types:
    \begin{itemize}
        \item $\gl_{k}$ for some $k \in \N$ with $p \mid k$; or
        \item the Lie algebra of a simply-connected, simple algebraic group for which $p$ is very good.
    \end{itemize}
    The stabilisers in $G$ of semisimple elements of $\g$ are Levi subgroups of $G$. All Levi subgroups of $G$ are standard reductive. 
    \item If $G$ is simple and standard reductive then every isogeny from $G$ is separable, i.e. every finite homomorphism of algebraic groups from $G$ induces an isomorphism of Lie algebras (this follows from \cite[Proposition~2.4.4]{Let} since $\SL_{kp}$ is not standard reductive).
    \end{enumerate}
\end{Remark}

\subsection{The associated cocharacter and the component group}
\label{ss:associatedcochar}

The following facts are taken from \cite[\textsection 5]{JaNO}. Given $e\in \Nc(\g)$ there exists a cocharacter $\gamma_e: \k^\times \to G$ such that:
\begin{enumerate}
\setlength{\itemsep}{4pt}
\item[(i)] $\gamma_e(t)\cdot e = t^2 e$ for all $t\in \k^\times$;
\item[(ii)] $e$ is distinguished in a Levi subalgebra $\l = \Lie(L)$ such that $\gamma_e : \k^\times \to [L, L]$.
\end{enumerate}
Any two cocharacters satisfying (i) and (ii) are $(G^e)^\circ$-conjugate. We refer to $\gamma_e$ as {\it an associated cocharacter for $e$}. 

Note that $\gamma_e$ gives rise to a grading $\g = \bigoplus_{i\in \Z} \g(i)$ where
\begin{eqnarray}
\label{e:gradedspaces}
\g(i) := \{x\in \g \mid \gamma_e(t)\cdot x = t^i x \text{ for } t\in \k^\times\}
\end{eqnarray}
which we call the {\it associated grading}. One of the key properties of the associated grading is:
\begin{itemize}
\item [(iii)] $\g^e \subseteq \bigoplus_{i \ge 0} \g(i)$.
\end{itemize}
When a grading satisfies property (iii) and $e\in\g(2)$ we say that it is a {\it good grading for $e$}. We direct the reader to \cite[\textsection 3]{GTmod} and the references therein for the background theory of good gradings.

The centraliser of $\gamma_e(\k^\times)$ in $G^e$ is denoted $C^e$, and is referred to as {\it the reductive part of the centraliser}. Thanks to \cite[\textsection 5.10]{JaNO} it is a reductive group, and $G^e$ is the semi-direct product of $C^e$ and its unipotent radical, which is connected. As a consequence
\begin{eqnarray}
\label{e:componentgroupreductivecentraliser}
G^e / (G^e)^\circ \cong C^e / (C^e)^\circ.
\end{eqnarray}

\subsection{The $p$-centre, coadjoint orbits and the dimensions of $\g$-modules}
The $p$-centre of the universal enveloping algebra $U(\g)$ is the central subalgebra $Z_p(\g)$ generated by the elements $x^p - x^{[p]}$ for all $x\in\g$. The $\k$-linear map $\g^{(1)} \to Z_p(\g)$ defined by $x \mapsto x^p - x^{[p]}$ extends to a $G$-equivariant algebra isomorphism $\k[(\g^*)^{(1)}] \xrightarrow{\sim} Z_p(\g)$.

Since we may identify $\g^*$ with $(\g^*)^{(1)}$, each $\chi\in \g^{*}$ corresponds to a unique maximal ideal $I_\chi$ of $Z_p(\g)$. The corresponding quotient $U_\chi(\g)=U(\g)/I_\chi U(\g)$ is known as a reduced enveloping algebra. More generally, if $X \subseteq \g^*$ is any closed subvariety with defining ideal $I_X\subseteq Z_p(\g)$, then we denote the corresponding quotient by $U_X(\g)=U(\g)/I_X U(\g)$.

Note that if $\chi\in\g^{*}$ and $\m \subseteq \g$ is a vector subspace then $$U_{\chi+\m^\perp}(\g)=U(\g) / U(\g) \{x^p - x^{[p]} - \chi(x)^p \mid x\in \m\},$$
where $\m^\perp = \{\eta \in \g^* : \eta|_\m = \chi|_\m\}$.

The following is a direct application of Premet's theorem \cite{PrKW} that all $U_\chi(\g)$-modules have dimension divisible by $p^{d_\chi}$ (where $d_\chi := \frac{1}{2} \dim G \cdot \chi$) and Premet-Topley's theorem \cite{PT2} that $U_\chi(\g)$ actually admits a module of dimension $p^{d_\chi}$.

\begin{Lemma}
\label{L:Humphreyscorollary}
If $G$ is a standard reductive group over $\k$ of characteristic $p>0$, then the following are equivalent for each $\chi\in\g^{*}$:
\begin{enumerate}
\setlength{\itemsep}{4pt}
\item $U_\chi(\g) \text{ admits a module of dimension } p^k$;
\item $\dim (G\cdot \chi) \le 2k$.
\end{enumerate}
\end{Lemma}

\subsection{Lusztig--Spaltenstein induction}
\label{ss:LSinduction}

In this section we assume that the reader is familiar with the basic theory of parabolic and Levi subgroups of algebraic groups over fields of positive characteristic, and their Lie algebras (see \cite[\textsection 14]{Bo}, for example). Recall that any parabolic subgroup $P$ of $G$ has a (non-unique) Levi decomposition $P=G_0R$ into a reductive group $G_0$ and a unipotent group $R$ which is the unipotent radical of $P$. We call $G_0$ a Levi subgroup, and note that Levi subgroups of standard reductive groups are also standard reductive. This induces a corresponding decomposition $\p=\g_0\oplus \r$. For any semisimple element $x_{\rm s}$ in $\g$, the centraliser $G^{x_{\rm s}}$ is a Levi subgroup of $G$ and its Lie algebra is the centraliser $\g^{x_{\rm s}}$ in $\g$. 

Given a Levi subgroup $G_0$ with Lie algebra $\g_0$ we may consider the adjoint $G_0$-action on $\g_0$ and thus the nilpotent orbits $\O_0\subseteq \Nc(\g_0)$. For $g\in G$, $g \cdot \g_0 := \Ad(g)(\g_0)$ is a Levi subalgebra of $\g$ and $\Ad(g)(\O_0)$ (written as $g\cdot\O_0$) is a nilpotent $g G_0 g^{-1}$-orbit in $g\cdot\g_0$. Thus $G$ acts on the set of pairs $(\g_0,\O_0)$ where $\g_0$ is a Levi subalgebra and $\O_0$ is a nilpotent orbit in $\g_0$.

We now discuss Lusztig-Spaltenstein induction, first introduced in \cite{LS} for conjugacy classes in algebraic groups, and studied in the setting of the present paper in \cite{PS}. Suppose that $\g_0\subseteq \g$ is a Levi subalgebra and $\O_0 \subseteq \Nc(\g_0)$ is a nilpotent orbit. Choose a parabolic subalgebra $\p$ admitting $\g_0$ as a Levi factor and write $\r = \Rad(\p)$ so that $\p = \g_0 \oplus \r$. Since $\O_0 + \r$ is an irreducible subvariety of $\Nc(\g)$ there is a unique nilpotent $G$-orbit intersecting $\O_0 + \r$ densely. More remarkably, this nilpotent orbit only depends upon the $G$-orbit of $(\g_0, \O_0)$, not upon the choice of parabolic $\p$ admitting Levi factor $\g_0$. This nilpotent orbit is denoted $\Ind_{\g_0}^\g(\O_0)$ and the operation $\Ind_{\g_0}^{\g} : \Nc(\g_0) /G_0 \to \Nc(\g)/G$ is known as {\it Lusztig--Spaltenstein induction}.

We denote by $\Indat(\g)$ the set of pairs $(\g_0,\O_0)$ where $\g_0$ is a proper Levi subalgebra of $\g$ and $\O_0$ is a nilpotent orbit in $\g_0$, and we denote by $\Indat(\g)/G$ the set of $G$-orbits of such pairs. We consider Lusztig-Spaltenstein induction to be a map $\Ind^\g:\Indat(\g)\to \Nc(\g)/G$ which factors through the obvious map $\InD^\g:\Indat(\g)/G\to \Nc(\g)/G$. 

When $\O \in \Nc(\g)/G$ is in the image of $\InD^\g$ (or equivalently $\Ind^\g$), we say that $\O$ is {\it properly induced} from each $(\g_0, \O_0)\in(\Ind^\g)^{-1}(\O)$ and we refer to such $(\g_0, \O_0)$ as an {\it induction datum for $\O$}. We denote $\Indat(\O)=(\Ind^\g)^{-1}(\O)$, the set of induction data for $\O$, and also write $\Indat(\O)/G=(\InD^\g)^{-1}(\O)$, the set of $G$-orbits of induction data for $\O$.

If an orbit is not properly induced then it is {\it rigid}. If $\O_0$ is rigid in $\g_0$ we call $(\g_0,\O_0)$ a {\it rigid induction datum for $\O$}. We write $\RIndat(\O):=\{(\g_0,\O_0)\in\Indat(\O)\mid \O_0 \mbox{ rigid in }\g_0\}$ and we set 
$$\RIndat(\g):=\bigcup_{\O\in\Nc(\g)/G} \RIndat(\O).$$

 If we write $G_\C$ for the reductive algebraic group over $\C$ with the same root datum as $G$ and write $\g_\C=\Lie(G_\C)$, then in good characteristic there is a canonical bijection $\Nc(\g)/G\simeq \Nc(\g_\C)/G_\C$, coming from the fact that both are indexed by the Bala-Carter labelling \cite[\textsection 4]{JaNO}. Since Levi subgroups of standard reductive groups remain standard reductive, the well-known classification of Levi subgroups also gives a canonical bijection $\Indat(\g)/G\simeq \Indat(\g_\C)/G_\C$.
 
 By \cite[Theorems~1.3, 1.4]{PS} the following diagram commutes:
\begin{eqnarray}
	\label{e: LS indep p}
	\begin{array}{c}\xymatrix{
			\Indat(\g)/G \ar@{->}[d]^{\Ind_{G}^{\g}} \ar@{->}[rr]^{1\text{-}1} & &  \Indat(\g_\C)/G_\C \ar@{->}[d]^{\Ind_{G_\C}^{\g_\C}} \\
			\Nc(\g)/G \ar@{->}[rr]^{1\text{-}1}	&  & \Nc(\g_\C)/G_\C.
		}
	\end{array}
\end{eqnarray}

Induction of orbits enjoys a number of nice properties, including preservation of dimension of stabilisers \cite[Theorem~2.8]{PS}. In particular, given $(\g_0,\O_0)\in \Indat(\O)$ there is equality
\begin{eqnarray}
	\label{e:LScodimension}
	\dim \g_0^{e_0} = \dim \g^e
\end{eqnarray}
for all $e_0\in\O_0$ and $e\in\O$.

For what follows, fix $(\g_0,\O_0)\in\Indat(\g)$ and a parabolic subgroup $P$ of $G$ such that $\g_0$ is a Levi factor of $\p=\Lie(P)$. Denote $\r=\Rad(\p)$, so that $\p=\g_0\oplus \r$, and let $\r_{-}$ be the opposite nilradical so that $\g=\r_{-}\oplus \g_0\oplus \r$. Let $\O=\Ind^\g(\g_0,\O_0)$, and fix $e\in \O\cap(\O_0 + \r)$.

We record here a useful lemma. 
\begin{Lemma}
	\label{L:centraliserinparabolic}
	Maintain the notation from before the statement of the  lemma. Then:
	\begin{enumerate}
		\setlength{\itemsep}{4pt}
		\item[(i)]  $(G^e)^\circ \subseteq P$; and
		\item[(ii)] if $g\in G$ and $g\cdot e \in \O_0  + \r$ then there exists $z \in G^e$ such that $zg \in P$. $\hfill \qed$
	\end{enumerate}
\end{Lemma}

\begin{proof}
	Part (i) is \cite[Proposition 13.17(c)]{JaNO}, and part (ii)  follows easily from the fact that $g\cdot e\in \O \cap (\O_0 + \r)=P\cdot e$. 
\end{proof}

\begin{Corollary}[cf. Lemma~6.1.3 of \cite{Lo}]\label{Cor: gamma in G0} 
	Maintain the notation from Lemma~\ref{L:centraliserinparabolic}(ii). Let $\gamma_e : \k^\times \to G$ be an associated cocharacter for $e$. Then, after replacing $G_0$ with a $P$-conjugate if necessary, $\gamma_e(\k^\times)\subseteq G_0$.
\end{Corollary}
\begin{proof}
	We first prove that $\gamma_e(\k^\times) \subseteq P$. For $t\in \k^\times$ we can choose $z_t \in G^e$ such that $z_t\gamma_e(t) \in P$, by Lemma~\ref{L:centraliserinparabolic}. By part (i) of that lemma $(G^e)^\circ \subseteq P$ and so $\{\gamma_e(t) P\mid t\in\k^\times\}$ is a finite, connected subset of the partial flag variety $G/P$. Thus $\gamma_e(t) P = P$ for all $t$, and $\gamma_e(\k^\times)\subseteq P$ by \cite[Theorem~11.16]{Bo}.
	
	Now, since $\gamma_e(\k^\times)$ is a torus it is linearly reductive and thus by \cite[Lemma 11.24]{JaNO} lies inside a Levi subgroup of $P$. By replacing $G_0$ with a $P$-conjugate if necessary, we may assume $\gamma_e(\k^\times)\subseteq G_0$.	
\end{proof}

\begin{Remark}
	 Note that the property $e\in\O\cap(\O_0+\r)$ is preserved when we replace $G_0$ with a $P$-conjugate, since $\O\cap(\O_0+\r)=P\cdot e$.
\end{Remark}

We observe finally that $\gamma_e(\k^\times)\subseteq G_0$ implies that $\z(\g_0)\subseteq \g(0)$, where $\g=\oplus_{i\in\Z}\g(i)$ is the grading associated to $\gamma_e$. Therefore, we have
\begin{eqnarray}
	\label{e: z perp -1}
	\kappa(\z(\g_0),\bigoplus_{i\leq -1}\g(i))=0.
\end{eqnarray}

\subsection{Decomposition classes and sheets}
\label{ss:decompclasses}
The sheets of $\g$ are the irreducible components of the varieties $\g_{[m]}:=\{x\in \g \mid \dim (G\cdot x) = m\}$. The standard hypotheses \cite[\textsection 6.3]{JaLA} which we have adopted ensure that $\g$ and $\g^*$ are isomorphic as $G$-modules, and so our remarks here apply equally well to the sheets of $\g^*$. Under the standard hypotheses, sheets are classified using the theory of decomposition classes and induced orbits. We recall here the basics of this theory, largely following \cite[\textsection 2]{PS}.

Given $(\g_0,\O_0)\in\Indat(\g)$ and $e_0\in\O_0$, we define the decomposition class $$\D(\g_0,e_0):=G\cdot (e_0 + \z(\g_0)_\reg),$$ where $\z(\g_0)_\reg = \{x\in \z(\g_0) \mid \g^x = \g_0\}$. Write $\Decomp(\g)$ for the set of decomposition classes of $\g$; since a decomposition class depends only on the orbit $\O_0\subseteq \g_0$ rather than our choice of $e_0$, we may view this construction as a map ${\D}:\Indat(\g)\twoheadrightarrow \Decomp(\g)$; this in fact induces a bijection ${\D}:\Indat(\g)/G\to\Decomp(\g)$.

There is a natural surjective map $$\rho:\g\to \Indat(\g)/G,\qquad x\mapsto (\g^{x_{\rm s}}, G^{x_{ \rm s}} \cdot x_{ \rm n}).$$ The decomposition class of $x$ is then defined to be $\D(\rho(x))$, and the pair $(\g^{x_{\rm s}}, G^{x_{ \rm s}} \cdot x_{ \rm n})$ is called the {\it induction datum associated to $x$}. The decomposition classes of $\g$ may then alternatively be defined as the equivalence classes of $\g$ under the relation $$x \sim y \iff \rho(x)=\rho(y).$$ In particular, $\D(\g_0,e_0)$ is the equivalence class containing $\z(\g_0)_\reg +e_0$.

There are finitely many $G$-orbits of induction data, and hence finitely many decomposition classes in $\g$. Since the orbits in a class have constant dimension, it follows that every sheet $\S\subseteq \g$ contains a unique dense decomposition class $\D$, and $$\S = \overline{\D}_\reg:=\{x\in\overline{\D}\mid \dim(G\cdot x)=\dim(G\cdot y) \mbox{ for all } y\in \D\}.$$ Part (iii) of the next result explains the classification of sheets.
\begin{Lemma}
\label{T:sheetstheorem}
Let $(\g_0, \O_0)\in\Indat(\g)$ and $\D=\D(\g_0,\O_0)$.
\begin{itemize}
\setlength{\itemsep}{4pt}
\item[(i)] $\Nc(\g) \cap \overline \D_\reg$ consists of the single nilpotent orbit $\Ind_{\g_0}^\g(\O_0)$.
\item[(ii)] $\dim \D = \dim \Ind_{\g_0}^\g(\O_0) + \dim \z(\g_0)$.
\item[(iii)] The sheets of $\g$ are the sets $\overline{\D}_\reg$, where $\D$ runs over the decomposition classes associated to $G$-orbits of rigid induction data.
\item[(iv)] Let $\D'\in\Decomp(\g)$. Then $\D'\subseteq \overline{\D}_\reg$ if and only if there exists $(\g_0',\O_0')$ with $\D' = \D(\g_0', \O_0')$ such that $\g_0\subseteq \g_0'$ and $\Ind_{\g_0}^{\g_0'}(\O_0)=\O_0'$.
\end{itemize}
\end{Lemma}
\begin{proof}
Parts (i), (ii), and (iii) are Proposition 2.5 and Theorem~2.8 of \cite{PS}. A proof of part (iv) can be extracted from the proof of \cite[Theorem~2.8]{PS} (in fact the equivalence is proven with the additional assumption that $\O_0'$ is rigid, and omitting that assumption yields the desired argument). 
\end{proof}

 The proof of the following result will occupy the remainder of Section~\ref{ss:decompclasses}.

\begin{Proposition}
\label{P:sheetstheorem}
    For $x\in \g$ with Jordan decomposition $x=x_{\rm s}+x_n$, the following sets are in bijection:
\begin{enumerate}
	\item The sheets of $\g$ containing $x$;
	\item The sheets of $\g^{x_{\rm s}}$ containing $x_{\rm n}$.
\end{enumerate}
\end{Proposition}

 To prepare for the proof of Proposition~\ref{P:sheetstheorem}, we note that the sheets of $\g$ containing $x$ are precisely those which contain the decomposition class $\D_x := G\cdot (\z(\g^{x_{\rm s}})_\reg + x_{\rm n})$, whilst the sheets of $\g^{x_{\rm s}}$ containing $x_n$ are those containing $G^{x_{\rm s}} \cdot x_{\rm n}$.

 By Lemma~\ref{T:sheetstheorem}(iv), the sheets of $\g$ containing $x$ are indexed by $\RIndat(G^{x_{\rm s}}\cdot x_{\rm n})/G$ and the sheets of $\g^{x_{\rm s}}$ containing $x_{\rm n}$ are indexed by $\RIndat(G^{x_{\rm s}}\cdot x_{\rm n})/G^{x_{\rm s}}$. 

There is a natural surjective map $$\RIndat(G^{x_{\rm s}}\cdot x_{\rm n})/G^{x_{\rm s}}\to \RIndat(G^{x_{\rm s}}\cdot x_{\rm n})/G$$ induced from the identity map on $\RIndat(G^{x_{\rm s}}\cdot x_{\rm n})$. We claim that this map is injective. Using \cite[\textsection 2.1]{PS}, it suffices to show this in the case where $G$ has indecomposable root system. We may assume that $G^{x_{\rm s}}$ is a standard Levi subgroup of $G$, which contains a fixed maximal torus $T$ of $G$.

Let $(\g_0,\O_0)$, $(\g_0',\O_0')\in \RIndat(G^{x_{\rm s}}\cdot x_{\rm n})$ and suppose that there exists $g\in G$ such that $g\cdot \g_0=\g_0'$ and $g\cdot \O_0=\O_0'$. We may assume without loss of generality that the corresponding Levi subgroups $G_0$ and $G_0'$ both contain  $T$, and in fact that they are standard Levi subgroups of $G^{x_{\rm s}}$ with respect to $T$. Let $\Phi$ be the root system of $G$ corresponding to $T$, and let $W=N_G(T)/T$ and $W_{\rm s}=N_{G^{x_{\rm s}}}(T)/T$ be the Weyl groups. Pick a choice of simple roots $\Pi$ for $\Phi$; there then exists a subset $\Pi_{\rm s}$ of simple roots such that $\Phi_{\rm s}:=\Phi\cap\Z \Pi_{\rm s}$ is the root system for $\g^{x_{\rm s}}$, a subset $\Pi_0\subseteq \Pi_{\rm s}$ of simple roots such that $\Psi:=\Phi\cap\Z \Pi_{0}$ is the root system for $\g_0$ and a subset $\Pi_0'\subseteq \Pi_{\rm s}$ of simple roots such that $\Psi':=\Phi\cap\Z \Pi_{0}'$ is the root system for $\g_0'$. Since $\g_0$ and $\g_0'$ are $G$-conjugate, there exists $w\in W$ such that $w(\Psi)=\Psi'$.

Let us write 
\begin{equation}\label{e: Phi decomp}
	\Phi_{\rm s}=\Phi_1\amalg\Phi_2\amalg\cdots\amalg \Phi_r
\end{equation} for the decomposition of $\Phi_{\rm s}$ into indecomposable root systems, and let us write $$\Psi=\Psi_1\amalg\Psi_2\amalg\cdots\amalg \Psi_r\qquad \mbox{and}\qquad \Psi'=\Psi'_1\amalg\Psi'_2\amalg\cdots\amalg \Psi'_r$$ for the associated decompositions of $\Psi$ and $\Psi'$, so that $\Psi_i\leq \Phi_i$ and $\Psi'_i\leq \Phi_i$ are (not-necessarily-indecomposable) root subsystems. Furthermore, for each $i=1,\ldots,r$, write 
$$\Psi_i=\Psi_{i1}\amalg\Psi_{i2}\amalg\cdots\amalg \Psi_{is_i}\qquad \mbox{and}\qquad \Psi'_i=\Psi'_{i1}\amalg\Psi'_{i2}\amalg\cdots\amalg \Psi'_{it_i},$$ for some $s_i,t_i\in\N$, for the decompositions into irreducible root systems. We may assume without loss of generality that $\Phi_i$ for $i<r$, $\Psi_{ij}$ for $i<r$ or $j<s_i$, and $\Psi'_{ij}$ for $i<r$ or $j<t_i$ are all of type ${\sf A}$. We also write $$W_{\rm s}=W_1\times \cdots \times W_r$$ for the decomposition of the Weyl group $W_{\rm s}$ corresponding to the decomposition \eqref{e: Phi decomp}. Let us write $$a_{in}=\#\{j\in\{1,\ldots,s_i\}\mid \Psi_{ij}\mbox{ is of type }{\sf A}_n\}\quad\mbox{and}\quad a'_{in}=\#\{j\in\{1,\ldots,t_i\}\mid \Psi'_{ij}\mbox{ is of type }{\sf A}_n\},$$ for $1\leq i\leq r$ and $n\in\N$.

To proceed further, we require the following lemma.

\begin{Lemma}
\label{L:rootsystemlemma}
	If $\Psi$ and $\Psi'$ are not $W_{\rm s}$-conjugate, then one of the following holds.
	\begin{enumerate}
		\item The pair $(\Psi_r,\Phi_r)$ lies on the following list:
		\begin{enumerate}
			\item $({\sf A}_{m_1}\times\cdots \times {\sf A}_{m_k},{\sf D}_{2l})$ where $(m_1+1)+\cdots +(m_k+1)=2l$ and each $m_i$ is odd.
			\item $({\sf A}_1\times {\sf A}_1\times {\sf A}_1,{\sf E}_7)$
			\item $({\sf A}_3\times {\sf A}_1,{\sf E}_7)$
			\item $({\sf A}_5,{\sf E}_7)$
			\item $({\sf A}_2\times {\sf A}_1,{\sf F}_4)$
			\item $({\sf A}_2,{\sf F}_4)$
			\item $({\sf A}_1,{\sf F}_4)$
			\item $({\sf A}_1,{\sf G}_2)$
		\end{enumerate}
		and the pair $(\Psi'_r,\Phi_r)$ has the same Cartan type as $(\Psi_r,\Phi_r)$ but is not $W_r$-conjugate to it.
		\item There exists $i<r$ and $n\in\N$ such that $a_{in}\neq a_{in}'$.
	\end{enumerate}
\end{Lemma}

\begin{proof}
	Suppose $a_{in}=a_{in}'$ for all $i<r$ and all $n\in\N$. Since $\Psi$ and $\Psi'$ are $W$-conjugate, this implies that in fact $a_{in}=a'_{in}$ for all $i\leq r$ and all $n\in\N$. 
	
	Let $i\in\{1,\ldots,r\}$. If $\Phi_i$ is of type ${\sf A}$ then this implies that $\Psi_i$ and $\Psi_i'$ are both parabolic root subsystems of $\Phi_i$ of the same type and thus that they are $W_i$-conjugate by \cite[Proposition 6.3]{BC}. There thus exists $w_i\in W_i$ with $w_i(\Psi_i)=\Psi'_i$. 
	
	On the other hand, suppose $\Phi_i$ is of type other than ${\sf A}$ (which by assumption can only happen at $i=r$). There are two cases to consider: (1) that both $\Psi$ and $\Psi'$ have indecomposable factors only of type {\sf A}, or (2) that both $\Psi$ and $\Psi'$ have exactly one indecomposable factor of type other than ${\sf A}$ (by our assumptions, these are $\Psi_{rs_r}$ and $\Psi'_{rt_r}$). Since $\Psi$ and $\Psi'$ are $W$-conjugate, the latter factors in both $\Psi$ and $\Psi'$ must have the same Cartan type. Thus, in both case (1) and case (2), the assumption that $a_{rn}=a'_{rn}$ for all $n\in\N$ implies that $\Psi_r$ and $\Psi_r'$ are both parabolic subsystems of $\Phi_r$ of the same Cartan type. By \cite[Proposition 6.3]{BC}, parabolic subsystems $\Psi_i$ and $\Psi'_i$ of the same Cartan type are $W_i$-conjugate unless the pairs $(\Psi_i,\Phi_i)$ and $(\Psi'_i,\Phi_i)$ occur on the list in the statement of the lemma. If $\Psi_i$ and $\Psi'_i$ are $W_i$-conjugate then there exists $w_i\in W_i$ such that $w_i(\Psi_i)=\Psi'_i$. In particular, setting $w_{\rm s}=(w_1,\ldots,w_r)\in W_1\times\cdots \times W_r\leq W_{\rm s}$ yields $w_{\rm s}(\Psi)=\Psi'$.
\end{proof}

We may now conclude the proof of Proposition~\ref{P:sheetstheorem}. Retain the above notation.

\begin{proof}
    Note that $\g^{x_{\rm s}}$ has a decomposition $\g^{x_{\rm s}}=\g_1\oplus\cdots\oplus\g_r\oplus \z$ as in \cite[\textsection 2.1]{PS}, where $\z$ is a central subalgebra and $\g_i$ is either a simple Lie algebra or a copy of $\gl_{kp}$ for some $k\in\N$. Furthermore, the corresponding root system of each $\g_i$ is precisely $\Phi_i$, and the adjoint action of $G^{x_{\rm s}}$ decomposes into actions of $G_1,\ldots,G_r$, the simple factors of $[G^{x_{\rm s}},G^{x_{\rm s}}]$.

If $\g_0$ and $\g_0'$ are not $G^{x_{\rm s}}$-conjugate then $\Psi$ and $\Psi'$ are not $W_{\rm s}$-conjugate and we are in the statement of the lemma. If (1) holds then $\O_0=0$ and $\O_0'=0$ and we may check from \cite[Corollary 7.3.4]{CM} and \cite[Tables 6--10]{DE} (which hold in positive characteristic by the commutativity of \eqref{e: LS indep p}) that $\Ind_{\g_0}^{\g^{x_{\rm s}}}(\O_0)\neq \Ind_{\g_0'}^{\g^{x_{\rm s}}}(\O_0')$, since the $\g_r$-components of these orbits do not coincide. On the other hand, if Lemma~\ref{L:rootsystemlemma}(2) holds and $i<r$ is such that $a_{in}\neq a'_{in}$ for some $n\in\N$ then $\Ind_{\g_0}^{\g^{x_{\rm s}}}(\O_0)\neq \Ind_{\g_0'}^{\g^{x_{\rm s}}}(\O_0')$ (since the $\g_i$-components of these orbits do not coincide, by the uniqueness up to conjugacy of rigid induction data in type ${\sf A}$). In both cases we thus have a contradiction.

This then proves that $\g_0$ and $\g_0'$ are $G^{x_{\rm s}}$-conjugate, so there exists $g_s\in G^{x_{\rm s}}$ such that $g_s\g_0'=\g_0$. What remains is therefore to show that we may choose $g_s$ such that $g_s\cdot \O_0'=\O_0$. This will follow if $g_sg\cdot \O_0=\O_0$ in $\g_0$. Letting $\O=\Ind_{\g^{x_{\rm s}}}^\g (G^{x_{\rm s}}\cdot x_{\rm n})$, we know that $(\g_0,\O_0)\in\RIndat(\O)$ and $(\g_0,g_sg\O_0)\in \RIndat(\O)$. The result will thus follow if the map $$\RIndat(\O)\to \{\mbox{Levi subalgebras of }\g\}, \qquad (\g_0,\O_0)\mapsto \g_0$$ is injective.

By the commutativity of \eqref{e: LS indep p}, it suffices to prove the injectivity of the corresponding map over $\C$. In this setting it follows easily from \cite[Remark 3.4]{GTW}.\end{proof}

\subsection{Extending the main theorem to arbitrary $p$-characters for standard reductive groups}
\label{ss:NilpotentReduction}

In this section, we explain that Theorem~\ref{T:main} generalises to standard reductive groups and arbitrary $p$-characters, with appropriate modifications. Before doing so, we articulate what we mean by parabolic induction in this paper. Let $\p$ be a parabolic subalgebra of $\g$ with Levi decomposition $\p=\g_0\oplus \r$, and let $\chi_0\in\g_0^*$. For $\chi\in\g^{*}$ with $\chi(\r)=0$ and $\chi\vert_{\g_0}=\chi_0$, we define the parabolic induction functor associated to $\p$ and $\g_0$ by
\begin{eqnarray}
	\label{e:pcharinductionfunctor}
	\begin{array}{rcl}
		U_{\chi_0}(\g_0)\lmod &\longrightarrow & U_{\chi}(\g)\lmod;\vspace{4pt}\\
		V &\longmapsto & U_{\chi}(\g) \otimes_{U_{\chi}(\p)} V,
	\end{array}
\end{eqnarray}
where we view $V$ as a $U_\chi(\p)$-module via the surjection $U_\chi(\p)\twoheadrightarrow U_{\chi_0}(\g_0)$. For $\chi\in\g^*$, we say that a $U_\chi(\g)$-module is {\it parabolically induced} if it lies in the essential image of a parabolic induction functor for some proper parabolic subalgebra $\p$ and some Levi decomposition $\p=\g_0\oplus \r$ with $\chi(\r)=0$.

We recall that we call an element $\chi \in \g^*$ rigid if the decomposition class of $\chi$ is dense in a sheet. Equivalently $\chi = \chi_{\rm s} + \chi_{\rm n}$ and $\chi_{\rm n}|_{\g^{\chi_{\rm s}}}$ is rigid nilpotent.

The following result explains how to generalise Theorem~\ref{T:main} to standard reductive groups and arbitrary $p$-characters.

\begin{Theorem}
\label{P:nilpotentreduction}
Suppose that Theorem~\ref{T:main} holds. Let $G$ be a standard reductive group and $\chi = \chi_{\rm s} + \chi_{\rm n}\in \g^*$ be the Jordan decomposition. Let $\g^{\chi_{\rm s}} = \z \oplus \bigoplus_i (\g^{\chi_{\rm s}})_i$ be the decomposition described in Remark~\ref{R:standardreductivedecomp} and suppose that for all $i$ the projection of $G^{\chi_{\rm s}} \cdot \chi_{\rm n}$ to $(\g^{\chi_{\rm s}})_i$ does not appear in Tables~\ref{Tab: Bad Induced Orbits} or \ref{Tab: Bad Good Orbits}. Then:
	\begin{enumerate}
		\setlength{\itemsep}{4pt}
		\item If $\chi$ lies in a unique sheet then every minimal $U_\chi(\g)$-module is parabolically induced from a minimal module for a reduced enveloping algebra with $p$-character.
		\item Every $G^\chi$-stable minimal $U_\chi(\g)$-module is parabolically induced from a minimal module for a reduced enveloping algebra with rigid $p$-character.
	\end{enumerate}
\end{Theorem}
\begin{proof}
	The proof is in five steps. Step 1 reduces to the case where $G$ is simple or $G=\GL_{pk}$ for some $k\in\N$. Steps 2 through 4 then reduce to the case of a nilpotent $p$-character when $G$ is a simple group. Step 5 concludes the argument for $G=\GL_{kp}$.
	

{\it Step 1: Reduction to the case $G$ simple or $G=\GL_{pk}$ for some $k\in\N$.}

Let $G$ be standard reductive. By Remark~\ref{R:standardreductivedecomp} there is a group $G' = Z \times G_1 \times \cdots \times G_r$ such that $\Lie(G')\cong \Lie(G)$, where $Z$ is a central torus and the factors $G_i$ are either isomorphic to $\GL_{pk}$ for some $k$, or are simply-connected, simple algebraic groups. Since the actions of both $G$ and $G'$ factor through $\Ad(G)$ it follows that a $\g$-module is $G^\chi$-stable if and only if it is $(G')^\chi$-stable.

Write $\chi = \chi_\z + \chi_1 +\cdots + \chi_r$ for the decomposition of $\chi$ along the decomposition $\g = \z \oplus \bigoplus_{i=1}^r \g_i$ described in Remark~\ref{R:standardreductivedecomp}. It is not hard to see the following:
\begin{enumerate}
    \item Every minimal $U_\chi(\g)$-module is parabolically induced from a rigid $p$-character if and only if every minimal $U_{\chi_i}(\g_i)$-module is parabolically induced from a rigid $p$-character for $i = 1,...,r$.
    \item $\chi$ lies in a unique sheet of $\g$ if and only if $\chi_i$ lies in a unique sheet of $\g_i^*$ for $i=1,...,r$ by Proposition~\ref{P:sheetstheorem}.
    \item Decomposing a minimal module $M_0 \in U_{\chi}(\g)\lmod$ as $M_1\otimes\cdots \otimes M_r$, where $M_i$ is a minimal $U_{\chi_i}(\g_i)$-module, we have that $M_0$ is $G_0^\chi$-stable if and only if $M_i$ is $G_i^{\chi_i}$-stable for $i=1,...,r$.
\end{enumerate}
This completes Step 1.

{\it Step 2: If the current theorem holds for nilpotent $p$-characters in simple groups then it holds for nilpotent $p$-characters for Levi subgroups of simple groups.}

The proof of this step is almost identical to Step 1 and so we omit the details. 

{\it Step 3: Reduction of Part (1) of the theorem to the nilpotent case.} 

Now let $G$ be a simple group, let $\chi \in \g^*$ be arbitrary, and suppose that conclusion (1) of the current theorem holds whenever $\chi$ is nilpotent. We will show that the same conclusion holds for arbitrary $\chi$. Thanks to Remark~\ref{R:standardreductivedecomp}(3) we can suppose that $G$ is simply-connected. Write $\chi = \chi_{\rm s} + \chi_{\rm n}$ for the Jordan decomposition, and note that $G_0 := G^{\chi_{\rm s}}$ is a Levi subgroup of $G$.

Let $M$ be a minimal $U_\chi(\g)$-module and $\p_0$ a parabolic subalgebra of $\g$ with Levi decomposition $\p_0=\g_0\oplus \u$ such that $\chi(\u)=0$. Since parabolic induction $U_\chi(\g_0)\lmod \to U_\chi(\g)\lmod$ is an equivalence of categories by \cite[Theorem~3.2]{FPmod}, there exists a $U_\chi(\g_0)$-module $M_0$ such that $U_\chi(\g)\otimes_{U_\chi(\p_0)}M_0\cong M$; by comparing dimensions, we see that $M_0$ is a minimal $U_\chi(\g_0)$-module. Let $E$ be a one dimensional $U_{-\chi_{\rm s}}(\g_0)$-module, which exists since $\chi_{\rm s}([\g_0, \g_0])=0$. Then $M_0\otimes E$ is a minimal $U_{\chi_{\rm n}}(\g_0)$-module and $\chi_{\rm n}\in \g_0^*$ is nilpotent.

If $\chi$ lies on a unique sheet of $\g^*$ then $\chi_{\rm n}$ lies on a unique sheet of $\g_0^{*}$, thanks to Proposition~\ref{P:sheetstheorem}. Now we can apply Step 2 of the proof to see that $M_0 \otimes E$ is parabolically induced from a rigid $p$-character. There thus exists a parabolic subalgebra $\p_1$ of $\g_0$ and a Levi factor $\g_1$ of $\p_1$ such that $\chi_{\rm n}\vert_{\g_1}\in\g_1^{*}$ is rigid nilpotent, and there exists a minimal $U_{\chi_{\rm n}}(\g_1)$-module $V$ such that $$M_0\otimes E \cong U_{\chi_{\rm n}}(\g_0)\otimes_{U_{\chi_{\rm n}}(\p_1)} V.$$ It is straightforward to check that this implies
\begin{eqnarray}
    \label{eq:parabolicinductiontwice}
M\cong U_\chi(\g)\otimes_{U_\chi(\p_0)}(U_{\chi}(\g_0)\otimes_{U_{\chi}(\p_1)} (V\otimes E^*))\cong U_\chi(\g)\otimes_{U_\chi(\p_1\oplus \u)} (V\otimes E^{*}).
\end{eqnarray}

Since $V\otimes E^{*}$ is a minimal $U_\chi(\g_1)$-module, and $\chi|_{\g_1}$ is rigid in the sense of the Introduction, we have shown that $M$ is parabolically induced from a minimal module for a rigid $p$-character.

{\it Step 4: Reduction of Part (2) of the theorem to the nilpotent case.}

Once again suppose that $G$ is simple and standard reductive. Suppose that $M \in U_\chi(\g)\lmod$ is a $G^{\chi}$-stable minimal module, and let $M_0$ and $E$ be defined as per Step 3. Then we claim that $M_0$ is $(G^{\chi_{\rm s}})^\chi$-stable and that $M_0\otimes E$ is  $(G^{\chi_{\rm s}})^{\chi_{\rm n}}$-stable.

The first of these claims follows from the last observation of Section~\ref{ss:twists}, setting $A=U(\g)$, $B=U(\p)$ and $V=M_0$, along with the fact that $G^\chi=(G^{\chi_{\rm s}})^\chi$. The second claim follows as $G^{\chi_{\rm s}}$ is a connected reductive group which acts regularly on the following finite set (see \cite[Theorem 3.14]{St} a proof of connectedness):  
	$$\Lambda_{\chi_{\rm s}}=\{\lambda:\g^{\chi_{\rm s}}/[\g^{\chi_{\rm s}},\g^{\chi_{\rm s}}]\to \k\mid \lambda(x)^p - \lambda(x^{[p]})=(-\chi_{\rm s}(x))^p\mbox{ for all }x \in \g^{\chi_{\rm s}}/[\g^{\chi_{\rm s}},\g^{\chi_{\rm s}}]\};$$ it thus acts trivially on $\Lambda_{\chi_{\rm s}}$ and hence twisting by $(G^{\chi_{\rm s}})^{\chi_{\rm n}}$ fixes the isoclass of $E$.

    Now we repeat the arguments which led to the isomorphism \eqref{eq:parabolicinductiontwice} to conclude that Theorem~\ref{T:main}(2) can be reduced to the nilpotent case. This concludes the proof.
    
{\it Step 5: The case $G=\GL_{kp}$ for $k\in\N$.}
The arguments in Steps 2, 3 and 4 apply equally in this case, and so we may reduce to $\chi$ nilpotent. The result then follows from \cite[Corollary 1.2]{GTmin}.\end{proof}

\begin{Remark}
\label{R: redtorep}
    When proving Theorem~\ref{T:main} it is sufficient to pick a representative for each orbit. If $\chi' = g\cdot \chi$ for some $g\in G$ then it is easy to show that twisting by $g$ induces a dimension preserving equivalence between $U_\chi(\g)\lmod$ and $U_{\chi'}(\g)\lmod$, which preserves the property of being parabolically induced, as well as sending $G^\chi$-stable modules to $G^{\chi'}$-stable modules.
\end{Remark}

\section{Modular finite $W$-algebras}
\label{S:parabolicsection}
We retain the notation and hypotheses of Section~\ref{S:groupsandLiealgebras}. On top of these, we fix the following:
\begin{itemize}
\setlength{\itemsep}{4pt}
\item $e\in \Nc(\g)$ is a fixed element and $\chi = \kappa(e, -) \in \g^*$.
\item $d_\chi = \frac{1}{2} \dim G\cdot \chi$.
\item $\gamma_e : \k^\times \to G$ is an associated cocharacter for $e$ and $\g = \bigoplus_{i\in \Z} \g(i)$ is the associated grading.
\item $C^e$ is the centraliser of $\gamma_e(\k^\times)$ in $G^e$.
\item $\Gamma \subseteq C^e$ is a finite subgroup such that $p$ does not divide $|\Gamma|$.
\end{itemize}

\subsection{Graded nilpotent Lie algebras and their algebraic groups} Before we begin to discuss finite $W$-algebras we demonstrate that a certain family of nilpotent Lie subalgebras of $\g$ are algebraic, which allows us to simplify the definition of the modular finite $W$-algebra given in \cite{GTmod}.

Let $\gamma : \k^\times \to G$ be a cocharacter, and let $\g = \bigoplus_{i\in \Z} \g(i)$ be the grading defined in \eqref{e:gradedspaces}. Fix $d > 1$ and let $\kk \subseteq \g(1-d)$ be a vector subspace. Write $\n_\kk = \kk + \bigoplus_{i \le -d} \g(i).$
\begin{Lemma}
\label{L:appendix}
There exists an algebraic group $N_\kk \subseteq G$ such that $\n_\kk = \Lie N_\kk$.
\end{Lemma}

\begin{proof}
Let $T \subseteq G$ be a maximal torus containing $\lambda(\k^\times)$, let $\Phi$ be the associated root system, and let $\{e_\alpha \mid \alpha\in \Phi\} \subseteq \g$ the corresponding root vectors. Set $\Phi_i = \{\alpha\in \Phi\mid e_\alpha \in \g(i)\}$ so that $\Phi = \bigsqcup_{i\in \Z} \Phi_i$. Now for all $\alpha \in \bigsqcup_{i \le -1}\Phi_{i}$ there is a one parameter subgroup $u_\alpha : \k \to G$ such that $d_1 u_\alpha = e_\alpha$.

Suppose $\Phi_{1-d} = \{\alpha_1,...,\alpha_m\}$ and abbreviate $e_j := e_{\alpha_j}$ for $j=1,...,m$. Pick a basis $\{x_1,...,x_k\}$ for $\kk \subseteq \g(1-d)$ and write $x_i = \sum_{j=1}^m \lambda_{i,j} e_j$. Now we define curves $v_i : \k \to G$ for $i =1,...,k$ by
\begin{eqnarray}
v_i(t) = u_1(\lambda_{i,1}t) u_2(\lambda_{i,2}t) \cdots u_m(\lambda_{i,m} t).
\end{eqnarray}
We let $N_0$ denote the subgroup of $G$ generated by the root subgroups $u_\alpha(\k)$ with $\alpha\in \Phi_{\le -d}$.
We claim that the smooth algebraic group $N_\kk := \langle N_0 , v_i(\k)\mid i=1,...,k\rangle$ (see \cite[Proposition 2.51]{Mi}) is the desired subgroup of $G$ with Lie algebra $\n_\kk$. 

The following three facts can be verified using \cite[\textsection 3.2]{Bo} and the formulas in \cite[(2.2)]{GTmod}, along with the chain rule.
\begin{enumerate}
\setlength{\itemsep}{4pt}
\item $d_1 v_i = x_i$ for $i =1,...,k$;
\item each $\{u_\alpha(\k) \mid \alpha \in \Phi_{1-d}\}$ normalises $N_0$;
\item $[v_i(s), v_j(t)] \in N_0$ for $1\leq i,j\leq k$ and $s,t\in \k$.
\end{enumerate}

As a consequence of (1) we have $\n_\kk \subseteq \Lie(N_\kk)$ whilst (2) and (3) together imply that $\dim \n_\kk = \dim N_\kk=\dim\Lie(N_\kk)$. This concludes the proof.
\end{proof}

\subsection{Modular finite $W$-algebras} 
\label{ss:modularWalgebras}
Restricted versions of modular finite $W$-algebra were introduced by Premet in \cite{PrST}. In his subsequent works \cite{PrCQ,PrGold} he made use of the unrestricted versions along with their central reductions, over fields of large positive characteristics. Finally, the theory was mapped out under the standard hypotheses by two of the current authors \cite{GTmod}. Using Lemma~\ref{L:appendix} we are able to streamline the definition slightly, as we now explain.

The assignment $\Psi:(x,y) \mapsto \chi([x,y])$ defines a symplectic form on $\g(-1)$. We pick any isotropic subspace $\l\subseteq \g(-1)$ such that $\chi(\l^{[p]}) = 0$, and denote the annihilator of $\l$ with respect to this form by $\l^\perp$. Note that a Lagrangian choice of $\l$ is always possible, as noted in \cite[\textsection 2.6]{PrST}.

We define
\begin{eqnarray}
\m_\l = \l \oplus \bigoplus_{i<-1} \g(i) & \text{ and } & \n_\l = \l^\perp \oplus \bigoplus_{i<-1} \g(i).
\end{eqnarray}
As noted in \cite[(4.2)]{GTmod} the dimensions of these Lie algebras satisfy
\begin{eqnarray}
\label{e:dimensionmn}
\dim \m_\l + \dim \n_\l = \dim (G\cdot e) = 2d_\chi.
\end{eqnarray}

Thanks to Lemma~\ref{L:appendix}, there is a unipotent algebraic group $N_\l \subseteq G$ such that $\Lie(N_\l) = \n_\l$. Write $\m_{\l, \chi} = \{x - \chi(x) \mid x\in \m_\l\} \subseteq U(\g)$ and consider the {\it generalised Gelfand--Graev module} $Q_\l = U(\g)/U(\g)\m_{\l,\chi}$. The adjoint actions of both $N_\l$ and $\n_\l$ on $U(\g)$ preserve $U(\g)\m_{\l,\chi}$ (see \cite[Lemma 4.1]{GTmod}) and so descend to $Q_\l$. A short calculation shows that $Q_\l^{\n_\l}$ inherits an algebra structure from $U(\g)$ with $Q_\l^{N_\l}$ embedded as a subalgebra \cite[Lemma~4.2]{GTmod}. Furthermore there are natural inclusions of algebras
\begin{eqnarray}
Q_\l^{N_\l} \into Q_\l^{\n_\l} \into \End_\g(Q_\l)^\op.
\end{eqnarray}
To explain the second inclusion we note that, for $q + U(\g)\m_{\l,\chi} \in Q_\l$ and $u + U(\g)\m_{\l,\chi} \in Q_\l^{\n_\l}$, the following right action is well-defined
\begin{eqnarray}
\label{e:rightmodulestructure}
(q + U(\g)\m_{\l,\chi})\cdot (u + U(\g)\m_{\l,\chi}) = qu + U(\g)\m_{\l,\chi}.
\end{eqnarray}

The {\it finite $W$-algebra} and the {\it extended finite $W$-algebra} are the invariant subspaces
\begin{eqnarray}
U(\g,e) := Q_0^{N_0} & \text{ and } &  \hU(\g,e) := Q_0^{\n_0}.
\end{eqnarray}

In previous work \cite{PrCQ, GTmod} this notation was always used for the invariant algebras $Q_\l^{\n_\l}$ and $Q_\l^{N_\l}$ corresponding to a Lagrangian $\l \subseteq \g(-1)$. Thanks to \cite[Proposition~6.1]{GTmod} we have an isomorphism $Q_0^{N_0} \isoto Q_\l^{N_\l}$, and so the algebra $Q_\l^{N_\l}$ does not depend on $\l$. By contrast $Q_\l^{\n_\l}$ {\it does} depend on the choice of isotropic space, and the precise relationship between $Q_\l^{\n_\l}$ and $Q_0^{\n_0}$ is explained by the PBW theorem (Theorem~\ref{T:PBWtheorem}).

The advantage of studying $Q_0^{N_0}$ and $Q_0^{\n_0}$ is that these algebras carry natural actions of $C^e$ by automorphisms, and this additional structure is vital to the present work.

\subsection{Slices and sheets}
\label{ss:extendedslice}
We continue with a fixed isotropic space $\l \subseteq \g(-1)$. 
Let $\m_\l^\perp$ be the set of $\eta \in \g^*$ vanishing on $\m_\l$ and let $\v \subseteq \g$ be any graded complement to $[\g,e]$. Since $\k\Gamma$ is semisimple and $\Gamma \subseteq C^e$ we may (and shall) choose $\v$ to be $\Gamma$-stable. We set $\cv = \kappa(\v) \subseteq \g^*$.

Define the {\it good transverse slice to $G\cdot \chi$} to be $\chi + \cv$. This good slice admits a contracting $\k^\times$-action given by
\begin{eqnarray}\label{e: contracting action}
\mu(t)\cdot \eta := t^2 \gamma_e(t^{-1}) \cdot \eta
\end{eqnarray}
for $t\in\k^\times$, with unique fixed point $\chi$. One important feature is that the dimension of the $G$-orbit of $\mu(t) \cdot \eta$ does not depend on $t \in \k^\times$, and so $\dim (G\cdot \eta) \ge \dim (G\cdot \chi)$ for all $\eta \in \chi + \cv$. It follows that $\chi + \cv$ is transversal to every orbit which it intersects.

\begin{Lemma}
\label{L:extendedslice} 
Let $\D$ be a decomposition class of $\g^{*}$ such that $\chi\in\overline{\D}_\reg$, write $(\g_0,\O_0)\in\Indat(\g)$ for a corresponding induction datum, and set $\z^* := \kappa(\z(\g_0)) \subseteq \g^*$. The following hold:
\begin{enumerate}
\setlength{\itemsep}{4pt}
\item[(i)] The coadjoint action map restricts to an isomorphism
$N_\l \times (\chi + \cv) \isoto \chi + \m_\l^\perp.$

\item[(ii)] $\dim \big((\chi + \cv) \cap \overline{\D}_\reg\big) = \dim \z^*$.

\item[(iii)] Suppose that there exists a parabolic subalgebra $\p$ of $\g$ containing $\g_0$ as a Levi factor such that $e\in\O_0+\Rad(\p)$. Then $\chi + \z^* \subseteq \overline{\D}_\reg$.
\end{enumerate}
\end{Lemma}
\begin{proof}
Part (i) is \cite[Lemma~5.1]{GTmod}. Part (ii) follows from Lemma~\ref{T:sheetstheorem}(ii) along with the fact that $\chi + \cv$ is transversal to $G\cdot \chi$ at $\chi$.

We now prove part (iii). The assumption $e\in \O_0+ \Rad(\p)$ implies that $e+ \z(\g_0)\subseteq e_0+ \Rad(\p) + \z(\g_0)\subseteq \overline{\D(\g_0,e_0)}$ for some $e_0\in\O_0$, using the description of $\overline{\D(\g_0,e_0)}$ from \cite[Theorem~2.3]{PS}. Applying $\kappa$ gives $\chi + \z^* \subseteq \overline{\D}$. Using part (i) and the transversality of $\chi + \cv$ again, it follows that the dimension of the $G$-orbit of $\eta \in \chi + \z^*$ cannot be strictly less that $\dim G\cdot \chi$. We conclude that $\chi + \z^* \subseteq \overline{\D}_\reg$ which completes the proof.
\end{proof}

\subsection{The $p$-centre and the PBW theorem for finite $W$-algebras}
\label{ss:pcentreandPBWforfiniteW}
The {\it $p$-centre $\hZ_p(\g,e)$ of the extended finite $W$-algebra} $\hU(\g,e)$ is the image of the natural map $Z_p(\g) \to \hU(\g,e)$. Applying the proof of \cite[Lemma~7.6]{GTmod} verbatim, we see that the kernel of this map is generated by $\{x^p - x^{[p]} - \chi(x)^p \mid x\in \m_0\}$, and so the identification $Z_p(\g) = \k[(\g^*)^{(1)}]$ descends to an identification $\hZ_p(\g,e) = \k[(\chi + \m_0^\perp)^{(1)}]$. 

Both $N_0$ and $C^e$ act on $\hU(\g,e)$ by automorphisms preserving $\hZ_p(\g,e)$, and the identification $\hZ_p(\g,e) = \k[(\chi + \m_0^\perp)^{(1)}]$ is $N_0$-equivariant and $C^e$-equivariant. Furthermore, $C^e$ normalises $N_0$ and thus also acts on $U(\g,e)=\hU(\g,e)^{N_0}$ by automorphisms. The {\it $p$-centre $Z_p(\g,e)$ of the finite $W$-algebra} $U(\g,e)$ is the invariant subalgebra $Z_p(\g,e) = \hZ_p(\g,e)^{N_0}$ and by Lemma~\ref{L:extendedslice}(i) we identify $Z_p(\g,e) = \k[(\chi + \m_0^\perp)^{(1)}]^{N_0}$ with $\k[(\chi + \cv)^{(1)}]$; this identification is also $C^e$-equivariant. Observe furthermore that both $\hU(\g,e)$ and $U(\g,e)$ are finite over their $p$-centres.

If $V$ is a $\hU(\g,e)$-module annihilated by the maximal ideal of $\hZ_p(\g,e)$ corresponding to $\eta \in \chi + \m_0^\perp$ then we say {\it $V$ has $p$-character $\eta$}, and use similar terminology for $U(\g,e)$-modules. For an arbitrary isotropic space $\l \subseteq \g(-1)$, the $p$-centres $Z_p(Q_\l^{N_\l})$ and $Z_p(Q_\l^{\n_\l})$ are defined analogously to $Z_p(\g,e)$ and $\hZ_p(\g,e)$, and we have identifications $Z_p(Q_\l^{N_\l}) = \k[(\chi + \m_\l^\perp)^{(1)}]^{N_\l} = \k[(\chi + \cv)^{(1)}]$ and $Z_p(Q_\l^{\n_\l}) = \k[(\chi + \m_\l^\perp)^{(1)}]$.

The Kazhdan filtration $U(\g) = \bigcup_{i\in \Z} \Fb_iU(\g)$ is defined by placing $\g(i)$ in degree $i+2$. The associated graded algebra $\gr U(\g)$ is isomorphic to $S(\g) = \k[\g^*]$ (see \cite[\textsection 5.2]{GTmod} for example). 
The filtration induced on $Q_\l$ is concentrated in non-negative degrees, with $\Fb_0 Q_\l = \k$, and $\gr Q_\l$ identifies with $\k[\chi + \m_\l^\perp]$. This leads to Kazhdan filtrations on both $U(\g,e)$ and $\hU(\g,e)$ with $C^e$ acting by filtered automorphisms.

The following is the PBW theorem for the (extended) finite $W$-algebra with an arbitrary choice of isotropic space.
\begin{Theorem}
\label{T:PBWtheorem}
Let $\l \subseteq \g(-1)$ be isotropic. 
\begin{enumerate}
\setlength{\itemsep}{4pt}
\item[(i)] There are natural isomorphisms
\begin{eqnarray}
\label{e:PBW}
& & \gr (Q_\l^{N_\l}) \isoto \k[\chi + \m_\l^\perp]^{N_\l} \isoto \k[\chi + \cv];\\ \label{e:extendedPBW}
& & \gr (Q_\l^{\n_\l}) \isoto \k[\chi + \m_\l^\perp]^{\n_\l}\isoto \k[\chi + \m_\l^\perp]^{N_\l} \otimes_{\k[(\chi+\m_\l^\perp)^{(1)}]^{N_\l}} \k[(\chi+\m_\l^\perp)^{(1)}].
\end{eqnarray}
\item[(ii)] $Q_\l^{\n_\l} \cong Q_\l^{N_\l} \otimes_{Z_p(Q_\l^{N_\l})} Z_p(Q_\l^{\n_\l})$.
\item[(iii)] The map $Q_0 \onto Q_\l$ induces
\begin{eqnarray}
\label{e:fwainvariantofl}
& & U(\g,e) \isoto  Q_\l^{N_\l};\\
\label{e:extendedfwadependsonl}
& & \hU(\g,e) \longtwoheadrightarrow Q_\l^{\n_\l}.
\end{eqnarray}
In particular, $Q_\l^{N_\l}$ is independent of $\l$ up to isomorphism.
\item[(iv)] There is a $\Gamma$-equivariant embedding $\Theta : \g^e \into U(\g,e)$ such that $\Theta(\g^e)$ generates $U(\g,e)$.
\end{enumerate}
\end{Theorem}
\begin{proof}
The isomorphism \eqref{e:PBW} is \cite[Theorem~5.2]{GTmod} and Lemma~\ref{L:extendedslice}(i).

The first isomorphism in \eqref{e:extendedPBW} follows from the more precise claim that the inclusion $\gr (Q_\l^{\n_\l}) \subseteq (\gr Q_\l)^{\n_\l}$ is an equality. The second isomorphism in \eqref{e:extendedPBW} follows from the claim that $\k[\chi+\m_\l]^{\n_\l}$ is generated by the subalgebras $\k[\chi + \m_\l^\perp]^{N_\l}$ and $\k[(\chi+\m_\l^\perp)^{(1)}] = \k[\chi+\m_\l^\perp]^p$, and that the natural map $\k[\chi + \m_\l^\perp]^{N_\l} \otimes_{\k[(\chi+\m_\l^\perp)^{(1)}]^{N_\l}} \k[(\chi+\m_\l^\perp)^{(1)}] \to \k[\chi + \m_\l^\perp]^{\n_\l}$ is an isomorphism. Both of these claims follow from \cite[\textsection 5.4]{GTmod} (see also \cite[Lemma~2.1]{GTrest}).

To prove (ii) from (i) we only need to show that both $\k[\chi + \m_\l^\perp]^{N_\l}$ and $\k[(\chi + \m_\l^\perp)^{(1)}]$ can be lifted to filtered subalgebras of $Q_\l^{\n_\l}$. Thanks to \eqref{e:PBW} these algebras are lifted by $Q_\l^{N_\l}$ and $Z_p(Q_\l^{\n_\l})$ respectively.

The isomorphism \eqref{e:fwainvariantofl} is explained in \cite[\textsection 6]{GTmod}, whilst  the surjection \eqref{e:extendedfwadependsonl} follows directly from (ii) along with the identification $Z_p(Q_\l^{\n_\l}) = \k[(\chi + \m_\l^\perp)^{(1)}]$.

It remains to prove (iv). Recall that $\chi + \cv$ is $\Gamma$-stable by construction. The identification $\gr(Q_0^{N_0}) \cong \k[\chi + \cv]$ is $\Gamma$-equivariant and the invariant bilinear form $\kappa$ allows us to identify $\k[\chi + \cv] = S(\g^e)$. In summary, $\gr U(\g,e) \cong S(\g^e)$ as algebras with $\Gamma$-action. Writing $S(\g^e)_i$ for the $i$th graded piece with respect to the Kazhdan filtration, we have for each $i \ge 0$ an exact sequence $\Fb_{i-1} U(\g,e) \into  \Fb_i U(\g,e) \onto S(\g^e)_i$ of $\Gamma$-modules. Using the semisimplicity of $\k\Gamma$ once more we see that these sequences all split. It follows that the embedding $\g^e \into S(\g^e)$ may be lifted to an embedding $\g^e \into U(\g,e)$ of $\Gamma$-modules such that the image generates $U(\g,e)$. This completes the proof of (iv).
\end{proof}

\subsection{The Poisson structure on the slice}
\label{ss:PoissonStructureandRank}
Some of our arguments rely on some elementary Poisson geometry and we gather together the necessary details here. If $X$ is any smooth affine Poisson variety then the Hamiltonian vector fields give rise to a space of tangent vectors at each $x\in X$. The {\it rank} of $X$ at $x$ is defined to be the dimension of this space. 

The identification $\k[(\chi + \m_0^\perp)/\!/N_0] = \k[\chi + \cv]$ from Lemma~\ref{L:extendedslice}(i) allows us to equip $\k[\chi + \cv]$ with a Poisson bracket by Poisson reduction, as follows. The preimage of $\k[\chi + \m_0^\perp]^{N_0}$ under $\k[\g^*] \to \k[\chi + \m_0^\perp]$ is contained in the Poisson idealiser of the kernel. Therefore if we write $J_0 = (\m_{0,\chi}) \subseteq \k[\g^*]$ for the defining ideal of $\k[\chi + \m_0^\perp]$ then for $f_1 + J_0, f_2 + J_0 \in \k[\chi+ \m_0^\perp]^{N_0}$ the bracket
$$\{f_1 + J_0, f_2 + J_0\} = \{f_1, f_2\} + J_0 \in \k[\chi + \m_0^\perp]^{N_0}.$$
is well-defined. Furthermore \eqref{e:PBW} implies that $U(\g,e)$ is a filtered quantization of this Poisson structure.

Let $f \in \k[\g^*]$ be in the preimage of $\k[\chi + \m_0^\perp]^{N_0}$ and $\eta \in \chi + \m_0^\perp$, so that $N_0\cdot \eta \in (\chi + \m_0^\perp)/\!/N_0$. Then the corresponding Hamiltonian tangent vector at $N_0\cdot \eta$ lies in $T_{N_0\cdot \eta}\big(G\cdot \eta \cap (\chi + \m_0^\perp)/\!/N_0\big)$. Using Lemma~\ref{L:extendedslice}(i) and the transversality of $\chi + \cv$ it follows that
\begin{eqnarray}
\label{e:rankequation}
\rank_\eta (\chi + \cv) \le \dim G \cdot \eta - \dim G\cdot \chi.
\end{eqnarray}
In fact, one can prove equality but we wont need this in what follows. 

\subsection{Reduced finite $W$-algebras}
\label{ss:reducedfiniteW}

We identify $\Spec \hZ_p(\g,e)$ with $(\chi + \m_0^\perp)^{(1)}$. On the level of sets of closed points we can also identify $(\chi + \m_0^\perp)^{(1)} = \chi + \m_0^\perp$, and, since this will cause no confusion, we allow ourselves to do this henceforth.
If $X \subseteq \chi + \m_0^\perp$ is a closed subscheme with defining ideal $I \subseteq \hZ_p(\g,e)$, then we write $\hU_X(\g,e)$ for the $p$-central reduction $\hU(\g,e) / I\hU(\g,e)$. If $X = \{\eta\}$ is a point we simply write $\hU_\eta(\g,e)$ for this algebra and call it a {\it reduced finite $W$-algebra with $p$-character $\eta$}. 

We use similar notation $U_X(\g,e)$ for a $p$-central reduction of $U(\g,e)$ whenever $X \subseteq \chi + \m_0^\perp/\!/ N_0 = \Spec \k[\chi + \m_0^\perp]^{N_0}$ is a closed subvariety. If $\eta \in \chi + \m_0^\perp$ then, writing $\pi : \chi + \m_0^\perp\to \chi + \m_0^\perp/\!/ N_0$ for the projection, Theorem~\ref{T:PBWtheorem}(ii) states that the map $U(\g,e) \to \hU(\g,e)$ induces an isomorphism
\begin{eqnarray}
\label{e:relatingreductions}
U_{\pi(\eta)}(\g,e) \isoto \hU_\eta(\g,e).
\end{eqnarray}
The significance of these algebras comes from {\it Premet's equivalence}, which states that for every $\eta \in \chi + \m_0^\perp$ we have
\begin{eqnarray}
\label{e:Premetsequivalence}
U_\eta(\g) \isoto \Mat_{p^{d_\chi}} \hU_\eta(\g,e).
\end{eqnarray}
\begin{Remark}
Although Premet's equivalence was originally proven for $\eta \in \chi + \m_\l^\perp$ where $\l$ is Lagrangian \cite[\textsection 2]{PrST}, this slightly more general version can be deduced by observing that for any $\eta\in \chi + \m_0^\perp$ there exists some Lagrangian $\l$ with $\eta \in \chi + \m_\l^\perp$.
\end{Remark}

\subsection{Abelian quotients of $W$-algebras}
\label{ss:abelianquotients}
Recall that $\Gamma\subseteq C^e$ is a finite subgroup such that $p$ does not divide $|\Gamma|$. The maximal abelian quotient $\hU(\g,e)^\ab$ of $\hU(\g,e)$ is the quotient algebra obtained by factoring out the ideal $([u,v] \mid u,v\in \hU(\g,e))$ (cf. Section~\ref{ss:basicsonabquotients}). The action of $C^e$ on $\hU(\g,e)$ descends to an action on $\hU(\g,e)^\ab$ and we consider the coinvariant algebra $\hU(\g,e)^\ab_\Gamma$ obtained from $\hU(\g,e)^\ab$ by factoring out the ideal $(u - \gamma \cdot u \mid \gamma\in \Gamma, u \in \hU(\g,e)^\ab)$.
Similarly we consider the abelian quotient $U(\g,e)^\ab$ and the coinvariant algebra $U(\g,e)^\ab_\Gamma$. If $X\subseteq \chi + \m_0^\perp$ and $Y \subseteq (\chi + \m_0^\perp)/\!/N_0$ are closed subvarieties then we use the notation
\begin{eqnarray*}
\hE_X(\g,e) := \Spec \hU_X(\g,e)^\ab & \text{ and } & \E_Y(\g,e) := \Spec U_Y(\g,e)^\ab
\end{eqnarray*}
and make the abbreviations $\hE(\g,e) = \hE_{\chi + \m_0^\perp}(\g,e)$ and $\E(\g,e) = \E_{\chi + \m_0^\perp/\!/N_0}(\g,e)$. As we observed in Section~\ref{ss:basicsonabquotients} the points of $\hE_X(\g,e)$ parametrise the one dimensional $\hU(\g,e)$-modules admitting a $p$-character in $X$, and similar for $\E_Y(\g,e)$. By the same token, the points of $\Spec \hU(\g,e)^\ab_\Gamma$ and $\Spec U(\g,e)^\ab_\Gamma$ parametrise the $\Gamma$-invariant one dimensional $\hU(\g,e)$-modules and $U(\g,e)$-modules, respectively. As reduced schemes we have equality
\begin{eqnarray*}
\hE(\g,e)^\Gamma = \Spec \hU(\g,e)^\ab_\Gamma & \text{ and } & \E(\g,e)^\Gamma = \Spec U(\g,e)^\ab_\Gamma.
\end{eqnarray*}

\begin{Proposition}
\label{L:generators}
Let $\c \subseteq \g^e$ be a $\Gamma$-stable complement to $[\g^e, \g^e]$.
\begin{enumerate}
\setlength{\itemsep}{4pt}
\item[(i)] There is a $\Gamma$-equivariant surjective homomorphism $S(\c)  \onto U(\g,e)^\ab$;
\item[(ii)] There is a surjective homomorphism $S(\c^\Gamma) \twoheadrightarrow U(\g,e)^\ab_\Gamma$.
\end{enumerate}
\end{Proposition}
\begin{proof}
By Theorem~\ref{T:PBWtheorem}(iv), we see that $U(\g,e)^\ab$ is a commutative algebra with an action of $\Gamma$ by automorphisms, generated by a $\Gamma$-submodule $\Theta(\g^e) \cong \g^e$. By the universal property of $S(\g^e)$ we have $\Gamma$-equivariant mapping $S(\g^e) \onto U(\g,e)^\ab$. We claim that the composition $S(\c) \into S(\g^e) \onto U(\g,e)^\ab$ is surjective. To see this we note that the ordered monomials in $\Theta(\c)$ together with $([u,v] \mid u,v\in U(\g,e))$ span $U(\g,e)$. This last claim can be proven using the same argument as \cite[Lemma~16]{PT1}, in view of  \cite[Remark~7.5]{GTmod}.

Observe that $\c_{-1} := \{c - \gamma\cdot c \mid c\in \c, \gamma\in \Gamma\}$ is a $\Gamma$-submodule of $\c$ which contains nonzero elements of every submodule which carries non-trivial $\Gamma$-action. Since $p\nmid \left\vert\Gamma\right\vert$, we can conclude that $\c_{-1}$ is a $\Gamma$-stable complement to $\c^\Gamma$ in $\c$, contained in $(f - \gamma\cdot f \mid \gamma \in \Gamma, f \in S(\c))$. Taking $\Gamma$-coinvariants is a right exact functor and so the map $S(\c) \onto U(\g,e)^\ab$ from (i) descends to a map $S(\c^\Gamma) \cong S(\c)/(\c_{-1}) \onto U(\g,e)^\ab_\Gamma$ as required.
\end{proof}

We proceed to provide lower bounds on the Krull dimensions of $U(\g,e)^\ab$ and $U(\g,e)^\ab_\Gamma$ in terms of sheets, similar to \cite[Theorem~1.2]{PrCQ}. Let $\S_1,...,\S_t$ be the sheets of $\g^*$ containing $\chi$, and make the notation $\{(\g_i, \O_i) \mid i=1,...,t\}$ for the corresponding ($G$-orbits of) rigid induction data (cf. Lemma~\ref{T:sheetstheorem}). Consider the varieties
\begin{eqnarray}
\label{e:Katsylovar}
\chi + Y := (\chi + \m_0^\perp) \cap \bigcup_{i=1}^t \S_i &\text{ and }& \chi + X = (\chi + \cv) \cap \bigcup_{i=1}^t \S_i.
\end{eqnarray}
Let $\pi : \chi + \m_0^\perp \to (\chi + \m_0^\perp)/\!/N_0$ be the quotient map. Using Lemma~\ref{L:extendedslice}(i) we have
\begin{eqnarray}
\label{e:XandY}
\pi(\chi + Y) = \chi + X.
\end{eqnarray}
\begin{Lemma}
\label{L:contrupsocleputtputt}
    The following hold:
    \begin{enumerate}
        \item $\chi + Y = \{\eta \in \chi + \m_0^\perp \mid \dim  G\cdot \eta \le \dim G\cdot \chi \}$.
        \item $\chi + Y$ is a closed subvariety of $\chi + \m_0^\perp$.
        \item $\chi + X$ is a closed subvariety of $\chi + \cv$.
    \end{enumerate}
\end{Lemma}
\begin{proof}
    Certainly the orbits $G\cdot \eta$ with $\eta \in \chi + Y$ have dimension $\le \dim G\cdot \chi$. In view of Lemma~\ref{T:sheetstheorem}(iii) the reverse inclusion follows from the general assertion: a decomposition class intersects $\chi + \m_0^\perp$ if and only if its closure contains $\chi$ (see \cite[Proposition~7.4]{ToMT}, for example). This proves (1).

    The collection of $\eta \in \g^*$ such that $\dim G\cdot \eta \le \dim G\cdot \chi$ is closed in $\g^*$, and so its intersection with $\chi + \m_0^\perp$ is closed in $\chi + \m_0^\perp$, whence (2). Now (3) follows from (2) and Lemma~\ref{L:extendedslice}(i).
\end{proof}

\begin{Proposition} (Cf. \cite[\textsection 3]{PrCQ})
\label{L:dimEge}

The homomorphism $Z_p(\g,e) \to U(\g,e)^\ab$ induces a finite, surjective, $\Gamma$-equivariant morphism $\E(\g,e) \to \chi + X$ hence
$$\dim \E(\g,e) = \dim(\chi + X) = \max_{1\le i\le t} \dim \z(\g_i).$$
\end{Proposition}
\begin{proof}
Write $\phi : Z_p(\g,e) \to U(\g,e)^\ab$ for the $\Gamma$-equivariant homomorphism. Since finite morphisms of algebraic varieties are closed, the image of the map $\phi^* : \E(\g,e) \to \Spec Z_p(\g,e) = (\chi + \m_0^\perp)^{(1)}/\!/N_0$ is a closed subvariety, and the set $\im \phi^*$ consists of those maximal ideals of $Z_p(\g,e)$ for which the corresponding quotient admits a one dimensional representation (cf. Lemma~\ref{L:basiclemma}).

Thanks to \eqref{e:Premetsequivalence}, $\pi^{-1} (\im \phi^*)$ consists of those $\eta\in \chi + \m_0^\perp$ such that $U_\eta(\g)$ admits a module of dimension $p^{d_\chi}$. Applying Lemma~\ref{L:Humphreyscorollary} and Lemma~\ref{L:contrupsocleputtputt}(1) we deduce $\pi^{-1}(\im \phi^*) = \chi + Y$.

Combining the remarks of the previous paragraph with \eqref{e:XandY} we have $\im \phi^* = \chi + X.$ Since $\phi^*$ is finite we have $\dim \E(\g,e) = \dim (\chi + X)$. It follows from Lemma~\ref{L:extendedslice}(ii) that $\dim (\chi + X) = \max_i \dim \z(\g_i)$, which completes the proof of the current proposition.
\end{proof}

Let $A$ be an affine Poisson algebra. If we choose a system of generators $a_1,...,a_m$ of $A$ and form the matrix $\pi_A = (\{a_i, a_j\})_{1\le i,j\le m}$ then the rank of $\Spec(A)$ at a point $x$ coincides with the rank of $\pi_A$ modulo the maximal ideal of $x$ (see \cite[1.2.3]{LPV} for example). Write $A^\ab$ for the {\it maximal Poisson abelian quotient} which is obtained by factoring out the ideal $(\{f_1, f_2\} \mid f_1, f_2\in A)$ generated by Poisson brackets. By the Leibniz rule, this ideal is generated by the entries of $\pi_A$. It follows that $\Spec (A^\ab)$ can be characterised as the points of $\Spec(A)$ where the rank is zero, in the sense of Section~\ref{ss:PoissonStructureandRank}.

\begin{Proposition}
\label{P:semiclassical}
There is an embedding of algebraic varieties $(\chi + X)^\Gamma \into (\c^\Gamma)^*$.
\end{Proposition}
\begin{proof}
Combining \eqref{e:rankequation} with our remarks prior to Proposition~\ref{L:dimEge}, we see that the rank of the Poisson structure at a point $\eta \in \chi + \cv$ is less than or equal to $r(\eta):= \dim G\cdot \eta - \dim G\cdot \chi$. By Lemma~\ref{L:contrupsocleputtputt} we see that $\chi + X$ consists of points such that $r(\eta) = 0$. Therefore $\chi + X \subseteq \Spec (\k[\chi+\cv]^\ab)$. By Lemma~\ref{L:contrupsocleputtputt}(3) this is a closed embedding of affine varieties, and it follows that $\k[\chi + \cv]^\ab \onto \k[\chi + X]$.

Applying an argument analogous to part (i) of Proposition~\ref{L:generators} we see that there is a $\Gamma$-equivariant homomorphism $S(\c) \onto \k[\chi + \cv]^\ab$.

Combining these remarks and taking $\Gamma$-coinvariants we have a surjective homomorphism $S(\c^\Gamma) \onto \k[(\chi + X)^\Gamma]$ which proves the claim.
\end{proof}

\section{The parabolic induction functor}
\label{S:parabolicinduction}

We retain the notation and hypotheses of Sections~\ref{S:groupsandLiealgebras} and \ref{S:parabolicsection}.

\subsection{Equivariance of Skryabin's equivalence}
\label{ss:Skryabinsequivalence}
Throughout this section, $\l \subseteq \g(-1)$ will denote a Lagrangian subspace. Skryabin's equivalence is a Morita equivalence between $Q_\l^{\n_\l}$ and a certain central quotient of $U(\g)$. It can be viewed as a global version of \eqref{e:Premetsequivalence}. Using Theorem~\ref{T:PBWtheorem}(ii) we obtain similar equivalences when the isotropic subspace of $\g(-1)$ is zero.
\begin{Proposition}
\label{P:Skryabin}
There is a $Z_p(\g)$-equivariant isomorphism $U_{\chi + \m_\l^\perp}(\g) \isoto \Mat_{p^{d_\chi}} \hU_{\chi + \m_\l^\perp}(\g,e).$
\end{Proposition}
\begin{proof}
This follows straight from Theorem~\ref{T:PBWtheorem}(ii) along with \cite[Theorem~9.2(iv)]{GTmod}; note that the proof there does not require $\l$ to be compatible with $T$ in the sense of \textsection 2.3 of {\it op. cit.}, and it works verbatim in the current situation.
\end{proof}

Combining \eqref{e:rightmodulestructure} and \eqref{e:extendedfwadependsonl} we see that $Q_\l$ is a $U(\g)$-$\hU(\g,e)$-bimodule. This bimodule gives rise to {\it Skryabin's functor}
\begin{eqnarray*}
s_\l : \hU(\g,e)\lmod & \longrightarrow & U(\g)\lmod ; \\
V & \longmapsto & Q_\l \otimes_{\hU(\g,e)} V.
\end{eqnarray*}
In fact one can check that the category equivalence which results from Proposition~\ref{P:Skryabin} is realised by lifting $\hU_{\chi+\m_\l^\perp}(\g,e)$-modules to $\hU(\g,e)$-modules and then applying $s_\l$. 

The reductive group $C^e$ acts on the algebras $U(\g)$ and $\hU(\g,e)$ and thus, as in Section~\ref{ss:twists}, acts on the categories $U(\g)\lmod$ and $\hU(\g,e)\lmod$ by twisting.

\begin{Theorem}
\label{T:equivariantSkryabin}
The following hold:
\begin{enumerate}
	\item[(i)]  $\dim s_\l(V) = p^{d_{\chi}} \dim V$ for all finite-dimensional $V\in\hU_{\chi+\m_\l^\perp}(\g,e)\lmod$.
	\item[(ii)] Let $\eta\in\chi+\n_0^\perp$ and let $V$ be a simple $\hU_\eta(\g,e)$-module. Then ${}^gs_\l(V) \cong s_{\l}({}^gV)$ for all $g\in C^e$.
\end{enumerate}
\end{Theorem}

\begin{Remark}
	Suppose that in Theorem~\ref{T:equivariantSkryabin}(ii) we had only assumed $\eta\in\chi+\m_\l^\perp$. If there were to exist $g\in C^e$ such that $\eta\notin \chi+\m_{g^{-1}\cdot\l}$ (and thus $g\eta\notin \chi+\m_\l^\perp$) then we would have $s_\l(V)\neq 0$ (and so ${}^{g}s_\l(V)\neq 0$) but $s_\l({}^{g}V)=0$. Assuming  $\eta\in\chi+\n_0^\perp$ avoids this problem.
\end{Remark}

The dimension claim follows from the fact that $s_\l$ induces the equivalence of Proposition~\ref{P:Skryabin}. The proof of the second claim will occupy the rest of the current section. Our method is to show that $s_\l$ occurs as a quotient of a $C^e$-equivariant functor $s_0$ and that $s_0$ splits into a direct sum of copies of $s_\l$ when applied to simple modules.

\begin{Lemma}
\label{L:bigGGRequivariance}
The following hold:
\begin{enumerate}
\setlength{\itemsep}{4pt}
\item[(i)] $Q_0$ is a $U(\g)$-$\hU(\g,e)$-bimodule and $Q_0 \onto Q_\l$ is a bimodule homomorphism;
\item[(ii)] The functor $s_0 := Q_0 \otimes_{\hU(\g,e)} (-) : \hU(\g,e)\lmod \to U(\g)\lmod$ is $C^e$-equivariant;
\item[(iii)] For finite-dimensional $V \in \hU(\g,e)\lmod$ we have $\dim s_0(V) = p^{\dim \n_0} \dim V$.
\end{enumerate}
\end{Lemma}
\begin{proof}
It is an easy exercise to see that the right $\hU(\g,e)$-action on $Q_0$ given by $(q + U(\g)\m_{0, \chi})(u + U(\g)\m_{0, \chi}) = qu + U(\g)\m_{0, \chi}$ is well-defined and (i) follows immediately because the right $\hU(\g,e)$-action on $Q_\l$ is defined by the same formula (see \eqref{e:rightmodulestructure}).

We have observed that the $C^e$-action on $U(\g)$ descends to both $Q_0$ and $\hU(\g,e)$. If $g \in C^e$ and $V \in \hU(\g,e)\lmod$ then as in Section~\ref{ss:twists} an isomorphism ${}^g(Q_0 \otimes_{\hU(\g,e)} V) \isoto Q_0 \otimes_{\hU(\g,e)} {}^gV$ is given by $q \otimes v \mapsto (g\cdot q) \otimes v$, which proves (ii). Finally Theorem~\ref{T:PBWtheorem}(i) and (ii) imply that $Q_0$ is a free right $\hU(\g,e)$-module of rank $p^r$ where $r = \dim \g - \dim \m_0 - \dim \g^e$ (cf. \cite[Theorem 9.2(i)]{GTmod}). The latter equals $\dim \n_0$ by \eqref{e:dimensionmn}, which proves part (iii).
\end{proof}

Following \cite{BGK} we define $\kk_\neu$ to be the {\it neutral copy of $\g(-1)$}. It is a symplectic vector space equipped with a linear isomorphism $\g(-1) \to \kk_\neu; x\mapsto x_\neu$. The symplectic form $\Psi_\neu$ is defined $\Psi_\neu(x_\neu, y_\neu) := \Psi(x,y) = \kappa(e, [x,y])$, so that $\g(-1) \to \kk_\neu$ is a symplectomorphism. We abuse notation by writing $x\mapsto x_\neu$ for the map $\g \to \kk_\neu$ obtained by projecting via the grading.

The Weyl algebra  associated to $(\kk_\neu, \Psi_\neu)$ will be denoted $U(\kk_\neu)$. 
We consider the algebra $U(\tg) := U(\g) \otimes U(\kk_\neu)$. Our next objective is to construct a functor $\hU(\g,e)\lmod \to U(\tg)\lmod$. The subspace
$$\tn_{0, \chi} := \{x - x_\neu - \chi(x) \mid x \in \n_0\} \subseteq U(\tg)$$
allows us to define the left $U(\tg)$-module $\tQ:= U(\tg)/ U(\tg) \tn_{0,\chi}$. We identify $U(\g)$ with the subalgebra $U(\g)\otimes 1 \subseteq U(\tg)$, and we note that $U(\tg)$ satisfies the following PBW theorem: there is a standard filtration placing $\g \oplus \kk_\neu$ in degree one and $\gr U(\tg) \cong S(\g \oplus \kk_\neu)$. 

\begin{Lemma} 
\label{L:bigGGRiso}
The following hold:
\begin{enumerate}
\item[(i)] The map $U(\g) \to U(\tg)$ gives an embedding $U(\g)\m_{0,\chi} \subseteq U(\tg)\tn_{0, \chi}$.

\item[(ii)] $\tQ$ is a $U(\tg)$-$\hU(\g,e)$-bimodule with right $\hU(\g,e)$-action
\begin{eqnarray}
\label{e:strangeaction}
(q + U(\tg)\tn_{0,\chi})(u + U(\g)\m_{0, \chi}) := qu + U(\tg)\tn_{0,\chi}.
\end{eqnarray}

\item[(iii)] The following map induced by part (i) is an isomorphism of $U(\g)$-$\hU(\g,e)$-bimodules
\begin{eqnarray}
\label{e:bigGGRiso}
Q_0 \isoto \tQ.
\end{eqnarray}

\end{enumerate}
\end{Lemma}
\begin{proof}
The inclusion $U(\g) \into U(\tg)$ restricts to $\m_{0, \chi} \into \tn_{0, \chi}$, which proves (i). Checking that the action \eqref{e:strangeaction} is well-defined is a short calculation which we leave to the reader, and it clearly commutes with the left $U(\tg)$-action.

Finally the map \eqref{e:bigGGRiso} is an isomorphism on the level of vector spaces, using (i) and the PBW theorems for $U(\g)$ and $U(\tg)$. This isomorphism intertwines the right $\hU(\g,e)$-actions, which may be seen by comparing the first line of the proof of Lemma~\ref{L:bigGGRequivariance} with \eqref{e:strangeaction}.
\end{proof}

We define the {\it $p$-centre of $U(\tg)$} to be the central subalgebra $Z_p(\tg) := Z_p(\g) \otimes Z(\kk_\neu)$. Using an argument almost identical to \cite[Lemma~7.6]{GTmod} we see that 
\begin{eqnarray}
\label{e:annihilatorQgtilde}
\Ann_{Z_p(\tg)}(\tQ) = (x^p - x^{[p]} - x_\neu^p - \chi(x)^p \mid x \in \n_0).
\end{eqnarray}

\begin{Proposition}
\label{P:splitfunctor}
Let $V$ be a simple $\hU(\g,e)$-module with $p$-character $\eta \in \chi + \m_\l^\perp$. Then
$$s_0(V) \cong s_\l(V)^{\oplus p^{\frac{1}{2}\dim \g(-1)}}.$$
\end{Proposition}
\begin{proof}
If $V$ is a $\hU(\g,e)$-module with $p$-character $\eta$ then $\tQ \otimes_{\hU(\g,e)} V$ has $p$-character $\eta$, and in particular, $\{x^p - x^{[p]} - \eta(x)^p \mid x\in \g(-1)\}$ lies in the annihilator. From \eqref{e:annihilatorQgtilde} it follows that $\{x^p - x^{[p]} - x_\neu^p - \chi(x)^p \mid x\in \g(-1)\}$ annihilates $\tQ \otimes_{\hU(\g,e)} V$. Combining, we see that $\{x_\neu^p - (\eta-\chi)(x)^p \mid x\in \g(-1)\}$ also annihilates.

Since $U(\kk_\neu)$ is Azumaya over the $p$-centre, we see that $\tQ \otimes_{\hU(\g,e)} V$ descends to a module over $$U(\g) \otimes U(\kk_\neu)/(x_\neu^p -  (\eta-\chi)(x)^p\mid x\in \g(-1)) \cong U(\g) \otimes \Mat_{p^{\frac{1}{2}\dim \g(-1)}}(\k) \cong \Mat_{p^{\frac{1}{2}\dim \g(-1)}} U(\g).$$
By \eqref{e:bigGGRiso} we conclude that $Q_0 \otimes_{\hU(\g,e)} V \cong W^{\oplus p^{\frac{1}{2}\dim \g(-1)}}$ for some $U(\g)$-module $W$. Using \eqref{e:dimensionmn} we calculate $\dim \n_0 = d_\chi + \frac{1}{2}\dim \g(-1)$. Combining Lemma~\ref{L:bigGGRequivariance}(i) and (iii) with \cite[Theorem~9.2(iii)]{GTmod} we see that $s_\l(V)$ is a (non-zero) simple $U_\eta(\g)$-module, occurring as a quotient of $s_0(V)$ and satisfying $\dim s_0(V) = p^{\frac{1}{2}\dim \g(-1)} \dim s_\l(V)$, and so we must have $W\cong s_\l(V)$.
\end{proof}

\begin{proofofskryabin}
	
Let $g\in C^e$ and let $V$ be a simple $\hU(\g,e)$-module having $p$-character $\eta\in \chi+\n_0^\perp$. In particular, $\eta\in \chi+\m_\l^\perp$ and $\eta\in \chi+ \m_{g\cdot \l}^\perp$ so $g\cdot\eta\in\chi+\m_\l^\perp$. We may therefore apply Proposition~\ref{P:splitfunctor} to get $$ s_\l({}^g V)^{\oplus p^{\frac{1}{2}\dim \g(-1)}}\cong s_0({}^g V)\cong {}^g s_0(V)\cong ({}^g s_\l(V))^{\oplus p^{\frac{1}{2}\dim \g(-1)}}.$$ The claim then follows since $s_\l({}^g V)$ and ${}^g s_\l(V)$ are simple.\hfill \qed

\end{proofofskryabin}

\subsection{The parabolic induction functor}
\label{ss:constructingthefunctor}
In this section we construct a parabolic induction functor from a certain full subcategory of $\hU(\g_0,e_0)\lmod$ to $U(\g,e)\lmod$.

As in Section~\ref{ss:LSinduction}, fix $(\g_0,\O_0)\in\Indat(\g)$ and a parabolic subgroup $P$ of $G$ such that $\g_0$ is a Levi factor of $\p=\Lie(P)$. Denote $\r=\Rad(\p)$, so that $\p=\g_0\oplus \r$, and let $\r_{-}$ be the opposite nilradical so that $\g=\r_{-}\oplus \g_0\oplus \r$. Let $\O=\Ind^\g(\g_0,\O_0)$, and fix $e\in \O\cap(\O_0 + \r)$. In particular, we write $e=e_0+r$ for $e_0\in\O_0$ and $r\in\r$. Let $\gamma_e:\k^\times\to G$ be an associated cocharacter for $e$, and assume (as we may by Corollary~\ref{Cor: gamma in G0}) that $\gamma_e(\k^\times)\subseteq G_0$. Furthermore, let $\gamma_{e_0}:\k^\times\to G_0$ be an associated cocharacter for $e_0$. As usual, we fix $C^e$ to be the centraliser of $\gamma_e(\k^\times)$ in $G^e$, and we also fix $C_0^{e_0}$ to be the centraliser of $\gamma_{e_0}(\k^\times)$ in $G_0^{e_0}$.

The restriction $\kappa_0:= \kappa|_{\g_0} : \g_0 \to \g_0^*$ defines a $G_0$-invariant isomorphism $\g_0 \to \g_0^*$ and we write $\chi = \kappa(e) \in \g^*$, $\chi_0 = \kappa_0(e_0) \in \g_0^*$ and set $d_{\chi_0} = \frac{1}{2}\dim (G_0\cdot\chi_0)$. It follows from our choices that
\begin{equation}
	\label{e:restrictchi}
	\chi|_{\g_0} = \chi_0.
\end{equation}

Consider the two subspaces
\begin{eqnarray}
& & \z^*_0 := \kappa_0\z(\g_0) = \Ann_{\g_0^*}([\g_0, \g_0]) \subseteq \g_0^*;\vspace{4pt}\\
& & \z^* := \kappa\z(\g_0) = \Ann_{\g^*}([\g_0, \g_0] \oplus \r \oplus \r_-) \subseteq \g^*.
\end{eqnarray}
Thanks to \eqref{e:restrictchi} the restriction map $\g^* \to \g_0^*$ induces a vector space isomorphism $\chi + \z^* \to \chi_0 + \z_0^*$, thus we have inclusions
\begin{eqnarray}
U_{\chi_0 + \z_0^*}(\g_0) \hookrightarrow U_{\chi + \z^*}(\p) \hookrightarrow U_{\chi + \z^*}(\g).
\end{eqnarray}

\begin{Lemma}
\label{L:inductiondimension}
$U_{\chi + \z^*}(\g)$ is a free $U_{\chi + \z^*}(\p)$-module of rank $p^{d_\chi - d_{\chi_0}}$.
\end{Lemma}
\begin{proof}
By construction we have $\g = \r_-\oplus\p$. For $\eta \in \chi + \z^*$ we have $\eta|_{\r_-} = \chi|_{\r_-}$ which implies that $U_{\chi + \z^*}(\g) = U_{\chi}(\r_-) \otimes U_{\chi + \z^*}(\p)$ as vector spaces, and in fact as right $U_{\chi + \z^*}(\p)$-modules. Therefore $U_{\chi + \z^*}(\g)$ is a free $U_{\chi + \z^*}(\p)$-module of  rank $p^{\dim \r_-}$ and applying \eqref{e:LScodimension} we have $\dim (\r_-) = \frac{1}{2}(\dim \g - \dim \g_0) = \frac{1}{2}((\dim \g - \dim \g^\chi) - (\dim \g_0 - \dim \g_0^{\chi_0})) = d_\chi - d_{\chi_0}.$
\end{proof}

Since $\chi(\r)=0$ and $\z^*(\r)=0$, it follows that every $U_{\chi_0 + \z_0^*}(\g_0)$-module may be lifted to a $U_{\chi + \z^*}(\p)$-module with $\r$ acting trivially. Thus we have a functor
\begin{eqnarray}
\label{e:classicalinductionfunctor}
\begin{array}{rcl}
U_{\chi_0 + \z_0^*}(\g_0)\lmod &\longrightarrow & U_{\chi + \z^*}(\g)\lmod;\vspace{4pt}\\
V &\longmapsto & U_{\chi + \z^*}(\g) \otimes_{U_{\chi + \z^*}(\p)} V.
\end{array}
\end{eqnarray}

Next we consider the natural homomorphisms $U(\g,e) \hookrightarrow \hU(\g,e) \to \hU_{\chi + \z^*}(\g,e)$ which leads to a pullback functor
\begin{eqnarray}
\label{e:pullbackfunctor}
\hU_{\chi + \z^*}(\g, e)\lmod \to U(\g,e)\lmod.
\end{eqnarray}

Let $\g_0=\bigoplus_i\g_0(i)$ be the grading corresponding to the cocharacter $\gamma_{e_0}$. Making choices as described in Section~\ref{ss:modularWalgebras} we construct the finite $W$-algebras $U(\g_0,e_0)$ and $\hU(\g_0, e_0)$. 
Let $\l \subseteq \g(-1)$ and $\l_0 \subseteq \g_0(-1)$ be choices of Lagrangian subspaces.

Thanks to \eqref{e: z perp -1}, we have
\begin{eqnarray}
	\label{e:zzeroisin}
	\chi_0 + \z_0^* & \subseteq & \chi_0 + \g_0(\!<\!0)^\perp;\\
	\label{e:zisin}
	\chi + \z^* & \subseteq & \chi + \n_0^\perp.
\end{eqnarray}

By Section~\ref{ss:Skryabinsequivalence} we obtain equivalences
\begin{eqnarray}
\label{e:skryabin1}
s_{\l_0}: \hU_{\chi_0 + \z_0^*}(\g_0, e_0)\lmod & \isoto & U_{\chi_0 + \z_0^*}(\g_0)\lmod;\\
\label{e:skryabin2}
s_\l^{-1}:  U_{\chi + \z^*}(\g)\lmod & \isoto & \hU_{\chi + \z^*}(\g, e)\lmod.
\end{eqnarray}
Now assembling \eqref{e:classicalinductionfunctor}, \eqref{e:pullbackfunctor}, \eqref{e:skryabin1}, \eqref{e:skryabin2} we define the {\it parabolic induction functor}
\begin{eqnarray}
\label{e:INDUCTION}
\Ind_{\g_0, e_0}^{\g,e}: \hU_{\chi_0 + \z_0^*}(\g_0, e_0) \lmod \longrightarrow U(\g,e)\lmod.
\end{eqnarray}
Note that $\hU_{\chi_0 + \z_0^*}(\g_0, e_0) \lmod$ is the category of {\it semi-restricted representations} described in the introduction. It follows from Lemma~\ref{L:extendedslice}(i) and \eqref{e:zisin} that every element of $\chi + \z^*$ is $N_0$-conjugate to a unique element of $\chi + \cv$. Composing with the inverse of the restriction isomorphism $\chi + \z^* \to \chi_0 + \z_0^*$ we obtain
\begin{eqnarray}
\label{e:inductiononpcharacters}
\chi_0 + \z_0^*\to\chi + \cv.
\end{eqnarray}

The following is the main result of Section~\ref{S:parabolicinduction}. Part (i) can be deduced from the work we have already done, but parts (ii) and (iii) will require some results from Sections~\ref{ss:classicalparabolicinduction} and \ref{ss:twistingpcharacters}.
\begin{Theorem}
\label{T:parabolicinductiontheorem}
The following hold:
\begin{enumerate}
\setlength{\itemsep}{4pt}
\item[(i)] The parabolic induction functor \eqref{e:INDUCTION} is dimension preserving.
\item[(ii)] On the level of $p$-characters the functor coincides with \eqref{e:inductiononpcharacters}.
\item[(iii)] The restriction $\hE_{\chi_0 + \z_0^*}(\g_0, e_0) \to \E(\g,e)$ is a finite morphism of algebraic varieties, which is $\Gamma$-equivariant if $\Gamma \subseteq C_0^{e_0}\cap C^e$.
\end{enumerate}
\end{Theorem}
\begin{proof}
Part (i) follows from Theorem~\ref{T:equivariantSkryabin} and Lemma~\ref{L:inductiondimension}.

Recall that Skryabin's equivalence respects $p$-characters and that parabolic induction \eqref{e:classicalinductionfunctor} sends a module with $p$-character $\eta|_{\g_0}\in\chi_0+\z_0^*$ to a module with $p$-character $\eta \in \chi + \z^*$. Part (ii) will therefore follow from Proposition~\ref{P:Walgebrapullback}.

The first part of (iii) follows from Propositions~\ref{P:classicalinduction} and  \ref{P:Walgebrapullback}, which occupy the rest of Section~\ref{S:parabolicinduction}. For the second part, note that $\Gamma \subseteq C_0^{e_0} \subseteq G_0$ implies that $\Gamma$ fixes $\chi_0$ and stabilises $\chi_0 + \z_0^*$. Therefore it acts on all of the categories involved in the construction of the functor. Now the claim follows from the final remarks of Section~\ref{ss:twists}, along with Theorem~\ref{T:equivariantSkryabin} (which is applicable by \eqref{e:zzeroisin}, \eqref{e:zisin}, and the fact that \eqref{e:classicalinductionfunctor} sends minimal modules to minimal modules).
\end{proof}
 
\subsection{Classical parabolic induction as a finite morphism} 
\label{ss:classicalparabolicinduction}

In this section we show that \eqref{e:classicalinductionfunctor} is a finite morphism between two representation schemes which parametrise minimal modules. First we need a lemma.

\begin{Lemma}
\label{L:dimensionofE}
\begin{enumerate}
\setlength{\itemsep}{4pt}
\item[(i)] For $\eta \in \chi_0 + \z_0^*$ the algebra $\hU_\eta(\g_0, e_0)$ admits a one dimensional module.
\item[(ii)]  For $\eta \in \chi + \z^*$ the algebra $\hU_\eta(\g, e)$ admits a one dimensional module.
\item[(iii)] $\dim \hE_{\chi_0 + \z_0^*}(\g_0,e_0) = \dim \hE_{\chi + \z^*}(\g,e) = \dim \z(\g_0)$.
\end{enumerate}
\begin{proof}
First we prove (i) and (ii). Combining Lemma~\ref{L:Humphreyscorollary} with \eqref{e:Premetsequivalence} it will suffice to show that $\dim \g_0^\eta = \dim \g_0^{\chi_0}$ for $\eta \in \chi_0 + \z_0^*$ and that $\dim \g^\eta = \dim \g^\chi$ for $\eta \in \chi + \z^*$. The first of these assertions follows from the definition of $\z_0^*$. The second follows from Lemma~\ref{L:extendedslice}(iii), making use of the hypotheses $\chi|_{\r} = 0$ and $\chi\vert_{\g_0}=\chi_0$.

Now we move on to prove (iii). The natural map $Z_p(\g_0) = \k[(\g_0^*)^{(1)}] \to \hU_{\chi_0 + \z_0^*}(\g_0, e_0)^\ab$ is finite because $U(\g_0, e_0)$ is finite over the $p$-centre and it certainly factors through the quotient  $\k[(\g_0^*)^{(1)}] \onto \k[(\chi_0 + \z_0^*)^{(1)}]$. By Lemma~\ref{L:basiclemma} the kernel of the map $\k[(\chi_0 + \z_0^*)^{(1)}] \to \hU_{\chi_0 + \z_0^*}(\g_0, e_0)^\ab$ lies inside the intersection of the maximal ideals $I_\eta \subseteq \k[(\chi_0 + \z_0^*)^{(1)}]$ such that the $p$-character $\eta$ supports a one dimensional $U(\g_0, e_0)$-module. By (i) we deduce that $\k[(\chi_0 + \z_0^*)^{(1)}] \into \hU_{\chi_0 + \z_0^*}(\g_0, e_0)^\ab$ and so there is a finite, dominant morphism $\hE_{\chi_0 + \z_0^*}(\g_0,e_0) \to (\chi_0 + \z_0^*)^{(1)}$, which confirms $\dim \hE_{\chi_0 + \z_0^*}(\g_0,e_0) = \dim \z(\g_0)$. The proof of $\dim \hE_{\chi + \z^*}(\g,e) = \dim \z(\g_0)$ is identical using (ii) in place of (i).
\end{proof}

\end{Lemma}

When $d \in \mathbb{N}$ and $A$ is a $\k$-algebra we write $\Rep_d(A)$ for the set of isomorphism classes of $d$-dimensional representations of $A$. By Lemma~\ref{L:inductiondimension} the functor \eqref{e:classicalinductionfunctor} restricts to a map
\begin{eqnarray}
\label{e:setmap}
\begin{array}{rcl}
\Rep_{p^{d_{\chi_0}}}U_{\chi_0 + \z_0^*}(\g_0) &\longrightarrow & \Rep_{p^{d_\chi}} U_{\chi + \z^*}(\g).
\end{array}
\end{eqnarray}

\begin{Proposition}
\label{P:classicalinduction}
$\Rep_{p^{d_{\chi_0}}} U_{\chi_0 + \z_0^*}(\g_0)$ and $\Rep_{p^{d_\chi}} U_{\chi + \z^*}(\g)$ each carry the structure of an affine algebraic variety, and \eqref{e:setmap} is a finite morphism.
\end{Proposition}
\begin{proof}
For the duration of the proof we use the notation $D_{\chi} = p^{d_\chi}$ and $D_{\chi_0} = p^{d_{\chi_0}}$.
\vspace{2pt}

(i) {\it $\Rep_{D_{\chi_0}} U_{\chi_0 + \z_0^*}(\g_0)$ as a variety}. \ \ 
Choose a basis $x_1,...,x_m$ for $\g_0$ and let $I$ denote the kernel of the resulting map $\k\langle x_1,...,x_m\rangle \onto U_{\chi_0 + \z_0^*}(\g_0)$. Let $X$ denote the affine variety consisting of tuples $(A_1,...,A_m) \in \Mat_{D_{\chi_0}}(\k)^m$ such that $F(A_1,...,A_m) = 0$ for all $F \in I$. The group $H_0 = \GL_{D_{\chi_0}}(\k)$ acts on $X$ by simultaneous conjugation and the points of $X/\!/ H_0 = \Spec \k[X]^{H_0}$ parametrise the semisimple representations of dimension $D_{\chi_0}$(see \cite[Theorem~4.1]{Pro} for example). Since $\z_0^* = \kappa_0\z(\g_0)$ we know that $\g_0^{\chi_0} = \g_0^{\eta}$ for all $\eta \in \chi_0 + \z_0^*$ and it follows from Proposition~\ref{P:Skryabin} that every finite dimensional $U_{\chi_0 + \z_0^*}(\g_0)$-module has dimension divisible by $D_{\chi_0}$. In particular, $\Rep_{D_{\chi_0}} U_{\chi_0 + \z_0^*}(\g_0)$ consists of simple modules, thus we may identify $X/\!/ H_0 = \Rep_{D_{\chi_0}} U_{\chi_0 + \z_0^*}(\g_0)$.
\vspace{2pt}

(ii) {\it $\Rep_{D_\chi} U_{\chi + \z^*}(\g)$ as a variety.} \ \
The description of the variety structure on $\Rep_{D_\chi} U_{\chi + \z^*}(\g)$ is very similar to part (i), however we provide some details for convenience, and to fix notation. Let $x_1,...,x_n$ be a basis for $\g$, for some $n \ge m$. We let $Y$ denote the subvariety of $\Mat_{D_\chi}(\k)^n$ consisting of tuples of matrices satisfying the relations in the kernel of $\k\langle x_1,...,x_n\rangle \onto U_{\chi + \z^*}(\g)$. Let $H = \GL_{D_\chi}(\k)$ act on $Y$ by simultaneous conjugation, and consider $Y/\!/ H$. Thanks to Lemma~\ref{L:extendedslice}(iii) we see that $\dim \g^\chi = \dim \g^\eta$ for all $\eta \in \chi + \z^*$ and thus $\Rep_{D_\chi} U_{\chi + \z^*}(\g)$ consists of simple modules, leading us to identify $\Rep_{D_\chi} U_{\chi + \z^*}(\g) = Y/\!/H$.
\vspace{2pt}

(iii) {\it Parabolic induction is a morphism of varieties.} \ \ 
If $V$ is a fixed representation of $\g_0$ corresponding to a tuple $(A_1,...,A_m) \in X$ then the entries of matrices $(A_1',...,A_n')$ corresponding to the action of $\g$ on $U_{\chi + \z^*}(\g) \otimes_{U_{\chi + \z^*}(\p)} V$ are polynomials in the entries of $(A_1,...,A_m)$. In other words parabolic induction defines a morphism $X \to Y$. Since induction respects isomorphism classes this morphism factors to a morphism $X/\!/H_0 \to Y/\!/H$ as required.
\vspace{2pt}

(iv) {\it The morphism is finite.} \ \ By Proposition~\ref{P:Skryabin} we have $U_{\chi_0 + \z_0^*}(\g_0) \cong \Mat_{D_{\chi_0}} \hU_{\chi_0 + \z_0^*}(\g_0, e_0)$ and $U_{\chi + \z^*}(\g) \cong \Mat_{D_\chi} \hU_{\chi + \z^*}(\g,e)$, leading to isomorphisms of algebraic varieties $\Rep_{D_{\chi_0}} U_{\chi_0 + \z_0^*}(\g_0) \cong \hE_{\chi_0 + \z^*_0}(\g_0, e_0)$ and $\Rep_{D_\chi} U_{\chi + \z^*}(\g) \cong \hE_{\chi + \z^*}(\g,e)$. We abuse notation slightly, replacing $\hU_{\chi_0 + \z_0^*}(\g_0,e_0)^\ab$ and $\hU_{\chi + \z^*}(\g,e)^\ab$ by their reduced quotients. Now the morphism $\Rep_{D_{\chi_0}}U_{\chi_0 + \z_0^*}(\g_0) \to \Rep_{D_\chi} U_{\chi + \z^*}(\g)$ induces an algebra homomorphism $\hU_{\chi + \z^*}(\g,e)^\ab \to \hU_{\chi_0 + \z_0^*}(\g_0,e_0)^\ab$. The $p$-character $\eta \in \chi + \z^*$ of an induced module $U_{\chi + \z^*}(\g) \otimes_{U_{\chi + \z^*}(\p)} V$ is just the inverse image of $\eta$ under the (bijective) restriction map $\chi + \z^* \to \chi_0 + \z_0^*$, so the following diagram commutes
\begin{center}
\begin{tikzpicture}[node distance=3cm, auto]
 \node (A) {$\k[(\chi_0 + \z_0^*)^{(1)}]$};
 \begin{scope}[node distance=1.5cm]
 \node (B) [below of= A] {$\hU_{\chi_0 + \z_0^*}(\g_0, e_0)^\ab$};
 \node (G) [right of=A] {};
\node (H) [right of=B] {};  
 \end{scope}
 \node (C) [right of= G] {$\k[(\chi + \z^*)^{(1)}]$};
  \node (D) [right of= H]{$\hU_{\chi + \z^*}(\g,e)^\ab$};
  \node (F) [right of= C]{$\hZ_p(\g,e)$};

 \draw[right hook->] (A) to node {$ $} (B);
 \draw[->] (A) to node {$\sim$} (C);
 \draw[->] (B) to node {$ $} (D);
 \draw[right hook->] (C) to node {$ $} (D);
  \draw[->>] (F) to node {$ $} (C);
\end{tikzpicture}
\end{center}
The injectivity of the vertical arrows was observed in the proof of Lemma~\ref{L:dimensionofE}. Since $\hU(\g,e)$ is finite over its $p$-centre it follows that $\hU_{\chi + \z^*}(\g,e)^\ab$ is finite over $\k[(\chi + \z^*)^{(1)}]$. We conclude that the composition $\k[(\chi_0 + \z_0^*)^{(1)}] \to \hU_{\chi_0 + \z_0^*}(\g_0, e_0)^\ab\to \hU_{\chi + \z^*}(\g,e)^\ab$ is finite, whence $\hU_{\chi + \z^*}(\g,e)^\ab$ is finite over $\hU_{\chi_0 + \z_0^*}(\g_0, e_0)^\ab$.
\end{proof}

\subsection{The pullback functor as a finite morphism}
\label{ss:twistingpcharacters}
Here we show that the functor \eqref{e:pullbackfunctor} induces a finite morphism on the varieties of one dimensional representations. 
In particular, the homomorphism $U(\g,e) \hookrightarrow \hU(\g,e) \to \hU_{\chi + \z^*}(\g,e)$ induces a map via pullback
\begin{eqnarray}
\label{e:Walginclusions}
\hE_{\chi + \z^*}(\g, e) \to \E(\g,e).
\end{eqnarray}
By Lemma~\ref{L:extendedslice} we have a map
\begin{eqnarray}
\label{e:slicemap}
\chi + \z^* \into \chi + \m_0^\perp \onto (\chi + \m_0^\perp)/\!/N_0 = \chi + \cv.
\end{eqnarray}

\begin{Lemma}
\label{L:finitequotient}
The composition \eqref{e:slicemap} is a finite morphism.
\end{Lemma}
\begin{proof}
We prove the equivalent statement in $\g$: the map $\rho : e + \z(\g_0) \to (e + \m_0^{\perp})/\!/N_0$ is a finite morphism, where we slightly abuse notation by writing $\m_0^{\perp} = \{x\in \g\mid \kappa(\m_0, x)= 0\} \subseteq \g$ for the duration of the proof.
 Equivalently we show that the algebra homomorphism $\rho^* : \k[e + \m_0^\perp]^{N_0} \to \k[e + \z(\g_0)]$ is finite. Observe that both $\k[e + \m_0^\perp]^{N_0}$ and $\k[e + \z(\g_0)]$ are graded via the contracting $\k^\times$-action on $\chi + \m_0^\perp$, and that $\rho^*$ is a graded homomorphism. Let $f_1,...,f_n \in \k[e + \m_0^\perp]^{N_0}$ be a collection of homogeneous generators of $\k[e + \m_0^\perp]^{N_0}$. By definition, their vanishing locus is precisely $\{e\}$.

By Lemma~\ref{L:extendedslice}(i) for $z\in \z(\g_0)$ there is a unique pair $(g, v) \in N_0 \times \v$ such that $g\cdot (e + v) = e + z$. With these choices we have $\rho^*(f_i)(e + z) = f_i(e + v)$ and it follows that the vanishing locus of the ideal $(\rho^*(f_i) \mid i=1,...,n)$ in $e + \z(\g_0)$ is $(N_0 \cdot e) \cap (e + \z(\g_0))$. 

We claim that an element $e + z \in e + \z(\g_0)$ is nilpotent if and only if $z = 0$. To see this, we recall from the set-up in Section~\ref{ss:constructingthefunctor} that $e = e_0 + r$ for $r\in\r$. Let $T \subseteq G_0$ be a maximal torus and choose a system of positive roots for the action of $T$ on $\g$ such that $e_0$ and $\r$ both lie in the subspace spanned by positive roots. Now $e + \z(\g_0)$ lies inside the corresponding Borel subalgebra of $\g$, and it follows that $e$ is the only nilpotent element, as claimed.

Combining the previous two paragraphs we see that the vanishing locus of the ideal $(\rho^*(f_i) \mid i=1,...,n)$ in $e + \z(\g_0)$ is $\{e\}$. Now the ideal $I$ generated by $\{\rho^*(f_i) \mid i=1,...,n\}$ satisfies the hypotheses of Lemma~\ref{L:finitelemma} which completes the proof.
\end{proof}

\begin{Proposition}
\label{P:Walgebrapullback}
The map \eqref{e:Walginclusions} is a 
finite morphism of algebraic varieties. On the level of $p$-characters the map coincides with \eqref{e:slicemap}.
\end{Proposition}
\begin{proof}
The claim about $p$-characters follows from \eqref{e:relatingreductions}. 
We have the following commutative diagram 
\begin{center}
\begin{tikzpicture}[node distance=3cm, auto]
 \node (A) {$\k[(\chi + \m_0^\perp)^{(1)}]^{N_0}$};
 \begin{scope}[node distance=1.7cm]
 \node (B) [below of= A] {$U(\g,e)^\ab$};
 \node (G) [right of=A] {};
\node (H) [right of=B] {};  

 \node (C) [right of= G] {$\k[(\chi + \z^*)^{(1)}]$};
  \node (D) [right of= H]{$\hU_{\chi + \z^*}(\g,e)^\ab$.};
\node (F) [left of = A]{$ $};
  \node (E) [left of = F]{$\hZ_p(\g_0,e_0)$};
   \end{scope}
 \draw[->] (A) to node {$ $} (B);
 \draw[->] (A) to node {$ $} (C);
 \draw[->] (B) to node {$ $} (D);
 \draw[right hook->] (C) to node {$ $} (D);
 \draw[->] (E) to node {$\sim$} (A);
\end{tikzpicture}
\end{center}
Since \eqref{e:slicemap} is finite and $\hU_{\chi + \z^*}(\g,e)$ is finite over the $p$-centre it follows that the bottom horizontal arrow is a finite map of algebras. This completes the proof.
\end{proof}

This completes the proof of Theorem~\ref{T:parabolicinductiontheorem}, showing in particular that the parabolic induction functor \eqref{e:INDUCTION} induces the map 
\begin{eqnarray}
	\label{e:slicemap2}
	\chi_0+\z_0^*\xrightarrow{\sim}\chi + \z^* \into \chi + \m_0^\perp \onto (\chi + \m_0^\perp)/\!/N_0 = \chi + \cv
\end{eqnarray}
on the level of $p$-characters. The finiteness of this map allow us to observe the following corollary.

\begin{Corollary}\label{cor: image of char map}
	Let $\D\subseteq \g^*$ be the decomposition class corresponding to the pair $(\g_0,\chi_0)$. Then the image of the map \eqref{e:slicemap2} is an irreducible component of  $\chi+Z:=(\chi+\cv)\cap\overline{\D}_\reg$ of maximal dimension.
\end{Corollary}

\begin{proof}
	We prefer to work in $\g$, and so prove instead that the map $e_0+\z(\g_0)\to e+\v$ has image an irreducible component of $(e+\v)\cap \overline{\D(\g_0,e_0)}_\reg$ of maximal dimension. As in the proof of Lemma~\ref{L:finitelemma}, there exists $r\in\r$ such that $e=e_0+r$. The map \eqref{e:slicemap2} (viewed in $\g$) is then defined by $$e_0+z\mapsto e+z\mapsto e+z\mapsto e+v,$$ where $e+v$ is the unique element in $e+\v$ such that there exists $n\in N_0$ with $n(e+v)=e+z$. In particular, $e+v=n^{-1}(e+z)\in \overline{\D(\g_0,e_0)}_\reg$ by Lemma \ref{L:extendedslice}(iii) (recalling that $\overline{\D(\g_0,e_0)}_\reg$ is $G$-stable). The map \eqref{e:slicemap2} thus restricts to a finite morphism $$e_0+\z(\g_0)\to (e+\v)\cap \overline{\D}_\reg.$$ Since $\dim(e_0+\z(\g_0))=\dim(\z(\g_0))=\dim((e+\v)\cap \overline{\D}_\reg)$ by Lemma \ref{L:extendedslice}(ii), and since finite morphisms between irreducible varieties of the same dimension are surjective, the image of this map is an irreducible component of $(e+\v)\cap \overline{\D}_\reg$ of maximal dimension. 
\end{proof}

If the composition $\Gamma\hookrightarrow G^e\twoheadrightarrow G^e/(G^e)^\circ$ is surjective (for example, if $\Gamma$ splits $G^e\twoheadrightarrow G^e/(G^e)^\circ$), then we may also prove the following.

\begin{Corollary}\label{cor: full image of char map}
	Maintain the setup of Corollary~\ref{cor: image of char map}. Suppose that $\Gamma\to G^e/(G^e)^\circ$ is surjective and that the adjoint action of $\Gamma$ on $\g$ preserves $e_0$ and $\z(\g_0)$. Then the image of the map \eqref{e:slicemap2} is  $\chi+Z$.
\end{Corollary}
\begin{proof}
	This will follow from two facts: (1) the map \eqref{e:slicemap2} is $\Gamma$-equivariant under these assumptions and (2) $\Gamma$ acts transitively on the irreducible components of $\chi+Z$. For (1), note that the map $e_0+ \z(\g_0)\to e+\z(\g_0)$ is $\Gamma$-equivariant since it coincides with translation by $r$ and $\Gamma$ fixes $r$ as it does $e$ and $e_0$. To see the $\Gamma$-equivariance of the map sending $e+z$, for $z\in\z(\g_0)$, to the unique element $e+v\in e+\v$ such that $n(e+v)=e+z$ for some unique $n\in N_0$, suppose that $\gamma\in\Gamma$ and that $\gamma(e+z)=n'(e+v')$ for unique $v'\in\v$ and $n'\in N_0$. Then $n'(e+v')=\gamma(e+z)=\gamma n\gamma^{-1}(\gamma(e+v))$. Since $\Gamma\subseteq C^e$ normalizes $N_0$, this implies by the uniqueness that $e+v'=\gamma(e+v)$ as required.
	
	For (2), we may apply the same argument as in \cite[Theorem 2.10]{Im} by observing that the required description of $\overline{\D}_\reg$ works in positive characteristic (under the standard hypotheses) by \cite[Theorem 2.3]{PS}, that $(G^e)^\circ\subseteq P$ by Lemma~\ref{L:centraliserinparabolic}(i), and that each irreducible component of $(e+\v)\cap \overline{\D(\g_0,e_0)}_\reg$ contains $e$ due to the fact that each is closed under the contracting action \eqref{e: contracting action}.
\end{proof}

\begin{Remark}
\label{R:glNtheorem}
    Fix $N > 0$. We now explain how the results proven so far can be applied to $\gl_N$, and we prove a version of the main theorem in this case. The discussion here can be seen as a simplified model of the proofs of Theorem~\ref{T:main} given in the subsequent sections. Every $\chi \in \gl_N^*$ lies in a unique sheet and the stabiliser $\GL_N^\chi$ is connected. So the appropriate analogue of the main theorem in type {\sf A} is that every minimal module over a reduced enveloping algebra is parabolically induced. This is what we now demonstrate.
    
    By Section~\ref{ss:NilpotentReduction} we can assume that $\chi(x) = \operatorname{Tr}(ex)$ for some nilpotent $e \in \Nc(\gl_N)$. Suppose that $\lambda \vdash N$ is the partition recording the sizes of Jordan blocks of $e$.
    
    Choose a parabolic subalgebra $\p = \g_0 \oplus \r$ such that $e\in \r$ and $\dim \GL_N \cdot e = 2\dim \r$. Then $\g_0$ is a product of general linear algebras, and the ranks give a partition of $N$ which is dual to $\lambda$. In particular, $\dim \z(\g_0) = \lambda_1$. By \cite[Lemma~3.2]{GTmin} we have $\dim \E(\g,e) = \lambda_1$.

    By Theorem~\ref{T:parabolicinductiontheorem} we have a finite morphism $\Ind_{\g_0, 0}^{\g,e} : \hE_{\chi_0 + \z_0^*}(\g_0, e_0) \to \E(\g,e)$ between irreducible varieties of the same dimension. Since finite morphisms are closed, we conclude that it is surjective. We deduce that every minimal $U_{\chi}(\g)$-module is induced from a one dimensional $\g_0$-module, giving an alternative proof of \cite[Corollary~1.2]{GTmod}.
\end{Remark}

\section{Classical Lie algebras}
\label{S:classicalLiealgebras}

In this section, unless otherwise stated we retain the notation and hypotheses of Sections~\ref{S:groupsandLiealgebras}, \ref{S:parabolicsection}, and \ref{S:parabolicinduction};  however some of this will be specialised and modified to focus on classical Lie algebras. In particular, we adopt the following:

\begin{itemize}
\setlength{\itemsep}{4pt}
\item The characteristic of $\k$ is $p > 2$.
\item $\epsilon = \pm 1$ and $N > 0$ is an integer such that $\epsilon^N = 1$.
\item $\langle -, -\rangle : \k^N \times \k^N \to \k$ is a non-degenerate form such that $\langle u,v\rangle = \epsilon \langle v,u\rangle$ for all $u,v \in \k^N$.
\item $G \subseteq \GL_N$ is the connected component of the group stabilising $\langle - , - \rangle$.
\item $\g = \Lie(G)$ so that $$\g = \left\{ \begin{array}{rl} \so_N & \mbox{ if }\epsilon = 1, \\ \sp_N & \mbox{ if } \epsilon = -1.\end{array} \right.$$
\item $\kappa : \g \times \g\to \k$ is the non-degenerate trace form corresponding to the natural representation.
\end{itemize}

Note that in Sections~\ref{S:groupsandLiealgebras}, \ref{S:parabolicsection} and \ref{S:parabolicinduction} we assumed that $G$ satisfied all of the standard hypotheses from \cite[\textsection 6.3]{JaLA}, including that the derived subgroup of $G$ was simply connected. In this section we omit this assumption, and instead work with classical groups $\SO_N$ and $\Sp_N$. This change in assumptions is justified by the fact that the simply connected covering map is separable by Remark~\ref{R:standardreductivedecomp}(3).

\subsection{Nilpotent orbits and partitions}
\label{ss:centralisers}

The classification of nilpotent $G$-orbits in $\g$ is best stated in terms of partitions. Every nilpotent $\GL_N$-orbit in $\gl_N$ is labelled by some $\lambda = (\lambda_1\ge \cdots \ge \lambda_n) \vdash N$ corresponding to the Jordan normal form, and an orbit $\O \subseteq \Nc(\gl_N)$ intersects $\g$ if and only if:
\begin{enumerate}
\item[(i)] the odd parts of $\lambda$ occur with even multiplicity (when $\epsilon = -1$);
\item[(ii)] the even parts of $\lambda$ occur with even multiplicity (when $\epsilon = 1$).
\end{enumerate}
Furthermore the intersection is a single $G$-orbit unless $\epsilon = 1$ and every part of $\lambda$ is even, in which case the intersection splits into two $G$-orbits, which are labelled I and I\!I \cite[\textsection 5]{CM}. Such orbits are called {\it very even}. We use the notation $\PeN$ for the partitions of $N$ corresponding to nilpotent orbits of $\g$.

For $\lambda = (\lambda_1,...,\lambda_n) \in \PeN$ and $1\le i < n$ we say that $(i, i+1)$ is a {\it 2-step of $\lambda$} if:
\begin{itemize}
\setlength{\itemsep}{4pt}
\item $\lambda_{i-1} \ne \lambda_i \ge \lambda_{i+1} \ne \lambda_{i+2}$; and
\item $\varepsilon(-1)^{\lambda_i} = \varepsilon(-1)^{\lambda_{i+1}} = -1.$
\end{itemize}
Here we use the convention that $\lambda_0=\infty$ and $\lambda_{n+1}=0$. The set of all 2-steps is denoted $\Delta(\lambda)$. When $N$ is even, we say that a partition $\lambda = (\lambda_1,....,\lambda_n) \in \mathcal{P}_1(N)$ is {\it exceptional} if there is a unique 2-step $\{i, i+1\} \subseteq \{1,...,n\}$ and all parts $\lambda_j$ for $j\notin \{i, i+1\}$ are even.

\subsection{The abelianisation of the centraliser of a nilpotent element}
\label{ss:abelianisationcentraliser}
Fix $e \in \Nc(\g)$ with partition $\lambda \in \PeN$. Since $G$ is classical it is well-known, and quite easy to see, that the surjection $C^e \onto C^e/(C^e)^\circ$ splits (see \cite[\textsection 3.8]{JaNO}). Therefore we may (and shall) take $\Gamma \subseteq C^e$ to be a finite subgroup which maps isomorphically to $C^e/(C^e)^\circ$.

Since $G$ is isomorphic to $\SO_N$ or $\Sp_N$ it follows from \cite[\textsection 3.8]{JaNO} that $\Gamma$ is an elementary abelian 2-group. Since $p \ne 2$ we may choose a $\Gamma$-stable complement $\c$ to $[\g^e, \g^e]$ inside $\g^e$. We make the notation
\begin{eqnarray}
& & c(e) := \dim \c; \\
& & c_{\Gamma}(e) := \dim \c^\Gamma; \label{e:def cgamma}\\
& & s(\lambda) := \sum_{i=1}^n \big\lfloor \frac{\lambda_i - \lambda_{i+1}}{2}\big\rfloor.
\end{eqnarray}
\begin{Lemma}
\label{L:celemma}
The following hold:
\begin{eqnarray}
\label{e:cgammae}
\label{e:ce}
& & c(e) = |\Delta(\lambda)| + s(\lambda); \\ \label{e:cGammae}
& & c_\Gamma(e) = \left\{ \begin{array}{cc} s(\lambda) + 1 & \textnormal{if }\lambda \textnormal{ is exceptional,}\\
s(\lambda) & \textnormal{otherwise.}
\end{array}\right.
\end{eqnarray}
Furthermore $e$ lies in a unique sheet if and only if $\lambda_{i-1} - \lambda_i \notin 2\Z_{> 0}$ and $\lambda_{i+1} - \lambda_{i+2} \notin 2\Z_{>0}$ whenever $(i, i+1)$ is a 2-step for the partition $\lambda$ of $e$. In this case the sheet $\S$ satisfies $\dim \S - \dim G\cdot e = c(e)$.
\end{Lemma}
\begin{proof}
The first two statements are proven in Corollary~1 and part (a) of the proof of Theorem~12 in \cite{PT1}. Although the authors worked over $\C$, the arguments there are valid in characteristic $p \ne 2$.

In Corollary~5 of {\it loc. cit.} the claim regarding the uniqueness of the sheet was proven in case $\k = \C$. It follows from the remarks of Section~\ref{ss:LSinduction} and \ref{ss:decompclasses} that the classification of sheets, the formula for their dimensions and the distribution of nilpotent orbits amongst sheets is the same over an algebraically closed field of any characteristic, provided the standard hypotheses hold. Furthermore, as observed in the prologue to this section, this is also true for $G=\SO_N$ when $p>2$.
\end{proof}

\subsection{Constructing the induction datum}
\label{ss:invariantparabolic}
We now introduce an important induction datum which will be used in our proof of part (2) of the main theorem. In order to do so we first make special choices of all of the ingredients $\langle - , - \rangle$, $\g$, $e$ and $\Gamma$. Fix a partition $\lambda = (\lambda_1,...,\lambda_n) \in \PeN$ and assume that $\lambda$ is not exceptional. We also assume, for simplicity, that if $\varepsilon=1$ then not every part of $\lambda$ is even - the advantage of the latter assumption is that we do not need to keep track of labels $\rm I$ and $\rm I\!I$ on the very even orbits and the corresponding Levi subvalgebra (cf. \cite[\textsection~7.3]{CM} and \cite[Remark~3]{PT1}).  

When $k \in \Z_{>0}$ we let $[k]$ denote the set of integers $\{-k-1+2j \mid j=1,...,k\}$.  Consider the vector space $V \cong \k^N$ spanned by elements $\{v_{i,j} \mid 1\le i \le n, \ j\in [\lambda_i]\}$. Let $e \in \gl(V)$ be the nilpotent element defined by $e(v_{i,j}) = v_{i,j+2}$ where we adopt the notation $v_{i,\lambda_i + t} := 0$ for $t > 0$. For $1\le i \le n$ we refer to the space spanned by $\{v_{i,j}\mid j \in [\lambda_i]\}$ as the {\it $i$th Jordan block space for $e$}.

Following the construction of \cite[\textsection\textsection 1.7, 1.8]{JaNO} we equip $V$ with a nondegenerate form $\langle - , - \rangle : V \times V \to \k$ such that $\langle eu, v\rangle = - \langle u, ev\rangle$ and $\langle u,v\rangle = \varepsilon \langle v, u\rangle$ for all $u,v \in V$. We let $\g$ be the Lie subalgebra of $\gl(V)$ consisting of matrices skew self-adjoint with respect to $\langle -,-\rangle$, so that $\g \cong \so_N$ for $\varepsilon = 1$ and $\g \cong \sp_N$ for $\varepsilon = -1$. By construction $e\in \g$. We write $G$ for the connected component of the subgroup of $\GL(V)$ preserving $\langle - , -\rangle$. Note that the trace form $\kappa : \g \times\g\to \k$ associated to $V$ is non-degenerate and $G$-equivariant.

Let
\begin{eqnarray}
\label{L:aspecialsequence}
(i_1,...,i_{s(\lambda)}) := (\underbrace{1,1,...,1,}_{\lfloor\frac{\lambda_1 - \lambda_2}{2}\rfloor \text{ times}} \underbrace{2,...,2,}_{\lfloor\frac{\lambda_2 - \lambda_3}{2}\rfloor \text{ times}}...,\underbrace{n,...,n}_{\lfloor\frac{\lambda_n}{2}\rfloor \text{ times}})
\end{eqnarray}
and for $k=1,...,s(\lambda)$ define vector subspaces $V_k^{+}$ and $V_k^{-}$ of $V$, where $V_k^+$ is spanned by $\{v_{i,\lambda_i + 1 - 2k } \mid 1\le i \le i_k\}$ and $V_k^-$ is spanned by $\{v_{i,-\lambda_i - 1 + 2k } \mid 1\le i \le i_k\}$. Write $V_k^{\pm} = V_k^+ \oplus V_k^-$. Also we let $U$ be the vector subspace spanned by $\{v_{i,j} \mid 1\le i \le n, \ j \in [\lambda_i - 2\sum_{k =i}^n \lfloor \frac{\lambda_k - \lambda_{k+1}}{2}\rfloor] \}$ so that 
\begin{eqnarray}
\label{e:decompofV}
V = U \oplus (\bigoplus_{k=1}^{s(\lambda)} V_k^{\pm}).
\end{eqnarray}
Notice that $\langle -,-\rangle$ restricts to a nondegenerate form on $U$, as well as a non-degenerate pairing $V_k^+ \times V_k^- \to \k$ for all $k$.

Let $l(i) := \lambda_i - 2\sum_{k =i}^n \lfloor \frac{\lambda_k - \lambda_{k+1}}{2}\rfloor-1$. We consider a nilpotent element $e_0$ which annihilates $\bigoplus_k  V_k^\pm$ and acts on $U$ by $e_0v_{i,j} = v_{i, j+2}$ for $j < l(i)$ and $e_0(v_{i,l(i)}) = 0$. By inspection the partition $\lambda^{(0)} \vdash N - 2\sum_{j} i_j$ which records the Jordan block sizes of $e_0|_U$ can be described inductively by
\begin{eqnarray}
\label{e:lambda0}
\begin{array}{rcl}
\lambda^{(0)}_i - \lambda^{(0)}_{i+1} & = & \lambda_i - \lambda_{i+1} - 2\lfloor\frac{\lambda_i - \lambda_{i+1}}{2}\rfloor  \text{ for }  i=1,...,n,\vspace{6pt} \\
\lambda_{n+1}^{(0)} & = & 0.  
\end{array}
\end{eqnarray}

Now we define diagonal semisimple elements $z_1,...,z_{s(\lambda)} \in \g$ by letting $z_k$ act by the identity of $V_k^+$, by minus the identity on $V_k^-$ and by zero on $U$ and $V_i^\pm$ for $i \ne k$. Also let $\g_0$ be the intersection of the centralisers of $z_1,...,z_{s(\lambda)}$ in $\g$; this is a Levi subalgebra of $\g$, since $z_1,...,z_{s(\lambda)}$ all lie in the Lie algebra of a maximal torus of $G$.

We claim now that
\begin{eqnarray}
	\label{e:theLevi}
	\g_0 = \h_{1} \times \cdots \times \h_{{s(\lambda)}} \times \kk
\end{eqnarray}
where $\h_j\cong \gl_{i_j}$ for each $j=1,\ldots,s(\lambda)$ and $\kk$ is isomorphic to a classical Lie algebra of the same type as $\g$ with natural representation of dimension $N - 2\sum_j i_j$. Consider the isotropic flag 
\begin{eqnarray}
	\label{e:iso flag}
	0\subseteq V_1^+\subseteq V_1^+\oplus V_2^+\subseteq \cdots \subseteq \bigoplus_{k=1}^{s(\lambda)}V_{k}^+.
\end{eqnarray}
Its stabiliser in $G$ is a parabolic subgroup $P$ of $G$ with Levi factor $$G_0=H_1\times\cdots \times H_{s(\lambda)}\times K$$ such that $H_j$ is isomorphic to $\GL(V_j^+)$ and $K$ is isomorphic to $\SO(U)$ (if $\varepsilon=1$) or $\Sp(U)$ (if $\varepsilon=-1$). It is clear that $\g_0'=\Lie(G_0)$ centralises $z_1,...,z_{s(\lambda)}$, and thus $\g_0'\subseteq \g_0$. By construction $\{z_1,...,z_{s(\lambda)}\} \subseteq \z(\g_0)$ and so $\dim \z(\g_0) \ge s(\lambda)=\dim \z(\g_0')$. Since $\g_0'\subseteq \g_0$ is an inclusion of Levi subalgebras of $\g$, this implies $\g_0=\g_0'$, as required (setting $\h_j=\Lie(H_j)$ and $\kk=\Lie(K)$).

This argument also implies
\begin{eqnarray}
\label{e:spanningzgzero}
\z(\g_0) = \sspan\{z_1,...,z_{s(\lambda)}\}.
\end{eqnarray}

The following can be deduced from \eqref{e:lambda0} and Lemma~\ref{L:celemma}.
\begin{Lemma}
\label{L:uniquesheet}
$e_0$ lies in a unique sheet of $\g_0$. $\hfill\qed$
\end{Lemma}

We have $e \in e_0 + \r$ where $\r$ is the nilpotent radical of the parabolic $\p=\Lie(P)$ containing $e$ which is determined by the partial isotropic flag \eqref{e:iso flag} as above. Applying an inductive argument using \cite[Corollary~7.3.4]{CM} (which applies in positive characteristic by Section~\ref{ss:LSinduction}) we see that
\begin{eqnarray}
\label{e:inducesup}
G\cdot e = \Ind_{\g_0}^\g(G_0 \cdot e_0).
\end{eqnarray}
Note however that the orbit $G_0\cdot e_0$ is not necessarily rigid. We write $\chi$ for the image of $e\in \g$ under the isomorphism $\g\xrightarrow{\sim}\g^*$ induced by $\kappa$ and $\chi_0$ for the image of $e_0\in \g_0$ under the isomorphism $\g_0\xrightarrow{\sim}\g_0^*$ induced by the restriction of $\kappa$ to $\g_0$.

We let $\gamma_e : \k^\times \to G$ be the cocharacter defined by $\gamma_e(t) v_{i,j} = t^jv_{i,j}$ for all $v_{i,j}\in V$ and let $\gamma_{e_0} : \k^\times \to G_0$ be the cocharacter which acts trivially on $\bigoplus_i (V_i^+ \oplus V_i^-)$ and acts by $\gamma_{e_0}(t) v_{i,j} = t^jv_{i,j}$ for all $v_{i,j} \in U$.
\begin{Lemma}
$\gamma_e$ is an associated cocharacter for $e$ and $\gamma_{e_0}$ is an associated cocharacter for $e_0$ when the latter is viewed as an element of $\g_0$.
\end{Lemma}
\begin{proof}
Since the trivial cocharacter $\k^\times \to \GL_m$ is associated to zero in $\gl_m$, the proof of the lemma is the same for $\gamma_e$ and $\gamma_{e_0}$. We prove the former.

It is straightforward to see that $\{h = d_1 \gamma_e, e\}$ can be completed to an $\sl_2$-triple in $\g$.
The final paragraph of the proof of \cite[Lemma~3.1]{PT2} shows that $\gamma_e$ is optimal for $e$, whilst \cite[Proposition~2.5]{PrKR} implies that optimal cocharacters for $e$ are the same as associated cocharacters. 
\end{proof}

These cocharacters induce gradings $\g = \bigoplus_{i \in \Z} \g(i, \gamma_e)$ and $\g_0 = \bigoplus_{i \in \Z} \g_0(i, \gamma_{e_0})$, and we have
\begin{eqnarray}
\label{e:centresubseteq}
\z(\g_0) \subseteq \g(0, \gamma_e) \cap \g_0(0, \gamma_{e_0})
\end{eqnarray}
because $\gamma_e(\k^\times)$ and $\z(\g_0)$ act diagonally on $V$ (use \eqref{e:spanningzgzero}), and $G_0$ acts trivially on $\z(\g_0)$. Recall from Section~\ref{ss:associatedcochar} that the reductive parts $C^e \subseteq G^e$ and $C_0^{e_0} \subseteq G_0^{e_0}$ are the subgroups stabilising $\gamma_e$ and $\gamma_{e_0}$ respectively. Note that we may view $e_0$ as lying in $\kk$ and $\gamma_{e_0}$ as a map $\k^\times \to K$. We then define $\widehat{C}_{0}^{e_0}\subseteq K^{e_0}$ to be the subgroup stabilising $\gamma_{e_0}$; this coincides with the image of $C_0^{e_0}$ under the projection $G_0\to K$.

Now let $\Lambda := \{1\leq i \le n \mid \varepsilon(-1)^{\lambda_i} = -1, \lambda_i > \lambda_{i+1}\}$ and for $k \in \Lambda$ let $g_k \in \GL(V)$ be the involution which acts by minus the identity on the $k$th Jordan block for $e$ and acts by the identity on the other blocks. Now let $\Gamma'$ be elementary abelian 2-group generated by the commuting involutions $\{g_k \mid k \in \Lambda\}$ and let $\Gamma := \Gamma' \cap G$. It can be deduced from \cite[\textsection\textsection 3.8, 3.13]{JaNO} that the group $\Gamma$ splits the surjection $C^e \onto C^e/ (C^e)^\circ$.

We similarly define $\Lambda_0:=\{1\leq i \le n \mid \varepsilon(-1)^{\lambda^{(0)}_i} = -1, \lambda^{(0)}_i > \lambda^{(0)}_{i+1}\}$ and define by $\Gamma'_0$ the subgroup of $\Gamma'$ generated by $\{g_k\mid k\in\Lambda_0\}$. Set $\Gamma_0=\Gamma'_0\cap G\subseteq \Gamma$.

Note that $\Gamma\subseteq G_0$ because each $g_k$ for $k\in\Lambda$ preserves each $V_{i}^{+}$ and $V_i^{-}$ for $i=1,\ldots,s(\lambda)$. In fact, we have
\begin{eqnarray}
\label{e:Gammasubseteq}
\Gamma \subseteq C^{e} \cap C_0^{e_0}
\end{eqnarray}
because each $g_k$, for $k\in\Lambda$, acts by scalar multiples of the identity on every Jordan block for $e$, whilst $e, e_0, \gamma_e(\k^\times)$ and $\gamma_{e_0}(\k^\times)$ all preserve the blocks. This implies that $\Gamma$ preserves the gradings $\g = \bigoplus_{i \in \Z} \g(i, \gamma_e)$ and $\g_0 = \bigoplus_{i \in \Z} \g_0(i, \gamma_{e_0})$ and thus normalises $N_0$. It furthermore implies that $\Gamma_0\subseteq C_0^{e_0}\subseteq G_0$; letting $\widehat{\Gamma}_0$ be the image of $\Gamma_0$ under the projection from $G_0$ to $K$, we then have that $\widehat{\Gamma}_0$ splits the surjection $\widehat{C}_0^{e_0}\twoheadrightarrow \widehat{C}_0^{e_0}/(\widehat{C}_0^{e_0})^\circ$.

\begin{Proposition}
\label{P:fixedinduction}
We let $\D \subseteq \g^*$ be the decomposition class corresponding to the induction datum $(\g_0, G_0\cdot e_0)$ and write $\chi + Z := (\chi + \cv) \cap \overline{\D}_\reg$.
\begin{enumerate}
\setlength{\itemsep}{4pt}
\item[(i)] $\Gamma$ preserves $\hE_{\chi_0 + \z_0^*}(\g_0, e_0)$ and the parabolic induction functor \eqref{e:INDUCTION} restricts to
\begin{eqnarray}
\label{e:fixedinduction}
\hE_{\chi_0 + \z_0^*}(\g_0, e_0)^\Gamma \to \E(\g,e)^\Gamma.
\end{eqnarray}
\item[(ii)] Every $p$-character in $\chi_0 + \z_0^*$ supports a one dimensional $\Gamma$-fixed $\hU(\g_0, e_0)$-module.
\item[(iii)] Every $p$-character in $\chi + Z$ supports a one dimensional $\Gamma$-fixed $U(\g, e)$-module in the image of \eqref{e:fixedinduction}.
\end{enumerate}
\end{Proposition}
\begin{proof}
First observe that $\Gamma$ preserves $\hU(\g_0, e_0)$ and $\chi_0 + \z_0^*$, which follows from \eqref{e:centresubseteq} and \eqref{e:Gammasubseteq}. This implies the first claim of (i). The parabolic induction functor is the composition
$$\hE_{\chi_0 + \z_0^*}(\g_0,e_0) \isoto \Rep_{p^{d_{\chi_0}}} U_{\chi_0 + \z_0^*}(\g_0) \to \Rep_{p^{d_{\chi}}} U_{\chi + \z^*}(\g) \isoto \hE_{\chi + \z^*}(\g,e) \to \E(\g,e).$$
It follows from \eqref{e:centresubseteq} that $$\z_0^* \subseteq \bigoplus_{i<0} \g_0(i,\gamma_{e_0})^\perp \subseteq \g_0^* \quad\mbox{and}\quad \z^* \subseteq \bigoplus_{i<0} \g(i, \gamma_e)^\perp\subseteq \g^*.$$ Therefore the isomorphisms $\hE_{\chi_0 + \z_0^*}(\g_0,e_0) \isoto \Rep_{p^{d_{\chi_0}}} U_{\chi_0 + \z_0^*}(\g_0)  $ and $\Rep_{p^{d_{\chi}}} U_{\chi + \z^*}(\g) \isoto \hE_{\chi + \z^*}(\g,e)$ are $\Gamma$-equivariant thanks to \eqref{e:Gammasubseteq} and Theorem~\ref{T:equivariantSkryabin}. The morphism $\hE_{\chi + \z^*}(\g,e) \to \E(\g,e)$ is $C^e$-equivariant because the homomorphism $U(\g,e) \to \hU_{\chi + \z^*}(\g,e)$ is. Finally $\Rep_{p^{d_{\chi_0}}} U_{\chi_0 + \z_0^*}(\g_0) \to \Rep_{p^{d_{\chi}}} U_{\chi + \z^*}(\g)$ is $\Gamma$-equivariant, by the final remarks of Section~\ref{ss:twists}. This proves (i).

Using \eqref{e:theLevi} we see that $\hU(\g_0, e_0)^\ab \cong S(\g_0/[\g_0,\g_0]) \times \hU(\kk, e_0)^\ab$ as $\Gamma$-algebras. Since $z_1,...,z_{s(\lambda)}$ and $\gamma_e(\k^\times)$ act diagonally on $V$ it follows from \eqref{e:spanningzgzero} that $\Gamma$ fixes $\z(\g_0)$ pointwise, and thus fixes $\z_0^{*}$, $\g_0/[\g_0,\g_0]$ and $S(\g_0/[\g_0,\g_0])$ pointwise. Furthermore, identifying $\chi_0$ with its restriction to $\kk$, there is by \eqref{e:relatingreductions} a $\Gamma$-equivariant isomorphism $U_{\chi_0}(\kk,e_0)\xrightarrow{\sim} \hU_{\chi_0}(\kk,e_0)$. It thus suffices to show that $\chi_0 \in \kk^*$ supports a one dimensional $\Gamma$-fixed representation of $U(\kk,e_0)$ (cf. \eqref{e:theLevi}). 
If we write $\c_0$ for a $\Gamma$-stable complement to $[\kk^{e_0}, \kk^{e_0}]$ inside $\kk^{e_0}$ then we can apply Proposition~\ref{L:generators}(i), Proposition~\ref{L:dimEge}, Lemma~\ref{L:celemma} and Lemma~\ref{L:uniquesheet} to see that $S(\c_0) \isoto U(\kk,e_0)^\ab$ as $\Gamma$-algebras. The latter implies that $\c_0^* \cong \E(\kk,e_0)$ as $\Gamma$-sets and thus $\E(\kk, e_0) \neq \emptyset$, since the origin of $\c_0^*$ is $\Gamma$-fixed.

Let $V$ be the one dimensional $\Gamma$-stable $U(\kk, e_0)$-module constructed above. As we observed in Lemma~\ref{L:basiclemma} the $p$-characters supporting such a representation correspond to maximal ideals of $Z_p(\kk,e_0)$ containing the kernel $K_\Gamma$ of the map $Z_p(\kk,e_0) \to U(\kk,e_0)^\ab_\Gamma$. It follows from Proposition~\ref{L:dimEge} that under the identification $Z_p(\kk,e_0) = \k[(\chi_0 + \cv_0)^{(1)}]$ the vanishing locus of $K_\Gamma$ lies inside $(\chi_0 + X_0)^\Gamma$ where $\chi_0 + X_0 = (\chi_0 + \cv_0) \cap \S_0$ and $\S_0$ denotes the unique sheet of $\kk$ containing $e_0$. We define $\cv_0$ for $\kk^*$ analogous to how we defined $\cv$ for $\g^*$ in Section~\ref{ss:extendedslice}, although here we require it to be $\k\Gamma$-stable.

It follows from \eqref{e:lambda0} that $s(\lambda^{(0)}) = 0$ and so Lemma~\ref{L:celemma} tells us that $\c_0^{\Gamma_0}=\c_0^{\widehat{\Gamma}_0}=0$. Since $\Gamma_0\subseteq\Gamma$, this implies $\c_0^\Gamma = 0$. Now apply Proposition~\ref{P:semiclassical} to get $(\chi_0 + X_0)^\Gamma = \chi_0$. Hence the $p$-character of $V$ is an element of $(\chi_0 + X_0)^\Gamma = \chi_0$. This concludes the proof of (ii).

Finally, (iii) follows by combining (i) and (ii) with Theorem~\ref{T:parabolicinductiontheorem}(ii) and Corollary~\ref{cor: full image of char map}. To apply the latter, note that $\Gamma$ preserves $e_0$ and $\z(\g_0)$ as above.
\end{proof}

\subsection{Abelian quotients in the classical case} We keep all the data fixed from Section~\ref{ss:invariantparabolic} and write $\c$ for a $\Gamma$-stable complement to $[\g^e, \g^e]$ in $\g^e$.

\begin{Theorem}
\label{T:abelianquotients}
\begin{enumerate}
\setlength{\itemsep}{4pt}
\item[(1)] If $e$ lies in a unique sheet of $\g$ then there is a $\Gamma$-equivariant isomorphism $S(\c) \isoto U(\g,e)^\ab,$ so that $U(\g,e)^\ab$ is a polynomial ring in $c(e)$ variables.
\item[(2)] We have $S(\c^\Gamma) \isoto U(\g,e)^\ab_\Gamma$ so $U(\g,e)^\ab_\Gamma$ is a polynomial ring in $c_{\Gamma}(e)$ variables.
\end{enumerate}
\end{Theorem}
\begin{proof}
Let $(\g_0', \O_0')\in\RIndat(\g)$ be a rigid induction datum corresponding to the unique sheet $\S$ containing $e$. By Proposition~\ref{L:generators} there is a $\Gamma$-equivariant surjection $S(\c) \onto U(\g,e)^\ab$ and by Proposition~\ref{L:dimEge} the Krull dimension of $U(\g,e)^\ab$ is precisely $\dim(\chi + X) = \dim \z(\g_0')$. Finally by Lemma~\ref{T:sheetstheorem}(ii) and Lemma~\ref{L:celemma} we have $\dim \z(\g_0') = c(e) = \dim \c$. This implies that $S(\c) \onto U(\g,e)^\ab$ is an isomorphism, which completes the proof of (1).

When $\lambda$ is exceptional, Lemma~\ref{L:celemma} implies that $e$ lies in a unique sheet of $\g$. Furthermore, by \cite[3.8, Proposition 2]{JaNO} we may choose $\Gamma$ to either be trivial or $\{\Id,-\Id\}\subseteq \SO(N)_e$ (depending on whether $\lambda_i=\lambda_{i+1}$ in the unique 2-step) and thus it acts trivially on $U(\g,e)^\ab$ and $\c$. So (2) reduces to (1) in this case. A similar argument shows that if $\varepsilon=1$ and every part of $\lambda$ is even then $\Gamma$ is trivial and $e$ lies in a unique sheet; thus (2) reduces to (1) in this case as well.

We now prove (2) under the assumption that $\lambda$ is not exceptional and that if $\epsilon=1$ then $\lambda$ is not very even. Adopt the notation of Proposition~\ref{P:fixedinduction}. Write $K_\Gamma$ for the kernel of the homomorphism $Z_p(\g,e) \to U(\g,e)^\ab_\Gamma$, and identify $\Spec Z_p(\g,e) / K_\Gamma$ with a subset of $\chi + \cv$. By Lemma~\ref{L:basiclemma} and Proposition~\ref{P:fixedinduction}(3) we see that $\chi + Z$ is contained $\Spec Z_p(\g,e) / K_\Gamma$. Since $U(\g,e)$ is finite over the $p$-centre we can apply Lemma~\ref{T:sheetstheorem}(i), Lemma~\ref{L:extendedslice}(ii) and \eqref{e:spanningzgzero} to deduce that the Krull dimension of $U(\g,e)^\ab_\Gamma$ is greater than or equal to $s(\lambda)$. According to Lemma~\ref{L:celemma} we have $\dim \c^\Gamma = s(\lambda)$ and it follows that the surjection $S(\c^\Gamma) \onto U(\g,e)^\ab_\Gamma$ from Proposition~\ref{L:generators} is an isomorphism.
\end{proof}

\subsection{Proof of the Main Theorem for classical groups}
\label{ss:proofofmain}
By Section~\ref{ss:NilpotentReduction} we see that the main theorem for classical groups can be reduced to the case where $\chi$ is nilpotent, so here we adopt this hypothesis. Throughout this Section we fix $e\in \Nc(\g)$ with partition $\lambda \in \PeN$ and write $\O = G\cdot e$.

For part (1), we begin by fixing some notation; for this part of the proof, we do not set things up as in Section~\ref{ss:invariantparabolic}. We instead fix as in Section~\ref{ss:LSinduction} a rigid induction datum $(\g_0,\O_0)\in\RIndat(\O)$ (which is unique up to conjugacy) and a parabolic subgroup $P$ of $G$ such that $\g_0$ is a Levi factor of $\p=\Lie(P)$. Denote $\r=\Rad(\p)$, so that $\p=\g_0\oplus \r$. By Remark~\ref{R: redtorep}, we may assume $e\in \O\cap(\O_0 + \r)$. Let $e_0 \in \O_0$ and $r\in\r$ be such that $e=e_0+r$ and retain the notation $\chi\in \g^*$ and $\chi_0=\chi\vert_{\g_0}\in \g_0^*$.

Since $e$ lies in a unique sheet, Theorem~\ref{T:abelianquotients} tells us that $\E(\g,e) \cong \mathbb{A}_\k^{c(e)}$. We also have $\dim \hE_{\chi_0 + \z_0^*}(\g_0, e_0) = \dim \z(\g_0)$ by Lemma~\ref{L:dimensionofE}, whilst $\dim \z(\g_0) = c(e)$ by Lemma~\ref{T:sheetstheorem}(ii) and Lemma~\ref{L:celemma}. The parabolic induction morphism $\hE_{\chi_0 + \z_0^*}(\g_0, e_0) \to \E(\g,e)$ is finite by Theorem~\ref{T:parabolicinductiontheorem} and, since finite morphisms are closed and $\E(\g,e)$ is irreducible, we deduce that $\hE_{\chi_0 + \z_0^*}(\g_0, e_0) \onto \E(\g,e)$. Finally, by the construction of the parabolic induction functor, this implies that parabolic induction $\Rep_{p^{d_{\chi_0}}} U_{\chi_0}(\g_0) \to \Rep_{p^{d_\chi}} U_\chi(\g)$ is surjective, proving (1).

To prove (2) we first suppose $\lambda$ is exceptional. Lemma~\ref{L:celemma} implies that $e$ lies in a unique sheet, and so (2) reduces to (1) in this case. A similar argument applies if $\varepsilon=1$ and every part of $\lambda$ is even, as in the proof of Theorem~\ref{T:abelianquotients}.

In the following we use the notation $\Rep_d^H(A)$ to denote the set of $d$-dimensional $A$-modules which are $H$-stable, when $A$ is an algebra and $H\subseteq \Aut(A)$. For this part of the proof, we adopt all of the set-up and notation from Section~\ref{ss:invariantparabolic} (as we may do by Remark~\ref{R: redtorep}).

Suppose that $\lambda$ is not exceptional and that if $\epsilon=1$ then $\lambda$ is not very even, and let $(\g_0, G_0\cdot e_0)\in\Indat(\O)$ be the induction datum defined in Section~\ref{ss:invariantparabolic}. By Theorem~\ref{T:abelianquotients} we see that $\E(\g,e)^\Gamma \cong \mathbb{A}_\k^{c_\Gamma(e)}$. By Proposition~\ref{P:fixedinduction}, we know that the parabolic induction functor \eqref{e:INDUCTION} restricts to a finite morphism $\hE_{\chi_0 + \z_0^*}(\g_0, e_0)^\Gamma \to \E(\g,e)^\Gamma$ and that $\dim \hE_{\chi_0 + \z_0^*}(\g_0, e_0)^\Gamma \ge \dim (\chi + Z)$. By Lemma~\ref{L:extendedslice}, Lemma~\ref{L:celemma}, Theorem \ref{T:abelianquotients} and \eqref{e:spanningzgzero}, we have $\dim (\chi + Z) = \dim \z(\g_0) = c_\Gamma(e)=\dim\E(\g,e)^\Gamma$. Since finite morphisms are closed and $\E(\g,e)^\Gamma$ is irreducible we conclude that $\hE_{\chi_0 + \z_0^*}(\g_0, e_0)^\Gamma \to \E(\g,e)^\Gamma$ surjects. 
It follows  that the parabolic induction map $\Rep_{p^{d_{\chi_0}}} U_{\chi_0}(\g_0) \to \Rep_{p^{d_\chi}} U_\chi(\g)$ restricts to a surjective map $\Rep_{p^{d_{\chi_0}}}^\Gamma U_{\chi_0}(\g_0) \onto \Rep_{p^{d_\chi}}^\Gamma U_\chi(\g) = \Rep_{p^{d_\chi}}^{G^\chi} U_\chi(\g)$. This last equality follows from the fact that the group $\Gamma$ splits the surjection $G^e \onto G^e / (G^e)^\circ$, which follows from \eqref{e:componentgroupreductivecentraliser} and the remarks preceding \eqref{e:Gammasubseteq}.

In order to complete the proof we must show that every element of $\Rep_{p^{d_\chi}}^{G^\chi} U_\chi(\g)$ can be parabolically induced from a minimal module with a rigid $p$-character. By Lemma~\ref{L:uniquesheet}, $e_0$ lies in a unique sheet of $\g_0$. Applying part (1) of Theorem~\ref{T:main} we see that every element of $\Rep_{p^{d_{\chi_0}}}^\Gamma U_{\chi_0}(\g_0)$ is parabolically induced from a minimal module with a rigid $p$-character and so the proof concludes by using the transitivity of parabolic induction.$\hfill\qed$

\begin{Remark}
We conclude the section by mentioning the natural question of understanding the fibres of the parabolic induction map on minimal modules. In \cite{GTmin} the results for $\GL_N$ imply that minimal modules can be parabolically induced, and what is more, that two induced modules are isomorphic if and only if the original modules are conjugate by a shifted action of the normaliser of the Levi subalgebra. This is analogous to the linkage principle for parabolically induced representations, similar to \cite[\textsection 11]{JaLA}. It seems plausible that this generalises to other types, in which case the main theorem of the current paper would subtend a classification of the minimal $U_\chi(\g)$-modules considered in the present article.
\end{Remark}

\section{Exceptional Lie algebras}\label{S: Exceptionals}

In this section, we explore the case where $G$ is of exceptional type.

\subsection{Set-up}\label{ss:Exc setup}

Our arguments in this section will depend on a comparison between certain values over $\C$ and certain values over $\k$. To set this up carefully, we need to redefine certain objects from previous sections.

Fix $\bG$ a simple algebraic group over $\C$ of adjoint type and let $\bg=\Lie(\bG)$. We prefer to work in this section with groups of adjoint type in order to be able to use the tables of \cite{LT}. In view of Remark~\ref{R:standardreductivedecomp}(3), this choice does not impact the generality of our result.

Fix a maximal torus $\bT$ of $\bG$, and let $\Phi=\Phi(\bG,\bT)$ be the corresponding root system of $\bG$. We always assume in this section that $\Phi$ is of exceptional type and that the prime $p$ is good for $\Phi$; we also fix an algebraically closed field $\k$ of characteristic $p$. Let $\Delta=\{\alpha_1,\ldots,\alpha_d\}$ be a choice of simple roots for $\Phi$, in the Bourbaki labelling \cite{Bou}. We then fix a Chevalley basis $$\bB:=\{\be_\alpha \mid \alpha\in\Phi\}\cup \{\bh_i\mid 1\leq i\leq d\}$$ of $\bg$, where $\be_\alpha$ is a root vector for $\alpha\in\Phi$ and $\bh_i=[\be_{\alpha_i},\be_{-\alpha_i}]$. We choose our Chevalley basis to have the structure constants as used in \cite{LT}. One property of such a Chevalley basis is that the $\Z$-span of $\bB$ in $\bg$ is a $\Z$-form of $\bg$, which we denote $\bg_\Z$. We may then define the Lie algebra $\g:=\bg_\Z\otimes_\Z \k$; since $p$ is good for $\Phi$, this is the (simple) Lie algebra of the algebraic group $G$ defined over $\k$ with the same root datum as $\bG$. We set $T$ to be the corresponding maximal torus of $G$. The Lie algebra $\g$ then has a Chevalley basis $$\B:=\{e_\alpha \mid \alpha\in\Phi\}\cup \{h_i\mid 1\leq i\leq d\},$$ where $e_\alpha:=\be_\alpha\otimes 1$ and $h_i:=\bh_i\otimes 1$. Set $\t$ to be the $\k$-span of the $h_i$, so that $\t=\Lie(T)$.

We now fix $\be$ a nilpotent element of $\bg$; specifically, an element listed in \cite[Table 2]{LT}. We write $e$ for the corresponding element of $\g$. By Remark~\ref{R: redtorep}, this choice does not affect the generality of our results. Since $\be$ lies in the Chevalley $\Z$-form of $\g$, we in fact have $e=\be\otimes 1$. Because $\bg$ and $\g$ are simple Lie algebras, the Killing forms $\bkappa:\bg\times\bg\to \C$ and $\kappa:\g\times\g\to\k$ are non-degenerate $\bG$-(resp. $G$-)invariant bilinear forms. There is thus a $\bG$-equivariant isomorphism $\bg\simeq \bg^*$ and a $G$-equivariant isomorphism $\g\simeq \g^*$; we write $\bchi$ for the image of $\be$ under the first isomorphism and $\chi$ for the image of $e$ under the second.

\begin{Remark}
	We have used $\Z$ here to base change from $\C$ to $\k$, however we could have used any domain $R\leq\C$ which is equipped with a map $R\to \k$ restricting to the natural map $\Z\twoheadrightarrow\F_p\hookrightarrow \k$. Regardless of the choice of $R$, the Lie algebra $\g=\g_R\otimes_R \k$ is the same. Later on, we will need to use such a domain $R$ for base change -- doing so will not change the Lie algebra we are considering, nor the Chevalley basis of $\g$.
\end{Remark}

The nilpotent elements $\be$ in $\bg$ and $e$ in $\g$ belong to corresponding nilpotent orbits in $\bg$ and $\g$. We write $\bG^\be$ for the stabiliser of $\be$ in $\bG$ and $G^e$ for the stabiliser of $e$ in $G$. Let $\bgamma_\be$ be a cocharacter associated to $\be$ and let $\gamma_e$ be a cocharacter associated to $e$; specifically, choose these to be the cocharacters denoted by $\tau$ in \cite[\textsection 10]{LT}. We then define, as in Section~\ref{ss:associatedcochar}, $\bC^\be:= \bG^\be\cap C_\bG(\bgamma_\be(\C^\times))$ and $C^e:=G^e\cap C_G(\gamma_e(\k^\times))$. Defining $$\pi=C^e/(C^e)^\circ,$$ there are, by \cite[Theorem 36]{MS}, \cite[Theorem 3.9]{PrKR} and \eqref{e:componentgroupreductivecentraliser}, group isomorphisms $$\bG^\be/(\bG^\be)^\circ\cong \bC^\be/(\bC^\be)^\circ \cong \pi \cong C^e/(C^e)^\circ \cong G^e/(G^e)^\circ.$$ The isomorphism type of each $\pi$ can be found in the tables of \cite[\textsection 8.4]{CM} (noting that one must omit any $\Z/m\Z$-factors, as $G$ is of adjoint type); from these tables, it is clear that $p\nmid\left|\pi\right|$ whenever $p$ is good for $\Phi$.

In fact, for each chosen representative $e$ of each nilpotent orbit the tables in \cite[\textsection 10]{LT} give elements $\bsigma_1,\ldots, \bsigma_r\in\bC^\be$ (resp. $\sigma_1,\ldots, \sigma_r\in C^e$) such that their images $\overline{\bsigma}_1,\ldots, \overline{\bsigma}_r$ in $\bC^\be/(\bC^\be)^\circ$ (resp. $\overline{\sigma}_1,\ldots, \overline{\sigma}_r$ in $C^e/(C^e)^\circ$) generate said group. We further define $\bPi\leq \bC^\be$ to be the subgroup generated by $\bsigma_1,\ldots, \bsigma_r\in\bC^\be$ and $\Pi\leq C^e$ to be the subgroup generated by $\sigma_1,\ldots, \sigma_r\in C^e$.

We now define $\bg^\be$ (resp. $\g^e$) to be the centraliser of $\be$ (resp. $e$) in $\bg$ (resp. $\g$), and define $$\bc:=\bg^\be/[\bg^\be,\bg^\be]\qquad\mbox{and}\qquad \c:=\g^e/[\g^e,\g^e].$$ There are obvious actions of $\bC^\be$ on $\bc$ and $C^e$ on $\c$. It is known by \cite[\textsection 5.1]{PT1} that $(\bC^\be)^\circ$ acts trivially on $\bc$, thus inducing an action of $\pi$ on $\bc$. Over $\k$, we have the following lemma.
\begin{Lemma}\label{L:dagger}
	$(C^e)^\circ$ acts trivially on $\c$.
\end{Lemma}
\begin{proof}
	Let $S$ be a maximal torus of $(C^e)^\circ$; it suffices to show that $S$ acts trivially on $\c$ (see \cite[Lemma 15.4]{MT}). By \cite[Lemma 4.6]{JaNO}, $L:=C_G(S)$ is a Levi subgroup of $G$ such that $e$ is distinguished in $\Lie(L)$. In particular, $Z(L)\subseteq G^e$ and, since $\gamma_e(\k^\times)\subseteq L$, we also have $Z(L)\subseteq C_G(\gamma_e(\k^\times))$. This implies $S\subseteq Z(L)^\circ\subseteq (C^e)^\circ$, and hence that $S=Z(L)^\circ$ (by the maximality of $S$). 
	
	Fix now a maximal torus $\widetilde{T}$ in $L$ which contains $S$ (this may be different from the maximal torus $T$ defined earlier), a Borel subgroup $B_L$ of $L$ which contains $\widetilde{T}$, a parabolic subgroup $P$ of $G$ for which $L$ is a Levi factor, and a Borel subgroup $B$ of $G$ with $B_L\subseteq B\subseteq P$. This data determines a root system $\widetilde{\Phi}$ of $G$, a set of simple roots $\Pi\subseteq \widetilde{\Phi}$, a parabolic root subsystem $\widetilde{\Phi}_L$ for $L$, and a set of simple roots $I=\Pi\cap\Phi_L$ for $\Phi_L$. Since $G$ is adjoint, \cite[II.1.18]{JanRAGS} implies that $X(S)=X(Z(L)^\circ)=X(T)/\Z I$; in particular the set of weights of $S$ on $\g$ is the image of $\widetilde{\Phi}\cup\{0\}$ under the natural map $f:X(T)\twoheadrightarrow X(T)/\Z I$ (this is a {\em restricted root system} in the language of \cite{BnG}, once zero is removed). The set of weights of $S$ on $\c$ is thus a subset of $f(\widetilde{\Phi}\cup\{0\})$.
	
	Since $\Lie(S)$ acts trivially on $\c$, to show that $S$ acts trivially on $\c$ it suffices to show that $f(\widetilde{\Phi}\cup\{0\}{})\cap pX(S)=0$. Since $G$ is adjoint, $X(T)=\Z I\oplus \Z (\Pi\setminus I)$ and thus the result follows in a straightforward manner from the fact that $p$ is good and thus does not divide any of the coefficients of any elements in $\widetilde{\Phi}$ when written in terms of simple roots $\Pi$.
\end{proof}

This lemma implies that there is thus an action of $\pi$ on $\c$. Our goal in the next subsection is to show that 
$$ c_\pi(\be):= \dim \bc^\pi \geq \dim \c^\pi=: c_\pi(e).$$

\begin{Remark}
	Note that our definition of $\c$ here is different from that used in the earlier sections of this paper, but for $\Gamma\subseteq C^e$ as in those sections the two notions are isomorphic as $\Gamma$-modules. In particular, $c_\pi(e)=c_\Gamma(e)$ where the latter is defined as in \eqref{e:def cgamma}.
\end{Remark}

Let $\brho:\pi\to\GL(\bc)$ and $\rho:\pi\to\GL(\c)$ be the representations of $\pi$ on $\bc$ and $\c$ respectively, and define $$\bpsi:=\frac{1}{\left|\pi\right|}\sum_{g\in\pi}\brho(g)\in \End(\bc)\qquad \mbox{and}\qquad \psi:=\frac{1}{\left|\pi\right|}\sum_{g\in\pi}\rho(g)\in \End(\c),$$ recalling that $p\nmid\left|\pi\right|$ under our assumptions.

It is then elementary representation theory that $$c_\pi(\be) = \rank \bpsi \qquad\text{ and }\qquad c_\pi(e) = \rank(\psi).$$

\subsection{Bound for $c_\pi(e)$}\label{ss:cpi bound}

For the following results, we make arguments involving base change. It will thus be helpful here to establish some conventions. Suppose that $R\leq \C$ is a unital subring equipped with a map $R\to \k$ restricting to the natural map $\Z\twoheadrightarrow\F_p\hookrightarrow \k$. We write $M_n(R)$ for the ring of all $n\times n$ matrices with entries in $R$. There is a natural inclusion $M_n(R)\hookrightarrow M_n(\C)$, and we denote by $\bA$ the image of a matrix $A\in M_n(R)$ under this map. Furthermore, by construction there is a natural map $M_n(R)\to M_n(\k)$; we write $\overline{A}$ for the image of $A\in M_n(R)$ under this map. Finally, given a basis $B$ of a vector space $V$ and linear map $T:V\to V$, we adopt the notation $[T]_B$ for the matrix of $T$ with respect to the basis $B$.

For the following proposition, note that $\bPi$ acts on $\bg$ and $\Pi$ acts on $\g$; write $\bomega:\bPi\to\GL(\bg)$ and $\omega:\Pi\to\GL(\g)$ for the corresponding representations. 

\begin{Proposition}\label{P: Re for g}
	There exists a unital subring $R_e\leq \C$ with the following properties:
	\begin{itemize}
		\item $R_e$ is a principal ideal domain.
		\item Each bad prime for $\Phi$ is invertible in $R_e$.
		\item $R_e$ admits a homomorphism $R_e\to\k$ extending the natural surjection $\Z\twoheadrightarrow\F_p$.
		\item For each $i=1,\ldots,r$, there exists a matrix $A_i\in M_{\dim\bg}(R_e)$, such that $$\bA_i=[\bomega(\bsigma_i)]_\bB\qquad\mbox{and}\qquad \overline{A}_i=[\omega(\sigma_i)]_\B.$$
	\end{itemize}
\end{Proposition}

\begin{proof}
	Let $\varsigma$ be a primitive third root of unity, $i$ a primitive fourth root of unity and $\zeta$ a primitive fifth root of unity. We then fix $R_e$ as given in Table~\ref{Tab: Re}; such $R_e$ is generally determined by the Cartan type of $\Phi$ and the group $\pi$, though when $\Phi$ has type $E_8$ and $\pi\cong S_3$ there are two $R_e$ we consider, depending on whether or not $\O$ is the orbit with Bala-Carter label $D_4(a_1)+A_1$. Note that, by construction, each bad prime for $\Phi$ is invertible in $R_e$. Furthermore, each $R_e$ is a principal ideal domain (since localisations of principal ideal domains are principal ideal domains, and it is well known that $\Z$, $\Z[\varsigma]$, $\Z[i]$ and $\Z[\zeta]$ are principal ideal domains), and each $R_e$ maps to $\k$ in such a way that restricts to the natural surjection $\Z\to\F_p$ and sends a primitive $k$-th root of unity to  a primitive $k$-th root of unity (where $k=3,4,5$, as appropriate).
	
	\begin{table}
		\begin{center}
			\caption{The ring $R_e$}\label{Tab: Re}
			\bgroup
			\renewcommand{\arraystretch}{2}
			\begin{tabular}{|c|c|c|}
				\hline 
				Type & $\pi$ & $R_e$ \\
				\hline
				$E_6, E_7, F_4, G_2$ & $1$ & $\Z[\frac{1}{2},\frac{1}{3}]$ \\
				\hline 
				$E_8$ & $1$ & $\Z[\frac{1}{2},\frac{1}{3},\frac{1}{5}]$ \\
				\hline
				$E_6, E_7, F_4$ & $S_2$ & $\Z[\frac{1}{2},\frac{1}{3}]$ \\
				\hline
				$E_8$ & $S_2$ & $\Z[\frac{1}{2},\frac{1}{3},\frac{1}{5},i]$ \\
				\hline
				$E_6, E_7, F_4, G_2$ & $S_3$ & $\Z[\frac{1}{2},\frac{1}{3},\varsigma]$ \\
				\hline
				$E_8$ & $S_3$ ($e\neq D_4(a_1)+A_1$) & $\Z[\frac{1}{2},\frac{1}{3},\frac{1}{5},\varsigma]$\\
				\hline
				$E_8$ & $S_3$ ($e=D_4(a_1)+A_1)$ & $\Z[\frac{1}{2},\frac{1}{3},\frac{1}{5},i]$ \\
				\hline
				$F_4$ & $S_4$ & $\Z[\frac{1}{2},\frac{1}{3},\varsigma]$\\
				\hline
				$E_8$ & $S_5$ & $\Z[\frac{1}{2},\frac{1}{3},\frac{1}{5},\zeta]$\\
				\hline
			\end{tabular}
			\egroup
		\end{center}
	\end{table}
	
	For each $\alpha\in\Phi$, we denote by $\bU_\alpha$  (resp. $U_\alpha$) the root subgroup of $\bG$ (resp. $G$) with respect to $\bT$ (resp. $T$). We then define a homomorphism of algebraic groups $\bx_\alpha:\G_a\to \bU_\alpha$ (where $\G_a$ denotes the additive group over the suitable field) by $\bx_\alpha(t)=\exp(t\ad(\be_\alpha))$ and a homomorphism of algebraic groups $x_\alpha:\G_a\to U_\alpha$ such that $dx_\alpha(1)=e_\alpha$. We then define $\bn_\alpha(t)=\bx_\alpha(t)\bx_{-\alpha}(-t^{-1})\bx_\alpha(t)$, $\bh_\alpha(t)=\bn_\alpha(t)\bn_\alpha(1)^{-1}$ and $\bn_\alpha=\bn_\alpha(1)$, for $t\in\C$, and similarly $n_\alpha(t)=x_\alpha(t)x_{-\alpha}(-t^{-1})x_\alpha(t)$, $h_\alpha(t)=n_\alpha(t)n_\alpha(1)^{-1}$ and $n_\alpha=n_\alpha(1)$, for $t\in \k$. 
	
	Now, let $\bsigma\in \bG$ be a product of some $\bx_\alpha(s)$, $\bn_\beta(t)$ and $\bh_\gamma(u)$ such that each $s,t,u$ lie in $R_e$ and each $t,u$ are invertible in $R_e$, and let $\sigma\in G$ be the corresponding product of $x_\alpha(s)$, $n_\beta(t)$ and $h_\gamma(u)$ (where we technically replace $s,t,u$ with their images under the map $R_e\to\k$). Then it follows from the observations in \cite[\textsection 2]{LT} that the matrix of the adjoint action of $\bsigma$ on $\bg$ with respect to the basis $\bB$ in fact lies in $M_n(R_e)$, and the matrix of the adjoint action of $\sigma$ on $\g$ with respect to the basis $\B$ is precisely the image of this matrix under the induced map $M_n(R_e)\to M_n(\k)$ coming from $R_e\to\k$. Thus the result will follow if each $\sigma_i$ and $\bsigma_i$ can be written (compatibly) in this form. 
	
	This can be checked easily from the tables in \cite{LT} (and indeed is implicit in the fact that the same tables cover the case of both characteristic zero and good characteristic). The only case that requires some comment is the orbit with Bala-Carter label $E_8(a_7)$, in which case we need to check that $\frac{2(1-3\zeta^2-3\zeta^3)}{25}$ and $\frac{3+\zeta^2+\zeta^3}{5}$ are invertible in $R_e$; this follows as the former has inverse $\frac{-11}{2}+\frac{9}{2}\zeta-3\zeta^2-3\zeta^3+\frac{9}{2}\zeta^4$ and the latter has inverse $\frac{11}{5}+\frac{1}{5}\zeta-\frac{4}{5}\zeta^2-\frac{4}{5}\zeta^3+\frac{1}{5}\zeta^4$.
\end{proof}

This proposition says, in essence, that $\bPi$ acts on $\bg$ as $\Pi$ acts on $\g$. What we want to conclude is that $\pi$ acts on $\bc$ as it acts on $\c$. To do this, we first need to establish an $R_e$-form of $\bc$.

In service of this, recall that $\g_{R_e}$ is the $R_e$-submodule of $\g$ generated by $\bB$. We define $\bg_{R_e}^\be:=\bg^\be\cap \bg_{R_e}$; this is an $R_e$-Lie algebra. We then define $[\bg_{R_e}^\be,\bg_{R_e}^\be]$ to be the derived subalgebra of $\bg_{R_e}^\be$. 

\begin{Lemma}\label{L:ReForms}
	The following hold.
	\begin{enumerate}
		\item The $R_e$-module $\bg_{R_e}^\be$ is a direct summand of $\bg_{R_e}$ and thus free of finite rank.
		\item $\bg_{R_e}^\be\otimes_{R_e}\C=\bg^\be$ and $\bg_{R_e}^\be\otimes_{R_e}\k=\g^e$.
		\item The submodule $[\bg_{R_e}^\be,\bg_{R_e}^\be]$ is a direct summand of $\bg_{R_e}^\be$ and thus free of finite rank.
		\item $[\bg_{R_e}^\be,\bg_{R_e}^\be]\otimes_{R_e}\C=[\bg^\be,\bg^\be]$ and $[\bg_{R_e}^\be,\bg_{R_e}^\be]\otimes_{R_e}\k=[\g^e,\g^e]$.
		\item The quotient module $\bg_{R_e}^\be/[\bg_{R_e}^\be,\bg_{R_e}^\be]$ is a free $R$-module of finite rank.
		\item $(\bg_{R_e}^\be/[\bg_{R_e}^\be,\bg_{R_e}^\be])\otimes_{R_e}\C=\bc$ and  $(\bg_{R_e}^\be/[\bg_{R_e}^\be,\bg_{R_e}^\be])\otimes_{R_e}\k=\c$. 		
	\end{enumerate}
\end{Lemma}

\begin{proof}
	Since $R_e$ is a principal ideal domain, arguing as in \cite[Proof of Lemma 3.7(i)]{PT2} yields that $\bg_{R_e}^\be$ is a direct summand of $\bg_{R_e}$. Since $\bg_{R_e}$ is a free $R_e$-module, and using again that $R_e$ is a principal ideal domain, we get that $\bg_{R_e}^e$ is a free $R_e$-module of finite rank. This proves (1).
	
	The first part of (2) follows by applying the exact  functor $-\otimes_{R_e} \C$ (exact since $R_e$ is a principal ideal domain) to the exact sequence $$0\to \bg_{R_e}^{\be}\to \bg_{R_e}\xrightarrow{\ad(\be)} [\be,\bg_{R_e}]\to 0,$$ while the second part follows as (1) implies that $\bg_{R_e}^\be \otimes_{R_e}\k\hookrightarrow \g^e$ and we know from \cite[Theorem 2.6(iv)]{PrKR} that $\dim_\C(\bg^\be)=\dim_\k(\g^e)$.
	
	Since all bad primes are invertible in $R_e$, the statement that $[\bg_{R_e}^\be,\bg_{R_e}^\be]$ is a direct summmand of $\bg_{R_e}^\be$ follows as in \cite[Proof of Lemma 4.5(i)]{PT2}, proving (3). This then immediately implies (5).
	
	It is easy to see from (2) that there exists an inclusion $[\bg_{R_e}^{\be},\bg_{R_e}^{\be}]\otimes_{R_e}\C\into [\bg^{\be},\bg^{\be}]$ and a surjection $[\bg_{R_e}^{\be},\bg_{R_e}^{\be}]\otimes_{R_e}\k\onto [\g^{e},\g^{e}]$. Since $R_e$ is a principal ideal domain, (3) implies that $[\bg_{R_e}^{\be},\bg_{R_e}^{\be}]$ is free of finite rank. Comparing $\dim_{\C}([\bg^{\be},\bg^{\be}])$ and $\dim_\k([\g^{e},\g^{e}])$ then yields (4) (see the proof of Theorem~\ref{T: exc sing sheet induced} below for more details on how to do this comparison).
	
	Finally, applying the functor $-\otimes_{R_e}\k$ to the split exact sequence $$0\to [\bg_{R_e}^{\be},\bg_{R_e}^{\be}]\to \bg_{R_e}^{\be}\to \bg_{R_e}^{\be}/[\bg_{R_e}^{\be},\bg_{R_e}^{\be}]\to 0$$ yields by (3) and (4) an exact sequence 
	$$0\to [\g^{e},\g^{e}]\to \g^{e}\to (\bg_{R_e}^{\be}/[\bg_{R_e}^{\be},\bg_{R_e}^{\be}])\otimes_{R_e}\k\to 0.$$ Part (6) then follows.
\end{proof}

Since $[\bg_{R_e}^\be,\bg_{R_e}^\be]$ is a direct summand of $\bg_{R_e}^\be$ (and so a free $R_e$-module of finite rank) and $\bg_{R_e}^\be$ is a direct summand of $\bg_{R_e}$ (and so also a free $R_e$-module of finite rank), there exists an $R_e$-basis $\bu_1,\ldots,\bu_{m}$ of $[\bg_{R_e}^\be,\bg_{R_e}^\be]$, which extends to an $R_e$-basis $\bu_1,\ldots,\bu_{m},\bu_{m+1},\ldots,\bu_{m+c}$ of $\bg_{R_e}^\be$, which in turn extends to an $R_e$-basis $\bu_1,\ldots,\bu_{m},\bu_{m+1},\ldots,\bu_{m+c},\bu_{m+c+1},\ldots,\bu_n$ of $\bg_{R_e}$. Setting $\bv_i:=\bu_{m+i} + [\bg_{R_e}^\be,\bg_{R_e}^\be]$, we then get that $\bv_1,\ldots,\bv_c$ is an $R_e$-basis of the free $R_e$-module $\bg_{R_e}^\be/[\bg_{R_e}^\be,\bg_{R_e}^\be]$. Let $v_1:=\bv_1\otimes 1, \ldots, v_c:=\bv_c\otimes 1$ be the corresponding basis of $\c$. Define then $$\bB_\bc=\{\bv_1,\ldots,\bv_c\}\subseteq \bc\qquad \mbox{and}\qquad \B_\c=\{v_1,\ldots,v_c\}\subseteq \c$$ to be bases of $\bc$ and $\c$ respectively.

\begin{Proposition}\label{P:R_e for c}
	For each $g\in\pi$ there exists a matrix $A_g\in M_c(R_e)$ such that $$\bA_g=[\brho(g)]_{\bB_\bc}\qquad\mbox{and}\qquad \overline{A}_g=[\rho(g)]_{\B_\c}.$$
\end{Proposition}

\begin{proof}
	Since the Chevalley basis $\bB$ and the basis $\bB':=\{\bu_1,\ldots,\bu_n\}$ are both free $R_e$-bases of $\bg_{R_e}$, there exists a change of basis matrix $X\in M_n(R_e)$ from $\bB$ to $\bB'$. In particular, setting $\B':=\{\bu_1\otimes 1,\ldots,\bu_n\otimes 1\}\subseteq \g$, we see that $\bX$ and $\overline{X}$ are also change-of-basis matrices from $\bB$ and $\B$ to $\bB'$ and $\B'$ respectively. For $i=1,\ldots,r$, set $A_i':=XA_iX^{-1}$; then by Proposition~\ref{P: Re for g} we have $$\bA'_i=[\bomega(\bsigma_i)]_{\bB'}\qquad \mbox{and}\qquad \overline{A}'_i=[\omega(\sigma_i)]_{\B'}.$$
	
	Since $\bPi\leq \bC^\be$ and $\Pi\leq C^e$, we deduce that $\bPi$ preserves $\bg^\be$ and $[\bg^\be,\bg^\be]$ (under $\bomega$) and $\Pi$ preserves $\g^e$ and $[\g^e,\g^e]$ (under $\omega$). Furthermore, by Proposition~\ref{P: Re for g}, $\bPi$ preserves $\bg_{R_e}^{\be}$ and $[\bg_{R_e}^{\be},\bg_{R_e}^{\be}]$. Therefore, the matrix $A'_i$ is block-upper-triangular, with the first block being an $m\times m$-submatrix, the second a $c\times c$-submatrix and the last an $(n-m-c)\times(n-m-c)$-submatrix. Let $A_{\overline{\bsigma}_i}\in M_c(R_e)$ be the middle of these, so the $c\times c$ submatrix of $A_i$ obtained by taking both the rows and columns of $A_i$ labelled $m+1,\ldots, m+c$. From Lemma~\ref{L:ReForms} and our above observations and constructions, it is then clear that
	$$\bA'_{\overline{\bsigma}_i}=[\bomega(\bsigma_i)\vert_{\bc\to\bc}]_{\bB_\bc}=[\brho(\overline{\bsigma}_i)]_{\bB_\bc}\qquad \mbox{and}\qquad \overline{A}'_{\overline{\bsigma}_i}=[\omega(\sigma_i)\vert_{\c\to\c}]_{\B_\c}=[\rho(\overline{\sigma}_i)]_{\B_\c}.$$ Since $\pi$ is generated by $\overline{\bsigma}_1=\overline{\sigma}_1,\ldots,\overline{\bsigma}_r=\overline{\sigma}_r$, the result then follows.
\end{proof}

To conclude, we need to appeal to the following result.
\begin{Lemma}\label{L:rankAgoesdown}
	Let $R\leq \C$ be a unital subring with the first and third properties of Proposition~\ref{P: Re for g} and let $A\in M_n(R)$. Then $\rank(\bA)\geq \rank(\overline{A})$.
\end{Lemma}
\begin{proof}
	Since $R$ is a principal ideal domain, we may put $A$ in Smith Normal Form. In other words, there exist invertible matrices $S,T\in M_n(R)$ such that $SAT$ is diagonal; say $SAT=\diag(\lambda_1,\ldots,\lambda_n)$ for $\lambda_1,\ldots,\lambda_n\in R$. This also implies that $\bS,\bT\in M_n(\C)$ and $\overline{S}, \overline{T}\in M_n(\k)$ are invertible and that $\bS\bA\bT=\diag(\lambda_1,\ldots,\lambda_n)$ and $\overline{S}\overline{A}\overline{T}=\diag(\lambda_1\otimes 1,\ldots,\lambda_n\otimes 1)$. Thus $$\rank(\bA)=\rank(\bS\bA\bT)=\#\{i\mid \lambda_i\neq 0\}\geq \#\{i\mid \lambda_i\otimes 1\neq 0\}=\rank(\overline{S}\overline{A}\overline{T})=\rank(\overline{A}).$$
\end{proof}

Putting this all together, we get the following result.
\begin{Proposition}\label{P:rankpsi}
	$c_\pi(\be)\geq c_\pi(e)$.
\end{Proposition}
	
\begin{proof}
	Set $$D=\frac{1}{\left|\pi\right|} \sum_{g\in\pi} A_g\in M_c(R_e).$$ Then by Proposition~\ref{P:R_e for c} and Lemma~\ref{L:rankAgoesdown} we have $$c_\pi(\be)=\rank(\bpsi)=\rank([\bpsi]_{\bB_\bc})=\rank(\bD)\geq \rank(\overline{D})=\rank([\psi]_{\B_\c})=\rank(\psi)=c_\pi(e).$$
\end{proof}	

\subsection{Orbits lying in a single sheet}\label{ss:Exc sing sheet}

Having established these preparatory results, we are now mostly ready to prove Theorem~\ref{T:main}. We first consider part (1) of the theorem, which applies when $e$ lies in a single sheet of $\g$.

For the following proposition, we define $c(e)=\dim \c$ (by Lemma~\ref{L:ReForms}, this coincides with $\dim(\bc)$).

\begin{table}
	\begin{center}
		\caption{Excluded rigid orbits}\label{Tab: Bad Rigid Orbits}
		\bgroup
		\renewcommand{\arraystretch}{2}
		\begin{tabular}{|c|c|c|c|c|c|}
			\hline 
			$G_2$ & $F_4$ & $E_7$ & $E_8$ & $E_8$ & $E_8$ \\
			\hline
			$\widetilde{A}_1$ & $\widetilde{A}_2+A_1$ & $(A_1+A_3)'$ & $A_3+A_1$ & $D_5(a_1)+A_2$ & $A_5+A_1$ \\
			\hline 
		\end{tabular}
		\egroup
	\end{center}
\end{table}

\begin{Theorem}\label{T: exc sing sheet induced}
	Suppose that $e$ lies in a single sheet of $\g$ but that the orbit of $e$ is not listed in Table~\ref{Tab: Bad Induced Orbits} or Table~\ref{Tab: Bad Rigid Orbits}, and let $(\g_0,\O_0)\in \RIndat(G\cdot e)$ be such that there exists a parabolic subalgebra $\p=\g_0\oplus \r$ with $e\in\O_0 + \r$. Fix $e_0\in\O_0$ and $r\in \r$ such that $e=e_0+r$. Then $U(\g,e)^\ab$ is a polynomial algebra in $c(e)$-variables, and every minimal $U_\chi(\g)$-module is parabolically induced from a minimal $U_{\chi_0}(\g_0)$-module.
\end{Theorem}

\begin{proof}
	The proof works identically to the proof of Theorem~\ref{T:abelianquotients} and the proof of Theorem~\ref{T:main}(1) in Section~\ref{ss:proofofmain}, except that we may no longer appeal to Lemma~\ref{L:celemma} for the statement that $c(e)=\dim\z(\g_0)$. Instead, we may determine $c(e)$ in all cases by consulting \cite[Tables 3 and 4]{PS} and we may determine $\dim\z(\g_0)$ by consulting the tables of \cite{DE} (which, as relevant here, apply in positive characteristic by our discussion in Section~\ref{ss:LSinduction}). Comparison of these sources yields that $c(e)=\dim(\z(\g_0))$ whenever $e$ lies on a single sheet and is not listed in Table~\ref{Tab: Bad Induced Orbits} or Table~\ref{Tab: Bad Rigid Orbits}. Strictly speaking, this comparison yields another exception to this equality for $(3A_1)''\subseteq E_7$; however, this is due an error in \cite{DE} (easily seen from the correct entry in the ``characteristic'' column of the appropriate table).
\end{proof}

This theorem implies that when $e$ lies in a unique sheet all $G^\chi$-stable minimal modules are parabolically induced from a minimal module with rigid $p$-character. This proves Theorem~\ref{T:main}(2) for a sizeable percentage of nilpotent orbits. The next section tackles the remaining orbits for which we have results: those that are even.

\subsection{Even orbits}\label{ss:Exc even}
	
	Given a nilpotent orbit $\O$ in $\g$, the weighted Dynkin diagram of $\O$ is a labelling of the Dynkin diagram of $G$ by the numbers $0$, $1$ and $2$. The construction of this labelling over $\C$ can be found in \cite[\textsection 3.5]{CM}; over $\k$, one may either construct it directly using associated cocharacters (as in \cite[\textsection 3]{LT}) or from the bijection between nilpotent orbits over $\k$ and over $\C$ (see \cite[\textsection 2]{PrKR}). Both approaches yield the same labelling, since both constructions involve the existence of $\hat{e}\in \O$ and a (unique) associated cocharacter $\gamma_{\hat{e}}$ for $\hat{e}$ with the property that \begin{equation}\label{e: wdd action}
		\gamma_{\hat{e}}(t)\cdot e_{\pm\alpha}=t^{\pm r_\alpha} e_\alpha\quad\mbox{for all}\quad \alpha\in\Delta\qquad \mbox{and} \qquad \gamma_{\hat{e}}(t)\cdot h=h\quad\mbox{for all}\quad h\in\t
	\end{equation} where $r_\alpha\in\{0,1,2\}$ is the label of the node on the weighted Dynkin diagram corresponding to the simple root $\alpha$.	
	
	 A nilpotent orbit $\O$ is called {\it even} if all the labels on its weighted Dynkin diagram are even (so $0$ or $2$); for (a representative $e$ of) such an orbit, we define $$d(e)=d(\O)=\mbox{\# 2s on the weighted Dynkin diagram of }\O.$$
	
	 For even orbits, the following result is known.
	
	\begin{Proposition}\label{P:Even nice setup}
		Let $\O\subseteq \g$ be an even nilpotent orbit with weighted Dynkin diagram $\Omega$, and let $\hat{e}\in \O$ have the property that an associated cocharacter $\gamma_{\hat{e}}$ for $\hat{e}$ satisfies \eqref{e: wdd action}. Let $\Omega_0$ be the subdiagram of $\Omega$ consisting of only those nodes with label $0$, and let $\hat{G}_0$ be the standard Levi subgroup of $G$ corresponding to the simple roots in $\Omega_0$. Write $\hat{\g}_0=\Lie(\hat{G}_0)$ and $\hat{e}_0=0\in\hat{\g}_0$. Then the following hold:
		\begin{enumerate}
			\item $(\hat{\g}_0,\hat{e}_0)$ is a rigid induction datum for $\O$.
			\item $\dim(\z(\hat{\g}_0))=d(\O)$.
			\item There exists a parabolic subgroup $\hat{P}$ of $G$ with $\hat{G}_0$ as a Levi factor such that $G^{\hat{e}}\subseteq \hat{P}$.
			\item $\hat{C}^{\hat{e}}\subseteq \hat{G}_0$ and $\gamma_{\hat{e}}(\k^\times)\subseteq \hat{G}_0$.
			\item $\hat{e}$ lies in the Lie algebra of the unipotent radical of $\hat{P}$.
		\end{enumerate} 
	\end{Proposition}

\begin{proof}
	Part (1) is \cite[Theorem 7.1.6]{CM} (which holds over $\k$ as in Section~\ref{ss:LSinduction}), while part (2) is clear from the construction. For parts (3) and (4), note that $\hat{G}_0$ is precisely the centraliser in $G$ of $\gamma_{\hat{e}}$. Then part (3) is \cite[Proposition 2.3(ii)]{PrKR} and part (4) follows immediately from the fact that $\hat{C}^{\hat{e}}:=G^{\hat{e}}\cap \hat{G}_0$. Finally, part (5) follows from the construction of $\hat{e}$ in \cite{PrKR} and the description of the Lie algebra of the unipotent radical of $P$.
\end{proof}

\begin{Remark}
	We have used the notation $\hat{e}$ (and similar adornments $\hat{C}, \hat{G}_0,...$) since the unadorned $e$ is reserved in this section for the nilpotent orbit representatives given in \cite{LT}. A conjugation argument allows us to then deduce the following.
\end{Remark}

\begin{Corollary}\label{Cor: even data}
	Let $\O\subseteq \g$ be an even nilpotent orbit and let $e\in \O$ have an associated cocharacter $\gamma_{e}$. Then there exists a Levi subgroup $G_0$ such that the following hold:
	\begin{enumerate}
		\item $(\g_0,0)$ is a rigid induction datum for $\O$.
		\item $\dim(\z(\g_0))=d(\O)$.
		\item There exists a parabolic subgroup $P$ of $G$ with $G_0$ as a Levi factor such that $G^e\subseteq P$.
		\item $C^{e}\subseteq G_0$ and $\gamma_{e}(\k^\times)\subseteq G_0$.
		\item $e$ lies in the Lie algebra $\r$ of the unipotent radical of $P$.
	\end{enumerate} 
\end{Corollary}

From now on we fix all the data from Corollary~\ref{Cor: even data}. We also fix $\Gamma$ to be a subgroup of $C^e$ which surjects onto $\pi$ and which has the property that $p\nmid \left|\Gamma\right|$.

\begin{Remark}
    Such a choice of $\Gamma$ is possible thanks to \cite[Theorem 10.1 and Lemma 10.2]{BDT} and \cite[\textsection 10]{LT}, noting that Lusztig's canonical quotient coincides with $\pi$ in the cases in which \cite[Lemma 10.2]{BDT} does not apply. Note that while the result in \cite{BDT} is for centralisers of unipotent elements, the same result holds for stabilisers of nilpotent elements due to the existence of a $G$-equivariant Springer isomorphism in good characteristic. In most cases $\Gamma$ splits the surjection $C^e\onto \pi$ - in fact, this always holds when the orbit is even.
\end{Remark}

With this set-up we have that
\begin{eqnarray}
	\label{e:Exc centresubseteq}
	\z(\g_0) \subseteq \g(0, \gamma_e) \cap \g_0
\end{eqnarray}
and that $G_0$ acts trivially on $\z(\g_0)$. Furthermore, by Corollary~\ref{Cor: even data}(4), we have
\begin{eqnarray}
	\label{e:Exc Gammasubseteq}
	\Gamma \subseteq C^{e} \cap G_0.
\end{eqnarray}
These two observations are analogues of \eqref{e:centresubseteq} and \eqref{e:Gammasubseteq} in our current setting. We are then able to obtain an analogue of Proposition~\ref{P:fixedinduction}; unexplained notation is as in earlier sections of the paper.

\begin{Proposition}
	\label{P:Exc fixedinduction}
	We let $\D \subseteq \g^*$ be the decomposition class corresponding to the induction datum $(\g_0, 0)$ and write $\chi + Z := (\chi + \cv) \cap \overline{\D}_\reg$.
	\begin{enumerate}
		\setlength{\itemsep}{4pt}
		\item[(i)] $\Gamma$ preserves $\E_{ \z_0^*}(\g_0, 0)$ and the parabolic induction functor \eqref{e:INDUCTION} restricts to
		\begin{eqnarray}
			\label{e:Exc fixedinduction}
			\E_{\z_0^*}(\g_0, 0)^\Gamma \to \E(\g,e)^\Gamma.
		\end{eqnarray}
		\item[(ii)] Every $p$-character in $\z_0^*$ supports a one dimensional $\Gamma$-fixed $U(\g_0, 0)$-module.
		\item[(iii)] Every $p$-character in $\chi + Z$ supports a one dimensional $\Gamma$-fixed $U(\g, e)$-module in the image of \eqref{e:Exc fixedinduction}.
	\end{enumerate}
\end{Proposition}
\begin{proof}
	In this context, $U(\g_0,0)=U(\g_0)$ and thus $\Gamma\subseteq G_0$ preserves $U(\g_0,0)$ and $\z_0^*$. The parabolic induction map in this context is then $$\Rep_{1} U_{\z_0^*}(\g_0) \to \Rep_{p^{d_{\chi}}} U_{\chi + \z^*}(\g) \isoto \hE_{\chi + \z^*}(\g,e) \to \E(\g,e).$$ The first of these maps is $\Gamma$-equivariant by the final remarks of Section~\ref{ss:twists}, the second is $\Gamma$-equivariant by Theorem~\ref{T:equivariantSkryabin} and \eqref{e:Exc centresubseteq}, and the third is $C^e$-equivariant as the map $U(\g,e)\to \hU_{\chi + \z^*}(\g,e)$ is. This proves (i).
	
	For (ii), let $\zeta\in\z_0^*$. Note that $[\g_0,\g_0]$ is a restricted ideal of $\g_0$ under our assumptions and so $\g_0/[\g_0,\g_0]$ is a restricted Lie algebra. Furthermore $\Gamma$ acts trivially on $\g_0/[\g_0,\g_0]$, since \eqref{e:Exc Gammasubseteq} implies that $\Gamma$ acts trivially on $\z(\g_0)$ and thus on $\z_0^*$ and $\z_0^{**}=\g_0/[\g_0,\g_0]$.
	
	Since $\zeta([\g_0,\g_0])=0$, we may define $\hat{\zeta}:\g_0/[\g_0,\g_0]\to \k$ by $\hat{\zeta}(x+[\g_0,\g_0])=\zeta(x)$ for $x\in\g_0$. A one dimensional $\Gamma$-stable $U(\g_0,0)$-module with $p$-character $\zeta$ is then the same thing as a one dimensional $U_{\hat{\zeta}}(\g_0/[\g_0,\g_0])$-module. Since all simple $U_{\hat{\zeta}}(\g_0/[\g_0,\g_0])$-modules are one dimensional (as $\g_0/[\g_0,\g_0]$ is abelian) and simple modules always exist, (ii) follows.
	
	Finally, as in the proof of \ref{P:fixedinduction}(iii), part (iii) follows by combining (i) and (ii) with Theorem~\ref{T:parabolicinductiontheorem}(ii) and Corollary~\ref{cor: full image of char map}. 
\end{proof}

\begin{Corollary}\label{Cor: Kd fWa even}
	The Krull dimension of $U(\g,e)^\ab_\Gamma$ is greater than or equal to $d(e)$.
\end{Corollary}

\begin{proof}
	Arguing as in the proof of Theorem~\ref{T:abelianquotients}(2), combining	Proposition~\ref{P:fixedinduction}(iii), Lemma~\ref{L:extendedslice} and Corollary~\ref{Cor: even data} implies that $\dim(U(\g,e)^\ab_\Gamma)\geq \dim(\chi+ Z)=\dim \z(\g_0)=d(e)$.
\end{proof}

Recall that in this section, $\c:=\g^e/[\g^e,\g^e]$ and $C^e$ acts on $\c$. With $\Gamma$ as above, we may then consider the fixed point space $\c^\Gamma$; note that as a vector space this is isomorphic to the vector space with the same notation considered in Proposition~\ref{P:semiclassical}(ii). Denote $$c_\Gamma(e):=\dim\c^\Gamma.$$

\begin{Lemma}\label{L: cgammae}
	Suppose that $\O\subseteq \g$ is an even nilpotent orbit with representative $e$ as in \cite[Table 2]{LT}. Then $c_\Gamma(e)=d(e)$.
\end{Lemma}

\begin{proof}
	By Proposition~\ref{P:semiclassical}(ii) there exists a surjective homomorphism $S(\c^\Gamma)\onto U(\g,e)^\ab_\Gamma$; thus, $c_\Gamma(e)$, which is the Krull dimension of $S(\c^\Gamma)$, is at least the Krull dimension of $U(\g,e)^\ab_\Gamma$. Corollary~\ref{Cor: Kd fWa even} then yields that $c_\Gamma(e)\geq d(e)$.
	
	On the other hand, the action of $\Gamma$ on $\c$ passes (surjectively) through the action of $\pi$ on $\c$. In particular, $c_\pi(e)=c_\Gamma(e)$. Furthermore, note that $c_\pi(\be)=d(\be)$ by \cite[Proposition 13]{PT1}. Since the weighted Dynkin diagram is independent of good characteristic, we also know $d(\be)=d(e)$.
	
	Combining these results with Proposition~\ref{P:rankpsi} thus yields
	$$d(e)=d(\be)=c_\pi(\be)\geq c_\pi(e)=c_\Gamma(e)\geq d(e),$$ which gives the result.
\end{proof}

We are now in a position to complete the proof of Theorem~\ref{T:main}(2) in the case of exceptional groups.

\begin{Corollary}\label{Cor: even Main Thm}
	Suppose that $\O\subseteq \g$ is an even nilpotent orbit with representative $e$ as in \cite[Table 2]{LT}, and fix the data from Corollary~\ref{Cor: even data}. Then $U(\g,e)^\ab_\Gamma$ is a polynomial algebra in $c_\Gamma(e)$-variables, and every $G^\chi$-stable minimal $U_\chi(\g)$-module is parabolically induced from a one dimensional $U_{0}(\g_0)$-module.
\end{Corollary}

\begin{proof}
	From the proof of Lemma~\ref{L: cgammae}, the surjective homomorphism of algebras $S(\c^\Gamma)\onto U(\g,e)^\ab_\Gamma$ from Proposition~\ref{P:semiclassical} is between algebras of the same Krull dimension; it is thus an isomorphism, proving the first part of the result.
	
	The second part of the result follows by an argument essentially identical to the proof of part (2) of Theorem~\ref{T:main} in Section~\ref{ss:proofofmain}. The only changes are to the referenced theorems: one must replace Lemma~\ref{L:celemma} with Lemma~\ref{L: cgammae}, Proposition~\ref{P:fixedinduction} with Proposition~\ref{P:Exc fixedinduction}, Theorem~\ref{T:abelianquotients} with the first part of this result, and references to \eqref{e:spanningzgzero} with Corollary~\ref{Cor: even data}(2). Furthermore, since $(\g_0,0)$ is already a rigid induction datum, one does not need an argument as in the final paragraph of Section~\ref{ss:proofofmain}.
\end{proof}

\end{document}